\documentclass[11pt]{amsart}

\usepackage{graphicx,amssymb,amsmath,xcolor,enumerate}
\usepackage{hyperref}
\hypersetup{
	bookmarks=true,         
	pdffitwindow=false,     
	pdfstartview={FitH},    
	colorlinks=true,      
	citecolor=red,
}
\usepackage{dsfont}
\usepackage{bbm}
\usepackage{enumitem}

\usepackage{enumitem}

\usepackage{float}

\makeatletter
\def\l@subsection{\@tocline{2}{0pt}{2.5pc}{5pc}{}}
\makeatother

\usepackage{geometry}
\DeclareSymbolFont{largesymbol}{OMX}{yhex}{m}{n}
\DeclareMathAccent{\Widehat}{\mathord}{largesymbol}{"62}
\newcommand*\di{\mathop{}\!\mathrm{d}}

\def\e{\varepsilon}

\numberwithin{equation}{section}              
\newtheorem{theorem}{Theorem}[section]

\newtheorem{lemma}{Lemma}[section]
\newtheorem{proposition}{Proposition}[section]
\newtheorem*{proposition*}{Proposition}

\newtheorem*{corollary*}{Corollary}

\newtheorem{definition}{Definition}[section]
\newtheorem*{definitions*}{Definitions}

\newtheorem*{acknowledgements*}{Acknowledgements}

\newtheorem*{conjecture*}{\bf Conjecture}
\newtheorem{example}{\bf Example}[section]
\newtheorem*{example*}{\bf Example}
\theoremstyle{remark}
\newtheorem{remark}{\bf Remark}[section]

\author{Yu Gao}
\address[Y. Gao]{School of Science, Harbin Institute of Technology (Shenzhen), Shenzhen, 518055 China}
\email{gaoyu2024@hit.edu.cn}

\author{Hao Liu}
\address[H. Liu]{Department of Mathematics, Faculty of Science, University of Macau, Taipa, Macau}
\email{haoliu@um.edu.mo}



\begin{document}

\date{}                                     
\title[2HS system]{Existence, uniqueness and long-time behavior of  the $\lambda$-dissipative solutions to the two-component Hunter-Saxton system}

\maketitle

\begin{abstract}
In this paper, we construct the explicit characteristics for the $\lambda$-dissipative solutions ($\lambda\in[0,1]$) to the two-component Hunter--Saxton (2HS) system. Using these characteristics, we provide a comprehensive study of the existence, uniqueness, and asymptotic behavior of these solutions. 

For the fully dissipative case ($\lambda=1$), uniqueness follows from the absence of outgoing cusps. In the partially dissipative regime ($0<\lambda<1$), outgoing cusps are present. We formulate an exact Eulerian dissipation rule which identifies the energy that has already passed through wave breaking by the intrinsic condition $\rho=0$ and $u_x\geq 2/t$. This rule determines the dissipated part of the energy measure and yields uniqueness. This uniqueness applies to the classical Hunter-Saxton equation and seems to be the first uniqueness result for the general $\lambda$-dissipative solutions.

Concerning the large-time dynamics, we show that the density $\rho$ and the singular part of the energy measure decay to zero as $t\to\infty$, indicating that all energy is eventually concentrated in the $u$-component. Moreover, we derive the leading-order asymptotic term, which takes the form of a kink-wave determined by the system's remaining energy. This kink-wave can be explicitly computed from the initial data and the dissipation parameter $\lambda$.

\end{abstract}

{
\hypersetup{linkcolor=blue}
\tableofcontents
}

\section{Introduction}\label{sec:intro}

\subsection{Background}
This paper focuses on the regularity structure, existence, uniqueness, and the large time asymptotic behavior of the $\lambda$-dissipative ($\lambda  \in [0,1]$) solutions to the two-component Hunter-Saxton system (2HS) on the whole real line $\mathbb{R}$:  for $x\in\mathbb{R}$, $t>0$,
\begin{equation}\label{eq:2HS}
\left\{
\begin{aligned}
&u_t+uu_x=\frac{1}{2}D_{\kappa}^{-1}(u_x^2+\rho^2),\\
&\rho_t+(\rho u)_x=0,
\end{aligned}
\right.
\end{equation}
Here $D_\kappa^{-1}$ stands for the anti-derivative operator, which is defined by
\[
D_\kappa^{-1} (u_x^2+\rho^2)=\left[(1-\kappa)\int_{-\infty}^x-\kappa\int_{x}^{+\infty}\right](u_x^2+\rho^2)(y,t)\di y
\]
for some real number $\kappa\in\mathbb{R}$.  
The number $\lambda  \in [0,1]$ is a parameter that measures the dissipation rate of the energy at the blow-up point (see Definition \ref{def:weak} for $\lambda$-dissipative solutions below).  The 2HS system \eqref{eq:2HS} was first discovered in \cite[Eq. (47)]{Olver} as an example derived by the tri-Hamiltonian duality approach from the It\^{o} system of coupled nonlinear wave equations. It was a particular case of the Gurevich-Zybin system describing the dynamics in a model of non-dissipative dark matter; see \cite{pavlov2005gurevich} and references therein. From a geometric point of view,  the 2HS system can be interpreted as the Euler equation on the superconformal algebra of contact vector fields on the $1|2$-dimensional supercircle \cite{lenells2009n}. Moreover, It is an integrable system, which has a bi-Hamiltonian structure and a Lax-pair \cite{Olver,constantin2008integrable,lenells2009n}. The bi-Hamiltonian structure of the  2HS system \eqref{eq:2HS} is given by:
\begin{align*}
\frac{\partial}{\partial t}
\begin{pmatrix}
u\\
\rho
\end{pmatrix}
=
\begin{pmatrix}
u_xD_{\kappa}^{-2}-D_{\kappa}^{-2}u_x & D_{\kappa}^{-2}\rho\partial_x\\
\partial_x\rho D_{\kappa}^{-2} &  0
\end{pmatrix}
\begin{pmatrix}
{\delta \mathcal{H}_1}/{\delta u}\\
{\delta \mathcal{H}_1}/{\delta \rho}
\end{pmatrix},
\end{align*}
and
\begin{equation*}
\frac{\partial}{\partial t}
\begin{pmatrix}
u\\
\rho
\end{pmatrix}
=
\begin{pmatrix}
D_{\kappa}^{-1} & 0\\
0 & -\partial_x
\end{pmatrix}
\begin{pmatrix}
\delta \mathcal{H}_2/\delta u \\ \delta \mathcal{H}_2/\delta \rho
\end{pmatrix},
\end{equation*}
where the two Hamiltonians $\mathcal{H}_1$ and $\mathcal{H}_2$ are 
\begin{align}\label{eq:energy}
\mathcal{H}_1(u,\rho)=\frac{1}{2}\int_{\mathbb{R}}(u_x^2+\rho^2)\di x,\quad \mathcal{H}_2(u,\rho)=\frac{1}{2}\int_{\mathbb{R}}u(u_x^2+\rho^2)\di x.
\end{align}
Here $\mathcal{H}_1$ is referred to as the total energy. 
The 2HS system is also the short-wave (or high-frequency) limit of the following two-component Camassa-Holm (2CH) system \cite{Olver,constantin2008integrable,chen2006two}:
\begin{equation}\label{eq:2CH}
\left\{
\begin{aligned}
&m_t+2mu_x+um_x- \rho\rho_x=0,\quad m=u-u_{xx},\\
&\rho_t+(\rho u)_x=0.
\end{aligned}
\right.
\end{equation}
The 2CH system first appeared in \cite[Eq. (43)]{Olver}. It was also derived from the Green-Naghdi equations describing water waves; see \cite{constantin2008integrable}. When the sign of the term $\rho\rho_x$ is changed, the system corresponds to the situation that the gravity acceleration points upwards  \cite{constantin2008integrable}, and it was first studied by Chen et al. in \cite{chen2006two}, where a reciprocal transformation between the 2CH system and the first negative flow of the AKNS hierarchy was established. Aratyn et al. showed that the modification of the Schr\"odinger spectral problem for the 2CH system leads to the 2HS system \eqref{eq:2HS}; see \cite{aratyn2006negative}.

When $\rho\equiv0$, the 2HS system \eqref{eq:2HS} becomes the Hunter-Saxton (HS) equation:
\begin{equation}\label{eq:HS}
u_t+uu_x=\frac{1}{2} D_{\kappa}^{-1} u_x^2.
\end{equation}
It was first proposed by Hunter and Saxton \cite{hunter1991dynamics} to study a nonlinear instability in the director field of a nematic liquid crystal, where $u(x,t)$ is related to the derivation of the average orientation of the molecules from an equilibrium position. Note that if we define $m=u_{xx}$, then the HS equation implies
\[
m_t+2mu_x+um_x=0,\quad m=u_{xx},
\]
which shares the same form with the Camassa-Holm (CH) equation (i.e., $\rho\equiv0$ in system \eqref{eq:2CH}):
\begin{align}\label{eq:CH}
m_t+2mu_x+um_x=0,\quad m=u-u_{xx}.
\end{align}
Hunter and Zheng proved that the Hunter-Saxton equation \eqref{eq:HS} is a completely integrable, bi-variational, bi-Hamiltonian system, and the corresponding equation for $m=u_{xx}$ belongs to the Harry Dym hierarchy; see \cite{hunter1994completely} for more details. 

If $(u,\rho)$ is a classical solution to \eqref{eq:2HS}, differentiating \eqref{eq:2HS}$_1$ with respect to the spatial variable $x$ yields
\begin{align*}
(u_{t}+uu_{x})_x=\frac{1}{2}(u_x^2+\rho^2).
\end{align*}
Combining the above equation with \eqref{eq:2HS}$_2$, we obtain the following energy conservation law:
\begin{align}\label{eq:conservationlaw}
(u_x^2+\rho^2)_t+[u(u_x^2+\rho^2)]_x=0.
\end{align}
According to the energy conservation equation \eqref{eq:conservationlaw}, it is nature to find the global-in-time conservative solution $(u,\rho)$ that satisfies $\|u_x(\cdot,t)\|_{L^2}^2+\|\rho(\cdot,t)\|_{L^2}^2=\|\bar{u}_x\|_{L^2}^2+\|\bar{\rho}\|_{L^2}^2$ for any $t>0$, provided that the initial datum $(u,\rho)|_{t=0}=(\bar{u},\bar{\rho})$  satisfies $\bar{u}_x,~\bar{\rho}\in L^2(\mathbb{R})$. However, the solution might blow up in finite time, and the solution $(u,\rho)$ will lose the information of energy conservation temporarily. To see this, consider the characteristics defined by  
$$
\frac{\partial}{\partial t}X(\xi,t)=u(X(\xi,t),t),
\quad X(\xi,0)=\xi\in\mathbb{R}.
$$
From \eqref{eq:2HS} and \eqref{eq:conservationlaw}, one can formally obtain the formula for initial data $(\bar{u},\bar{\rho})$
\begin{align}\label{eq:chara}
X(\xi,t)=\xi+\bar{u}(\xi)t+\frac{t^2}{4}D_\kappa^{-1} (\bar{u}_x^2+\bar{\rho}^2).
\end{align}
From the above characteristics, one can verify the following fact: if the initial datum $(\bar{u},\bar{\rho})$ satisfies $\bar{u}_x(x_0)=\inf_{x\in\mathbb{R}}\bar{u}_x(x)<0$ and $\bar{\rho}(x_0)=0$ for some $x_0\in\mathbb{R}$, then we have $u_x(\cdot,t)\to-\infty$ as $t\to -\frac{2}{\bar{u}_x(x_0)}$. 
By the characteristics, we can also construct the following example ($\kappa=0$):
\begin{example}\label{example}
Consider the solution with initial datum 
\begin{equation*}
\bar{u}(x)=
\begin{cases}
0, & x\leq0,\\[6pt]
-x, & 0<x\leq 1,\\[6pt]
-\dfrac{x+1}{2}, & 1<x\leq 3,\\[6pt]
-2, & x>3,
\end{cases}
\qquad 
\bar{\rho}(x)=
\begin{cases}
\frac{1}{2}, & x\in[1,3],\\
0, & \textrm{otherwise}.
\end{cases}
\end{equation*}
By the characteristic method ($\kappa=0$) for $t<2$, the solution $(u,\rho)$ is 
\begin{equation*}
u(x,t)=
\begin{cases}
0, & x\leq0,\\[8pt]
-\dfrac{2x}{2-t}, & 0<x\leq \left(1-\frac{t}{2}\right)^2,\\[10pt]
-\left(1-\frac{t}{2}\right)\dfrac{x+1}{1+(1-\frac{t}{2})^2}, & \left(1-\frac{t}{2}\right)^2\leq x\leq 2\left(1-\frac{t}{2}\right)^2+1,\\[10pt]
-2+t, & x\geq  2\left(1-\frac{t}{2}\right)^2+1,
\end{cases}
\end{equation*}
and
\begin{equation*}
\quad
\rho(x,t)=
\begin{cases}
\dfrac{1}{1+\left(1-\frac{t}{2}\right)^2}, & \left(1-\frac{t}{2}\right)^2\leq x\leq 2\left(1-\frac{t}{2}\right)^2+1,\\[12pt]
0,  & \textrm{otherwise.}
\end{cases}
\end{equation*}
Moreover, we have
\[
u_x(x,t)=
\begin{cases}
0, &  x\le 0,\\
-\dfrac{1}{1-\frac{t}{2}}, 
&  0<x<\left(1-\frac{t}{2}\right)^2,\\[12pt]
-\dfrac{1-\frac{t}{2}}{1+\left(1-\frac{t}{2}\right)^2}, 
& \left(1-\frac{t}{2}\right)^2<x<1+2\left(1-\frac{t}{2}\right)^2,\\[12pt]
0, & x\ge 1+2\left(1-\frac{t}{2}\right)^2.
\end{cases}
\]
As \(t\to 2^-\), the interval \(0<x<\left(1-\frac{t}{2}\right)^2\) collapses to \(x=0\), and \(u_x(x,t)\to -\infty\), so an incoming cusp forms at \((x,t)=(0,2)\). See Example \ref{example2} for $t>2$.
\end{example}
From this example, we have $u(x,2)\equiv 0$ and $u_x^2(\cdot,t)\to \delta$ in the sense of distributions as $t\to 2$, where $\delta$ is the Dirac delta mass at the origin $x=0$. Notice that $\|u_x(\cdot,t)\|^2_{L^2}+\|\rho(\cdot,t)\|_{L^2}^2= 2$ for any $t\neq 2$ and $\|u_x(\cdot,2)\|_{L^2}=0$, $\|\rho(\cdot,2)\|_{L^2}=1$. Therefore, if we only study the solution $(u,\rho)$, the energy $\|u_x(\cdot,t)\|^2_{L^2}+\|\rho(\cdot,t)\|_{L^2}^2$ is not conserved at $t=2$. It is worth noting that at the blow-up time $t=2$, part of the energy density is converted into the singular measure $\delta$. This phenomenon shares some similarities with the HS equation and the CH equation; see \cite{bressan2007global,holden2007global,gao2022regularity} for instance. 
After the blow-up time, one way to extend the solution globally leads to the global dissipative solutions, where part of the energy $\mathcal{H}_1$ given by \eqref{eq:energy} will dissipates. Another way gives the  conservative solutions, which makes the energy conserved for all time.  

The global strong solutions to the 2HS system \eqref{eq:2HS} in the periodic setting were obtained  by Wunsch in \cite{wunsch2009hunter} for initial data $(\bar{u},\bar{\rho})\in H^s(\mathbb{S})\times H^{s-1}(\mathbb{S})$ ($s>3$) with some smallness conditions. Later, Wunsch \cite{wunsch2010generalized} showed the existence of global dissipative solutions on the real line by the characteristics method used in \cite{bressan2005global} for the HS equation  \eqref{eq:HS}. Explicit solutions on the unit circle were obtained in \cite{wunsch2011weak,lenells2013spheres} by using the geometric interpretation of system \eqref{eq:2HS}. The existence of conservative solutions to 2HS system with $\kappa=1/2$ on the real line was proved in \cite{nordli2016lipschitz}, where a Lipschitz metric was constructed for the stability of the obtained solutions. Recently, Grunert and Nordli \cite{grunert2018existence} proved the Lipschitz stability for the $\lambda$-dissipative solutions to the 2HS system with $\kappa=1/2$ by a generalized method of characteristics. 
The ideas for the generalized method of characteristics come from a series of papers for the CH equation \eqref{eq:CH} (see \cite{bressan2005optimal,bressan2007global,holden2007global,holden2008global,holden2009dissipative}), and the HS equation \eqref{eq:HS} (see \cite{bressan2005global,bressan2007asymptotic,Bressan2010}), which were also applied to the 2CH system \eqref{eq:2CH}; see \cite{grunert2012global,grunert2015continuous}.We remark here that if we change the sign of the term $\rho\rho_x$ in the 2HS system, some existence results can also be found; see \cite{wunsch2009hunter,wunsch2010generalized,lenells2013hunter,guan2011global}.

The existence of both conservative and dissipative solutions to the 2HS system \eqref{eq:HS} has been established in the literature. However, the uniqueness and asymptotic behavior of weak solutions have remained open questions. These properties are precisely among the results obtained in the present paper.

\subsection{Methodology and main results} 

In this paper, we are going to study the regularity structure, existence, uniqueness, and the asymptotic behavior of the $\lambda$-dissipative solutions to the 2HS system \eqref{eq:2HS} by an explicit characteristics method. The following  generalized framework was studied (see \cite[Eqs. (1.4)-(1.6)]{gao2022regularity} for HS equation with $\kappa=0$):
\begin{align}
&u_t+uu_x=\frac{1}{2}\left[(1-\kappa)\int_{-\infty}^x-\kappa\int_{x}^{+\infty}\right]\di \mu(t),\label{eq:gHS1}\\
&\rho_t+(\rho u)_x=0,\label{eq:gHS4}\\
&\mu_t+(u\mu)_x=0,\label{eq:gHS2}\\
&\di \mu_{ac}(t)=(u_x^2+\rho^2)(x,t)\di x.\label{eq:gHS3}
\end{align}
Here, the energy measure $\mu(t)$ is a nonnegative Radon measure, and $\mu_{ac}(t)$ is the absolutely continuous part of $\mu(t)$ with respect to the Lebesgue measure $\di x$. Equation \eqref{eq:gHS2} is  Equation \eqref{eq:conservationlaw} with $(u_x^2+\rho^2)\di x$ replaced by $\di\mu$. Equation~\eqref{eq:gHS3} just says that the absolutely continuous part of $\mu$ corresponds to $u_x^2+\rho^2$ exactly, which is a compatibility condition between $\mu$ and $u_x^2+\rho^2$. 
In this way, the weak solution we found is the so-called conservative solutions.
Except for conservative solutions, the 2HS system also has a class of weak solutions that dissipates the energy.
It is also reasonable to wonder whether there are intermediate solutions between conservative and dissipative solutions that can dissipate the energy partially. This leads to the concept of $\lambda$-dissipative solutions, which was first introduced in \cite{grunert2015continuous}  in the context of the 2CH system, and then studied for HS equation \cite{grunert2022lipschitz} and 2HS system with $\kappa=1/2$ \cite{grunert2018existence}.

We first give a clear description of the  $\lambda$-dissipative solutions.
Since for the (fully) dissipative solutions, i.e., $\lambda=1$, all the singular parts of the energy measure are dissipated, we only need to study the absolutely continuous part $(u,\rho)$ of the energy measure. Hence we define the solution spaces $\mathcal{D}_{\lambda}$  separately for $0\leq\lambda<1$ (see also \cite[Definition 2.2, 2.3]{grunert2022lipschitz})
 and $\lambda=1$:
\begin{definition}[Solution spaces]\label{def:D}
The solutions are studied in the following spaces:
\begin{enumerate}
\item For $0\leq\lambda<1$, let $\mathcal{D}_{\lambda}$ be the set of all quadruples $(u,\rho, \mu,\nu)$ satisfying
\begin{enumerate}
\item[(i)] $u\in C_b(\mathbb{R})$, $u_x\in L^2(\mathbb{R}),~~\rho\in L^2(\mathbb{R})$;
\item[(ii)] $\mu,~\nu\in \mathcal{M}_+(\mathbb{R})$, $\mu\ll\nu$;
\item[(iii)] $\di \mu_{ac}=(u_x^2+\rho^2)\di x$,  where  $\mu_{ac}$ and $\mu_s$ are the absolutely continuous  part and singular part of measure $\mu$ with respect to the Lebesgue measure $\di x $;
\item[(iv)]  the Radon-Nikodym derivative $\frac{\di\mu}{\di\nu}(x)\in\{1,~1-\lambda\}$, and $\frac{\di\mu}{\di\nu}(x)=1$ for a.e. $x\in\mathbb{R}$ satisfying $u_x(x)<0$ or $\rho(x)\neq0$.
\end{enumerate}
Here, $C_b(\mathbb{R})$ is the space of continuous and bounded functions defined on $\mathbb{R}$ equipped with the sup-norm, and $\mathcal{M}_+(\mathbb{R})$ stands for the set of all finite nonnegative Radon measures endowed with the weak topology. 
\item For $\lambda=1$, let
\[
\mathcal{D}_1=\{(u,\rho):~~ u\in C_b(\mathbb{R}), ~~u_x,~\rho\in L^2(\mathbb{R})\}.
\]
\end{enumerate}

\end{definition}
In Definition \ref{def:D} for $0\leq\lambda<1$, the measure $\mu$ is used to describe the remaining energy of the $\lambda$-dissipative solution  in the evolution process of the system, and $\nu$ is the total energy measure associated with the transport induced by the velocity field $u$.  In (iv), for a.e. $x\in\mathbb{R}$, if $u_x(x)<0$ or $\rho(x)\neq0$, then at the position $x$, there has not been a blow-up yet, which means the energy at $x$ is not losing  and hence $\frac{\di\mu}{\di\nu}(x)=1$; see (vii) in Theorem \ref{thm:measure} for more details. To take the dissipation on $t=0$ into consideration, we impose one more condition for the initial data $(\bar{u},\bar{\rho},\bar{\mu},\bar{\nu})\in \mathcal{D}_{\lambda}$
\begin{equation}\label{eq:initialdecay}
\bar{\mu}=\bar{\nu}_{ac}+(1-\lambda)\bar{\nu}_s.
\end{equation}
In this paper, the dissipation before $t=0$ is ignored. If one wants to take the dissipation before  $t=0$ into consideration, one needs to add more precise descriptions for initial data, which is left for interested readers.

Next, we give the definition of the $\lambda$-dissipative solutions for $0\leq\lambda<1$.
\begin{definition}[$\lambda$-dissipative solutions]\label{def:weak}
For $0\leq\lambda<1$, let $(\bar{u},\bar{\rho},\bar{\mu},\bar{\nu})\in \mathcal{D}_{\lambda}$ satisfy \eqref{eq:initialdecay}. The quadruple
$(u(t),\rho(t),\mu(t),\nu(t))\in\mathcal{D}_{\lambda}$ is said to be a global-in-time $\lambda$-dissipative solution to the 2HS system with initial data $(\bar{u},\bar{\rho},\bar{\mu},\bar{\nu})$ if it satisfies the following conditions.
\begin{enumerate}
\item[(i)] We have
\[
\nu\in C([0,+\infty);\mathcal{M}_+(\mathbb{R})),\qquad
\nu(t)(\mathbb{R})=\bar\nu(\mathbb{R}),
\]
$u_x(\cdot,t)\in L^2(\mathbb{R})$ for every $t\in[0,\infty)$,
\[
u\in C([0,\infty);C_b(\mathbb{R}))\cap C^{1/2}_{loc}(\mathbb{R}\times[0,\infty)),
\qquad u_t\in L_{loc}^2(\mathbb{R}\times(0,\infty)),
\]
and
\[
\rho\in L^\infty(0,\infty;L^2(\mathbb{R}))\cap C_w([0,\infty);L^2(\mathbb{R})).
\]

\item[(ii)] The initial conditions are
\[
(u(\cdot,0),\rho(\cdot,0),\mu(0),\nu(0))
=(\bar{u},\bar{\rho},\bar{\mu},\bar{\nu}).
\]
Moreover,
\[
T_s:=\{t\geq0:\nu_s(t)\neq0\}
\]
is at most countable, and
\[
\di\mu(t)=(u_x^2+\rho^2)(x,t)\di x,
\qquad t\in[0,\infty)\setminus T_s.
\]

\item[(iii)] The equations
\begin{align}\label{eq:weakformula}
\int_{0}^\infty\int_{\mathbb{R}}u\phi_t-\phi\left(uu_x-F\right)\di x\di t
=-\int_{\mathbb{R}}\bar{u}\phi(x,0)\di x
\end{align}
and
\begin{align}\label{eq:weakformula1}
\int_{0}^\infty\int_{\mathbb{R}}\rho(\phi_t+u\phi_x)\di x\di t
=-\int_{\mathbb{R}}\bar{\rho}\phi(x,0)\di x
\end{align}
hold for every $\phi\in C_c^1(\mathbb{R}\times[0,\infty))$, where
\[
F(x,t):=\frac12\left[(1-\kappa)\int_{-\infty}^x-\kappa\int_x^{+\infty}\right]\di\mu(t).
\]

\item[(iv)] The total energy measure is transported by $u$, namely
\begin{equation}\label{eq:energyconservenu}
\int_{0}^\infty\int_{\mathbb{R}}(\phi_t+u\phi_x)\di\nu(t)\di t
+\int_{\mathbb{R}}\phi(x,0)\di\bar\nu=0
\end{equation}
for every $\phi\in C_c^1(\mathbb{R}\times[0,\infty))$. There exists a nonnegative Radon measure
\[
\mathfrak D\in\mathcal M_+(\mathbb{R}\times(0,\infty))
\]
such that
\begin{equation}\label{eq:energydissipation}
\int_{0}^\infty\int_{\mathbb{R}}(\phi_t+u\phi_x)\di\mu(t)\di t
+\int_{\mathbb{R}}\phi(x,0)\di\bar\mu
=\int_{\mathbb{R}\times(0,\infty)}\phi(x,t)\di\mathfrak D(x,t)
\end{equation}
for every $\phi\in C_c^1(\mathbb{R}\times[0,\infty))$. The measure $\mathfrak D=-[\mu_t+(u\mu)_x]$ is determined by the quadruple and is not additional initial data. In particular, for every nonnegative $\phi\in C_c^1(\mathbb{R}\times[0,\infty))$,
\begin{equation}\label{eq:fourth}
\int_{0}^\infty\int_{\mathbb{R}}(\phi_t+u\phi_x)\di\mu(t)\di t
+\int_{\mathbb{R}}\phi(x,0)\di\bar\mu\geq0.
\end{equation}
For $\lambda=0$, one has $\mathfrak D=0$.

\item[(v)] Equation \eqref{eq:gHS3} holds for every $t\geq0$. In addition, the following exact Eulerian dissipation rule holds for every $t>0$:
\[
\di(\nu-\mu)(t)
=\frac{\lambda}{1-\lambda}u_x^2(x,t)
\mathbf 1_{\{\rho(x,t)=0,\;u_x(x,t)\geq 2/t\}}\di x
+\lambda\di\nu_s(t).
\]
Thus the energy is unchanged on
$\{\rho\neq0\}\cup\{u_x<2/t\}$, while precisely the fraction $\lambda$ is removed from the energy that has already passed through wave breaking. For $\lambda=0$, the preceding identity is understood as $\mu(t)=\nu(t)$.

\item[(vi)] The map $t\mapsto\mu(t)$ is right-continuous in the weak-star topology and admits a weak-star left limit at every $t>0$. Moreover,
\begin{equation}\label{eq:decay}
\mu(t+)=\mu(t)=\mu_{ac}(t-)+(1-\lambda)\mu_s(t-)
=\mu_{ac}(t)+(1-\lambda)\nu_s(t),\qquad t>0.
\end{equation}
More precisely, $\mu(s)\overset{\ast}{\rightharpoonup}\mu(t)$ as $s\to t+$ and
$\mu(s)\overset{\ast}{\rightharpoonup}\mu(t-)$ as $s\to t-$.
Here $\mu_s(t)$ and $\nu_s(t)$ denote the singular parts of $\mu(t)$ and $\nu(t)$ with respect to the Lebesgue measure.
\end{enumerate}
\end{definition}

For $\lambda=1$, i.e., the  dissipative case, 
there is no need to study the energy measures $\nu$ and $\mu$. Hence one can use a simpler definition. Consider the initial data $(\bar{u},\bar{\rho})\in\mathcal{D}_1$.
The dissipative solutions are defined by
\begin{definition}[Dissipative solutions]\label{def:dissipative}
A binary pair $(u,\rho)$ is a weak solution of the 2HS system \eqref{eq:2HS} ($\kappa=0$) with initial data $(\bar{u},\bar{\rho})\in\mathcal{D}_1$ if $u(\cdot,t)$ is absolutely continuous on $\mathbb R$ for every $t\geq0$,
\[
u_x,\rho\in L^\infty(0,\infty;L^2(\mathbb R)),
\]
and \eqref{eq:2HS} holds on $\mathbb R\times(0,\infty)$ in the sense of distributions, with
$u(x,0)=\bar u(x)$ and $\rho(x,0)=\bar\rho(x)$ for a.e. $x\in\mathbb R$.
A weak solution $(u,\rho)$ is called dissipative if its derivative $\partial_xu$ is bounded from above on any compact subset of the upper half-plane $\mathbb{R}\times[0,\infty)$, and
\[
\partial_xu(\cdot,t)\to \bar{u}',~~\rho(\cdot,t)\to\bar{\rho},\quad \text{strongly in }~L^2(\mathbb{R}),\quad \text{ as }t\to0.
\]
\end{definition}
\begin{remark}\label{rmk:defweak}
\begin{enumerate}
\item The identities \eqref{eq:energyconservenu} and \eqref{eq:energydissipation} are equivalent to the distributional balance laws
\[
\nu_t+(u\nu)_x=0,
\qquad
\mu_t+(u\mu)_x=-\mathfrak D.
\]
The measure $\mathfrak D$ may have an absolutely continuous, a singular continuous, and an atomic component in time. The exact Eulerian rule in Definition~\ref{def:weak}(v), rather than an arbitrary choice of $\mathfrak D$, specifies which part of the transported energy has been dissipated.
\item The boundedness property for dissipative solution in Definition \ref{def:dissipative} allows for incoming cusps but rules out outgoing cusps. In contrast, for $0<\lambda<1$, the $\lambda$-dissipative solutions allow for both incoming and outgoing cusps, as illustrated in Examples \ref{example} and \ref{example2}. This causes the main difficulty for uniqueness for $0<\lambda<1$. 
\end{enumerate}
\end{remark}

For any initial data in $\mathcal{D}_{\lambda}$, we will use the characteristic methods to construct $\lambda$-dissipative solutions in the sense of Definition \ref{def:weak}. The regularity structure of the solutions will also be discussed. We have the following theorem for existence and regularity structure:
\begin{theorem}\label{thm:introduce}
We have the following results for existence and regularity structure:
\begin{enumerate}
\item For $0\leq\lambda<1$, let $(\bar{u},\bar{\rho},\bar{\mu},\bar{\nu})\in\mathcal{D}_{\lambda}$ satisfy \eqref{eq:initialdecay}. Then there exists a global-in-time $\lambda$-dissipative solution  $(u(t),\rho(t),\mu(t),\nu(t))\in\mathcal{D}_{\lambda}$ to the 2HS system  with  initial date $(\bar{u},\bar{\rho},\bar{\mu},\bar{\nu})$ in the sense of Definition \ref{def:weak}.  

Let $\mu(t)=\mu_{ac}(t)+\mu_{pp}(t)+\mu_{sc}(t)$, where $\mu_{ac}(t)$, $\mu_{pp}(t)$, and $\mu_{sc}(t)$ are the absolutely continuous part, pure point part and the singular continuous part of $\mu$ respectively.
The following statements hold:
\begin{enumerate}
\item [(i)] (Energy conservation) we have $\nu \in C([0,\infty); \mathcal{M}_+(\mathbb{R}))$ and
\begin{align}\label{eq:conserE}
\nu(t)(\mathbb{R})= \bar{\nu}(\mathbb{R}),\quad t>0. 
\end{align}
\item [(ii)] (Formation of singularities) For any $t\neq0$, $\mu_{pp}(t)$ and $\mu_{sc}(t)$ are determined by the absolutely continuous part $\bar{\nu}_{ac}$, namely determined by $\bar{u}_x$ and $\bar{\rho}$. More precisely, for any $t> 0$, we define
\begin{align}\label{eq:singulartE}
A_t^E =\left\{x:~~\bar{u}_x(x)=-\frac{2}{t},~~\bar{\rho}(x)=0\right\}.
\end{align}
If $\mathcal{L}(A_t^E)\neq 0$, then $\mu_{pp}(t)+\mu_{sc}(t)\neq 0$  (i.e., $\mu_{pp}(t)+\mu_{sc}(t)$ is not a zero measure).
All the intervals with positive length  in $A_t^E$ will generate the pure point part $\mu_{pp}(t)$, and the rest of $A_t^E$ will generate the singular continuous part $\mu_{sc}(t)$.

\item [(iii)] (Countably many singular times) There are at most countably many times $t\in[0,\infty)$, such that either the pure point part or the singular continuous part of $\mu(t)$ or $\nu(t)$ is not zero.

\item [(iv)] (Regularity) For every $t\in[0,\infty)$, the function $u(\cdot,t)$ is globally absolutely continuous, $\rho(\cdot,t)\in L^2(\mathbb R)$, and
\begin{align}\label{eq:acpartE}
\di\mu_{ac}(t)=(u_x^2+\rho^2)(x,t)\di x.
\end{align}
Furthermore,
\begin{align}\label{eq:propertiesuE}
u\in C([0,\infty);C_b(\mathbb{R}))\cap C^{1/2}_{loc}(\mathbb{R}\times [0,\infty)),~~ u_x\in L^\infty(0,\infty;L^2(\mathbb{R})),~~ u_t\in L_{loc}^2(\mathbb{R}\times [0,\infty))
\end{align}
and
\[
\rho\in L^\infty(0,\infty;L^2(\mathbb{R}))\cap C_w([0,\infty);L^2(\mathbb{R})),\quad 	\|\rho(\cdot,t)\|_{L^2}^2\leq\bar{\mu}(\mathbb{R}).
\]

\item [(v)] The exact Eulerian dissipation rule in Definition~\ref{def:weak}(v) holds. In particular,
$\frac{\di\mu(t)}{\di\nu(t)}(x)\in\{1,1-\lambda\}$, and
$\frac{\di\mu(t)}{\di\nu(t)}(x)=1$ for a.e. $x\in\mathbb R$ satisfying $u_x(x,t)<0$ or $\rho(x,t)\neq0$.

\end{enumerate}

\item
For $\lambda=1$, let initial data $(\bar{u},\bar{\rho})\in\mathcal{D}_1$,
then there exists a dissipative solution $(u,\rho)$ to the 2HS system in the sense of Definition \ref{def:dissipative}, and for $(u,\rho)$ similar regularity results hold.
\end{enumerate}
\end{theorem}

In this paper, we employ the method of characteristics to establish uniqueness of $\lambda$-dissipative solutions to the 2HS system for any $\lambda\in[0,1]$. Our main goal is to show that, for any fixed initial datum in $\mathcal{D}_{\lambda}$ satisfying \eqref{eq:initialdecay}, all $\lambda$-dissipative solutions in the sense of Definition~\ref{def:weak} (or Definition~\ref{def:dissipative} when $\lambda=1$) coincide with the one constructed via the explicit characteristic equations. Uniqueness then follows immediately.

This characteristic-based approach has been successfully applied in related settings: Dafermos~\cite{dafermos2011generalized} used it to establish uniqueness of dissipative solutions to the HS equation; Bressan, Chen, and Zhang~\cite{bressan2015unique} proved uniqueness of conservative solutions to the Camassa--Holm equation; and the present authors employed the same method in~\cite{gao2022regularity} to obtain uniqueness of conservative solutions to the HS equation. Further applications to other models can be found in, e.g.,~\cite{bressan2017lipschitz,bressan2016uniqueness,chen2018existence}.

The main difficulties in the present work arise from the coupling between $u$ and $\rho$, as well as the presence of outgoing cusps when $0<\lambda<1$. The proof differs according to the three parameter regimes: $\lambda=0$, $0<\lambda<1$, and $\lambda=1$.

For $\lambda=0$, the method of \cite{gao2022regularity} can be applied directly; see Theorem~\ref{thm:uniqueness1}.

For $\lambda=1$, the proof closely follows the approach of \cite{dafermos2011generalized}. The main difference---and the principal difficulty---lies in analyzing the singularities of $u_x$ along characteristics. Unlike the Riccati equation for the HS equation (cf.~\cite[Eq.~(4.21)]{dafermos2011generalized}), here we must derive an ODE system governing the coupled evolution of $u_x$ and $\rho$ along characteristics; this is carried out in Lemma~\ref{lmm:ODEs}. The coupling between the two quantities is resolved by reducing the system to a single complex-valued Riccati equation, whose explicit solution formula is presented in Proposition~\ref{eq:Jt_tildeAt}.

For $0<\lambda<1$, $\lambda$-dissipative solutions admit both incoming and outgoing cusps. The decisive point is therefore to distinguish, in Eulerian variables, an outgoing branch created by wave breaking from a regular expanding branch. The intrinsic phase condition
\[
\rho(x,t)=0,
\qquad
u_x(x,t)\geq\frac2t
\]
identifies precisely the outgoing part. The exact Eulerian dissipation rule in Definition~\ref{def:weak}(v) then determines $\nu-\mu$. Proposition~\ref{pro:B} converts this rule into the Lagrangian energy representation, and Theorem~\ref{thm:uniqueness3} yields uniqueness.

Consequently, the uniqueness proof depends on the value of $\lambda$. We summarize the results as follows.
\begin{theorem}\label{thm:uniqueness}
The following uniqueness results hold.
\begin{enumerate}
\item For $0\leq\lambda<1$, let
$(\bar u,\bar\rho,\bar\mu,\bar\nu)\in\mathcal D_\lambda$
satisfy \eqref{eq:initialdecay}. Then the global-in-time
$\lambda$-dissipative solution in the sense of Definition~\ref{def:weak} is unique and is given by \eqref{eq:lambda_dissipation}--\eqref{eq:solution2}.
\item For $\lambda=1$ and initial data
$(\bar u,\bar\rho)\in\mathcal D_1$, the dissipative solution in the sense of Definition~\ref{def:dissipative} is unique.
\end{enumerate}
\end{theorem}

To study the asymptotic behavior of the $\lambda$-dissipative solutions, the main difficulty is that we have to keep track of all the possible blow-ups and the concentrated energies. 
By incorporating the scaling properties of the equation into the constructed explicit characteristics, we find out that the rescaled limit is determined by the remaining energy after accounting for all possible blow-ups in the system which can be calculated through the initial data and the dissipation parameter $\lambda$. 
Then, the large-time asymptotic expansions in spaces $L^{\infty}(\mathbb{R})$  and ${\dot{H}}^1(\mathbb{R})$ can be obtained rigorously; especially leading order term is give by a special self-similar solution called  kink-wave.
The main theorem for large-time behavior is:
\begin{theorem}\label{thm:mainasym}
We have the following results for the large-time behavior of solutions:
\begin{enumerate}
\item
For $0\leq\lambda<1$, let $(u(t),\rho(t),\mu(t),\nu(t))\in\mathcal{D}_{\lambda}$ be a $\lambda$-dissipative solution to the 2HS system subject to the initial data $(\bar{u},\bar{\rho},\bar{\mu},\bar{\nu})\in\mathcal{D}_{\lambda}$ satisfying \eqref{eq:initialdecay}, in the sense of Definition~\ref{def:weak}.
Let
\begin{equation}\label{eq:finalenergy} 
E_\lambda:=\bar{\nu}(\mathbb{R}) -  \lambda\int_{\mathbb{R}}f(\eta) \mathbf{1}_{J}(\eta)\di \eta\geq 0,
\end{equation}
where the density function $f$ and the set $J$ are determined by initial data only; specifically,
$f(\alpha):= 1-\bar{x}'(\alpha)$ with $\bar{x}(\alpha)$ is defined by \eqref{eq:barx1} and $J$ is defined by \eqref{eq:blowupsets1} below.
Then we have
\begin{enumerate}
\item[(i)] Rescaled limit: the following rescaled limit of $u$ exists and is the kink-wave determined by the initial data, the parameter $\lambda$ and $\kappa$:
\begin{equation} \label{eq:v}
v(x):=\lim\limits_{t\to+\infty} \frac{2}{t}{u\left(\frac{t^2}{4}x,t\right)} = 
\left\{
\begin{aligned}
&-\kappa E_\lambda, \quad x< -\kappa E_\lambda,\\
&x, \quad -\kappa E_\lambda\le x\le (1-\kappa)E_\lambda,\\
&(1-\kappa)E_\lambda, \quad x >(1-\kappa)E_\lambda.
\end{aligned}
\right.
\end{equation}
\item[(ii)] Asymptotic expansions: we have the following asymptotic expansions in terms  of the kink wave: 
\begin{align}\label{eq:Linfty}
u(x,t) = \frac{t}{2} v\left(\frac{4x}{t^2}\right)+o(t) ~\textrm{in}~ {L^\infty}(\mathbb{R}),\quad\mbox{as }t\to+\infty,
\end{align}
and
\begin{align}\label{eq:L2}
u_x(x,t) = \partial_x\left[\frac{t}{2} v\left(\frac{4x}{t^2}\right)\right]  + o(1)~\textrm{in}~ {L^2}(\mathbb{R}),\quad\mbox{as }t\to+\infty.
\end{align}
For $\rho$, we have
\begin{align}\label{eq:L2rho}
\rho(x,t)=o(1)~\textrm{in}~L^2(\mathbb R),
\qquad\mbox{as }t\to+\infty.
\end{align}
As a direct consequence, we also have
\begin{align}\label{eq:singular part}
\lim_{t\to+\infty} {\mu}_{s}(t)(\mathbb{R})=\lim_{t\to+\infty} \|\rho(\cdot,t)\|_{L^2(\mathbb{R})}=0,
\end{align}
where ${\mu}_{s}(t)$ is the singular part of the energy measure ${\mu}(t)$ with respect to the Lebesgue measure.
\end{enumerate}

\item
For $\lambda=1$, let $(u,\rho)$ be a dissipative solution to the 2HS system in the sense of Definition \ref{def:dissipative} with initial data $(\bar{u},\bar{\rho})\in\mathcal{D}_1$,
then for $(u,\rho)$, the asymptotic behavior similar to \eqref{eq:v}-\eqref{eq:singular part} hold with $E_\lambda$ replaced by 
\[
E_1 =  \int_{\mathbb{R}\setminus J} (\bar{u}_x^2+\bar{\rho}^2)(\eta)\di \eta,\quad J:=  \left\{\xi\in \mathbb{R}:\bar{u}'(\xi)<0,~~ \bar{\rho}(\xi)=0\right\}. 
\]
\end{enumerate}
\end{theorem}

In the remainder of this paper, we focus solely on the case $0\leq\lambda<1$. The dissipative case $\lambda=1$ can be treated similarly; we sketch the relevant details in Appendix~\ref{app:dissipative}.

The rest of this paper is organized as follows. In Section \ref{sec:char}, the explicit formula for generalized characteristics of $\lambda$-dissipative solutions to the 2HS system will be introduced. The regularity structure, and the existence of the solutions will be then discussed.   In Section \ref{sec:uniqueness}, we prove the uniqueness via characteristics.
In Section \ref{sec:asymp}, the asymptotic behavior is studied and Theorem \ref{thm:mainasym} will be proved. We put the uniqueness of weak solutions to the fully dissipative case into Appendix \ref{app:dissipative}.

\section{Explicit characteristics and regularity structure}\label{sec:char}
\subsection{Explicit characteristics} In this subsection, we introduce the generalized characteristics for $\lambda$-dissipative solutions, which has an explicit formula determined by the initial data $(\bar{u},\bar{\rho},\bar{\mu},\bar{\nu})\in\mathcal{D}_{\lambda}$. The main idea follows from \cite{gao2022regularity}, where an explicit formula for characteristics of conservative solutions to the HS equation \eqref{eq:HS} was introduced and then used in \cite{gao2023asymptotic} for analyzing the asymptotic behavior. To motivate the generalized characteristics for $\lambda$-dissipative solutions, let us begin with the case that the conservative solution $(u(x,t),\rho(x,t),\mu(t),\nu(t))$ is always classical, in the sense that $\di \mu(t)=\di\nu(t)=(u_x^2+\rho^2)(x,t)\di x$ for all time $t\in [0,\infty)$. In this case, the energy conservation  \eqref{eq:conservationlaw} also holds.
For any given $\beta\in\mathbb{R}$, $t\geq0$, we define $x(\beta,t)$ by
\begin{align}\label{eq:defx}
x(\beta,t)+\int_{(-\infty, x(\beta,t))}(u_y^2+\rho^2)(y,t)\di y=\beta. 
\end{align}
Differentiating \eqref{eq:defx} with respect to $t$ and $\beta$ respectively, we obtain
\begin{equation}\label{eq:dtx&dbetax}
\partial_tx(\beta,t)=\frac{[u(u_x^2+\rho^2)](x(\beta,t),t)}{1+(u_x^2+\rho^2)(x(\beta,t),t)},\quad\mbox{and}\quad \partial_{\beta}x(\beta,t)=\frac{1}{1+(u_x^2+\rho^2)(x(\beta,t),t)},
\end{equation}
where we have applied \eqref{eq:conservationlaw} when computing $\partial_tx(\beta,t)$.
For any fixed $\alpha\in \mathbb{R}$, we denote  $\beta(t)$ the unique solution to the following initial-value problem:
\begin{align}\label{eq:betat}
\beta'(t)=u(x(\beta(t),t),t),\quad\mbox{and}\quad \beta(0)=\alpha.
\end{align}
It follows from the chain rule, \eqref{eq:dtx&dbetax} and \eqref{eq:betat} that
\begin{align}\label{eq:translation}
\frac{\di }{\di t}x(\beta(t),t)=\partial_tx(\beta(t),t)+\partial_\beta x(\beta(t),t)\cdot \beta'(t)=u(x(\beta(t),t),t)=\beta'(t).
\end{align}
Next, we denote  $\bar{x}(\alpha)$ the unique solution of
\begin{align}\label{eq:smooth barx} 
\bar{x}(\alpha)+\int_{(-\infty, \bar{x}(\alpha))}(u_y^2+\rho^2)(y,0)\di y=\alpha.
\end{align}
Then it follows from \eqref{eq:defx}, \eqref{eq:betat} and \eqref{eq:translation} that 
\begin{align}\label{eq:rela}
x(\beta(0),0)=\bar{x}(\alpha),\quad\mbox{and}\quad \beta(t)-x(\beta(t),t)=\alpha-\bar{x}(\alpha).
\end{align}
Furthermore, differentiating \eqref{eq:betat} with respect to $t$, and then using the chain rule, \eqref{eq:HS}, \eqref{eq:defx}, \eqref{eq:betat}, and \eqref{eq:translation}, one can verify that 
\[
\beta''(t)=\frac{1}{2}[\beta(t)-x(\beta(t),t)]-\frac{\kappa}{2}\bar{\nu}(\mathbb{R}).
\]
Taking one more time derivative and using \eqref{eq:translation}, we finally obtain
\begin{align}\label{eq:thirdorder}
\beta'''(t)=0.
\end{align}
Moreover, we have the following initial data for the ODE \eqref{eq:thirdorder}: 
\begin{align}\label{eq:initial}
\beta(0)=\alpha,\quad \beta'(0)=u(\bar{x}(\alpha),0),\quad \beta''(0)=\frac{1}{2}[\alpha-\bar{x}(\alpha)]-\frac{\kappa}{2}\bar{\nu}(\mathbb{R}).
\end{align}
Therefore, solving the initial value problem \eqref{eq:thirdorder} and \eqref{eq:initial} yields
\begin{align}\label{eq:global1}
\beta(t)=\alpha+u(\bar{x}(\alpha),0)t+\frac{t^2}{4}\left[\alpha-\bar{x}(\alpha)-\kappa\bar{\nu}(\mathbb{R})\right],
\end{align}
and hence, using \eqref{eq:rela}, we also have 
\begin{align}\label{eq:global}
x(\beta(t),t)=\bar{x}(\alpha)+u(\bar{x}(\alpha),0)t+\frac{t^2}{4}\left[\alpha-\bar{x}(\alpha)-\kappa\bar{\nu}(\mathbb{R})\right].
\end{align}
Therefore, to obtain the global formulae for $\beta(t)$ and $x(\beta(t),t)$, we only need the information of initial data $u(x,0)$.  

Now we apply the same idea to get the explicit formulae for characteristics of $\lambda$-dissipative solutions in the generalized framework.
For general initial data $(\bar{u},\bar{\rho},\bar{\mu},\bar{\nu})$ in $\mathcal{D}_{\lambda}$ satisfying \eqref{eq:initialdecay}, we define $\bar{x}(\alpha)$ by 
\begin{align}\label{eq:barx1}
\bar{x}(\alpha)+\bar{\nu}((-\infty, \bar{x}(\alpha)))\leq \alpha \leq \bar{x}(\alpha)+\bar{\nu}((-\infty, \bar{x}(\alpha)]). 
\end{align} 
Here instead of $\bar{\mu}$, the measure $\bar{\nu}$ was used to define $\bar{x}$ in order to take the dissipation on  $t=0$ into consideration.
It can be directly checked that $\bar{x}(\alpha)$ is well-defined and $\bar{x}(\alpha) \leq  \alpha$  for all $\alpha  \in  \mathbb{R}$. Moreover, $\bar{x}(\alpha)$ is a nondecreasing and Lipschitz continuous function, with Lipschitz constant bounded by 1; see \cite[Proposition 2.1]{gao2022regularity} for the proofs of these facts. The role of $\bar{x}(\alpha)$ is to ``flatten" the singular part of $ \bar{\nu}$ in the $\alpha$-coordinate. 

First, consider $\lambda=0$, i.e., the energy conservative case. 
In this case, we can still use \eqref{eq:global} to define the generalized characteristics. Let $y_1(\alpha,t)=x(\beta(t),t)$ and we have the following explicit formula of characteristics for conservative solutions:
\begin{align}\label{eq:generalizedchara1} 
y_1(\alpha,t):=x(\beta(t),t)=\bar{x}(\alpha)+\bar{u}(\bar{x}(\alpha))t+\frac{t^2}{4}\left[\alpha-\bar{x}(\alpha)-\kappa\bar{\nu}(\mathbb{R})\right].
\end{align}
Moreover, the conservative solution $(u_1,\rho_1,\mu_1,\nu_1)$ (in this case $\mu_1=\nu_1$) can be expressed explicitly as follows: 
\begin{equation}\label{eq:measuresolutionu1}
\left\{
\begin{aligned}
&u_1(x,t)=\frac{\partial}{\partial t}y_1(\alpha,t)=\bar{u}(\bar{x}(\alpha))+\frac{t}{2}\left[\alpha-\bar{x}(\alpha)-\kappa\bar{\nu}(\mathbb{R})\right], ~\textrm{for}~  x =y_1(\alpha,t),\\
&\rho_1(t)\di x=y_1(\cdot,t)\#[(\bar{\rho}\circ \bar{x})\bar{x}' \di\alpha],\\
&\nu_1(t)=\mu_1(t)=y_1(\cdot,t)\# (f\di \alpha), \quad f(\alpha) := 1-\bar{x}'(\alpha).
\end{aligned}
\right.
\end{equation}
Define 
\begin{align}\label{eq:AB0L}
A_0^L:=\{\alpha\in\mathbb{R}:~~\bar{x}'(\alpha)=0\},\qquad B_0^L:=\{\alpha\in\mathbb{R}:~~\bar{x}'(\alpha)>0\}.
\end{align}
\begin{remark}
Some sets in this paper are understood in the almost‑everywhere sense; for convenience (which will not affect the results), we shall not repeat this henceforth.
\end{remark}
From \eqref{eq:barx1} and \cite[Eq. (2.21)]{gao2022regularity}, one has
\begin{equation}\label{eq:flambda}
f(\alpha)=1-\bar{x}'(\alpha)=\left\{
\begin{aligned}
&(\bar{u}_x^2+\bar{\rho}^2)(\bar{x}(\alpha))\bar{x}'(\alpha),\quad  \alpha\in B_0^L,\\
&1,\quad \alpha\in A_0^L.
\end{aligned}\right.
\end{equation}
By \eqref{eq:generalizedchara1} and \eqref{eq:flambda}, the following identity holds:
\begin{equation}\label{eq:separation1}
y_{1\alpha}(\alpha,t)=\left\{
\begin{aligned}
&\bar{x}'(\alpha)\left(\left[1+\frac{t}{2}\bar{u}_x(\bar{x}(\alpha))\right]^2+\frac{t^2}{4}\bar{\rho}^2(\bar{x}(\alpha))\right),\quad \alpha\in B_0^L,\\
&\frac{t^2}{4},\quad \alpha\in A_0^L.
\end{aligned}
\right.
\end{equation}

For initial data $(\bar{u},\bar{\rho},\bar{\mu},\bar{\nu})\in\mathcal{D}_{\lambda}$ satisfying \eqref{eq:initialdecay}, to describe the corresponding characteristics of the $\lambda$-dissipative solutions, we first define the following sets:
\begin{equation}\label{eq:blowupsets}
I_t:=\left\{\alpha\in \mathbb{R}:\bar{u}_x(\bar{x}(\alpha))\leq-\frac{2}{t},~~ \bar{\rho}(\bar{x}(\alpha))=0\right\},\quad J_t:=I_t\cup A_0^L,\quad t>0,
\end{equation}
and
\begin{equation}\label{eq:blowupsets1}
J:=  \cup_{s>0} J_s = \left\{\alpha\in \mathbb{R}:\bar{u}'(\bar{x}(\alpha))<0,~~ \bar{\rho}(\bar{x}(\alpha))=0\right\}\cup A_0^L. 
\end{equation}
Here the set $J$ contains all the Lagrangian labels $\alpha$ that make $y_{1\alpha}(\alpha,t)=0$ at some time $t\geq0$. 
The set $I_t$ contains all the Lagrangian labels which will make $y_{1\alpha}(\alpha,s)=0$ for $0<s=-{2}/{\bar{u}'(\bar{x}(\alpha))}\leq t$. Clearly, 
$I_t$ is an increasing sequence of sets with respect to $t$.
From \eqref{eq:separation1},  $I_t\cap A_0^L=\emptyset$ for any $t>0$. 
With the help of $J_t$, the acceleration of the characteristics $y(\alpha,t)$ for $\lambda$-dissipative solutions is described by
\[
\frac{\partial^2}{\partial t^2}y(\alpha,t) = \frac{1}{2}\left[(1-\kappa)\int_{-\infty}^{\alpha}-\kappa\int_{\alpha}^{+\infty}\right]f(\eta)[1-\lambda \mathbf{1}_{J_t}(\eta)]\di \eta.
\]
For $\lambda=0$, the right hand side becomes $\alpha-\bar{x}(\alpha)-\kappa\bar{\nu}(\mathbb{R})$, which is exactly the the acceleration for conservative characteristics $y_1(\alpha,t)$.
From the above relation, the  explicit formula of  $y(\alpha,t)$ is given by:
\begin{equation}\label{eq:lambda_dissipation}
\begin{aligned}
y(\alpha&,t)=\bar{x}(\alpha)+\bar{u}(\bar{x}(\alpha))t+ \frac{1}{2}\int_{0}^{t}\int_0^s\left[(1-\kappa)\int_{-\infty}^{\alpha}-\kappa\int_{\alpha}^{+\infty}\right]f(\eta)[1-\lambda \mathbf{1}_{J_\tau}(\eta)]\di \eta\di\tau\di s\\
=&\bar{x}(\alpha)+\bar{u}(\bar{x}(\alpha))t+\frac{t^2}{4}\left[\alpha-\bar{x}(\alpha) -  \kappa \bar{\nu}(\mathbb{R}) \right]  - \frac{\lambda}{2}\int_{0}^{t}(t-s)\left[\int_{-\infty}^{\alpha}-\kappa\int_{\mathbb{R}}\right]f(\eta)\mathbf{1}_{J_s}(\eta)\di \eta\di s,
\end{aligned}
\end{equation}
where we used the repeated integral formula for the last equality.
This characteristic implies that $\lambda$-proportion of the concentrated energy corresponding to singularities formed before time $t$ will be dropped.
For $\lambda=0$, the formula \eqref{eq:lambda_dissipation} becomes \eqref{eq:generalizedchara1}, which is the characteristic of the energy conservative solutions. 
The $\lambda$-dissipative solution $(u(t),\rho(t),\mu(t),\nu(t))$ is then recovered by:
\begin{equation}\label{eq:solution2}
\left\{
\begin{aligned}
u(x,t)&=y_t(\alpha,t)=\bar{u}(\bar{x}(\alpha))+\frac{t}{2}\left[\alpha-\bar{x}(\alpha) -\kappa\bar{\nu}(\mathbb{R})\right]\\
&-\frac{\lambda}{2}\int_{0}^{t} \left[\int_{-\infty}^{\alpha}-\kappa\int_{\mathbb{R}}\right]f(\eta)\mathbf{1}_{J_s}(\eta)\di \eta\di s, \quad\textrm{where } x=y(\alpha,t),\\
\rho(x,t)\di x&=y(\cdot,t)\#[(\bar{\rho}\circ \bar{x})\bar{x}' \di\alpha],\\
\mu(t)&=y(\cdot,t)\# [f(1-\lambda \mathbf{1}_{J_t})\di \alpha],\\
\nu(t)&=y(\cdot,t)\# (f\di \alpha).
\end{aligned}
\right.
\end{equation}
Notice that the formulae \eqref{eq:lambda_dissipation} and \eqref{eq:solution2} are explicitly determined by initial data, and convenient to deal with the structure and asymptotic behavior of solutions. 
\begin{remark}[dissipative case $\lambda=1$]\label{rmk:dissipative}
When $\lambda=1$, the singular part of the energy measure vanishes completely, reducing the system to the binary pair $(u,\rho)$ without the measures $\mu$ and $\nu$. Although the full energy-measure framework remains available, this reduced setting removes the associated technical overhead and considerably simplifies the subsequent analysis; see Appendix~\ref{app:dissipative}.
\end{remark}

\subsection{Regularity structure}
In this subsection, we show the structure of singularities of $\lambda$-dissipative solutions to the 2HS system in  Theorem \ref{thm:introduce}, and the existence result in  Theorem \ref{thm:introduce} will be proved in the next section. 
We have the following theorem:
\begin{theorem}\label{thm:measure}
Let $0\leq\lambda<1$. Assume $(\bar{u},\bar{\rho},\bar{\mu},\bar{\nu})\in\mathcal{D}_{\lambda}$ satisfying \eqref{eq:initialdecay},  and functions $\bar{x}(\alpha)$, $f(\alpha)$  are defined by \eqref{eq:barx1} and \eqref{eq:flambda}. 
The characteristic $y(\alpha,t)$ and $(u(t),\rho(t),\mu(t),\nu(t))$ are defined by \eqref{eq:lambda_dissipation} and \eqref{eq:solution2}, separately. Let
\begin{equation}\label{eq:z12}
z_1(x,t)=\inf\{\alpha:~y(\alpha,t)=x\},\quad z_2(x,t)=\sup\{\alpha:~y(\alpha,t)=x\}
\end{equation}
be two pseudo-inverses of $y(\cdot,t)$ for a fixed $t\geq0$, and
\[
A_t^{L,pp}=\{\alpha:~~y_\alpha(\alpha,t)=0,~~z_1(y(\alpha,t),t)<z_2(y(\alpha,t),t)\},
\]
\begin{align}\label{eq:ABt}
A_t^{L,sc}=\{\alpha:~~y_\alpha(\alpha,t)=0,~~z_1(y(\alpha,t),t)=z_2(y(\alpha,t),t)\},\quad\mbox{and}\quad B_t^L=\{\alpha:~~y_\alpha(\alpha,t)>0\}.
\end{align}
Let $\mu(t)=\mu_{ac}(t)+\mu_{pp}(t)+\mu_{sc}(t)$, where $\mu_{ac}(t)$, $\mu_{pp}(t)$, and $\mu_{sc}(t)$ are the absolutely continuous part, pure point part and the singular continuous part of ${\mu}(t)$ respectively. (Same for $\nu(t)$)
Then we have 
\begin{enumerate}
\item [$\bullet$] \textbf{Properties of $\mu(t)$ and $\nu(t)$:}
\item [(i)] Energy conservation: we have $\nu \in C([0,\infty); \mathcal{M}_+(\mathbb{R}))$ and
\begin{align}\label{eq:conser}
\nu(t)(\mathbb{R})= \bar{\nu}(\mathbb{R}),\quad t>0.
\end{align}
\item [(ii)] We have $\mu(t)\ll\nu(t)$, and the following structure of  $\mu(t)$ and $\nu(t)$ holds
\begin{equation}\label{eq:decompmutmu}
\begin{aligned}
\mu_{pp}(t)=y(\cdot,t)\#[f(1-&\lambda \mathbf{1}_{J_t})|_{A_t^{L,pp}}\di \alpha],\quad \mu_{sc}(t)=y(\cdot,t)\#[f(1-\lambda \mathbf{1}_{J_t})|_{A_t^{L,sc}}\di \alpha],\\
&\mu_{ac}(t)=y(\cdot,t)\#[f(1-\lambda \mathbf{1}_{J_t})|_{B_t^{L}}\di \alpha],
\end{aligned}
\end{equation}
and
\begin{align}\label{eq:decompmutnu}
\nu_{pp}(t)=y(\cdot,t)\#(f|_{A_t^{L,pp}}\di \alpha),\quad \nu_{sc}(t)=y(\cdot,t)\#(f|_{A_t^{L,sc}}\di \alpha),\quad \nu_{ac}(t)=y(\cdot,t)\#(f|_{B_t^L}\di \alpha).
\end{align}
Moreover, for any $t> 0$, we have
\begin{align}\label{eq:singulart}
A_t^{L,pp}\cup A_t^{L,sc}=\left\{\alpha\in B_0^L:~~\bar{u}_x(\bar{x}(\alpha))=-\frac{2}{t},~~ \bar{\rho}(\bar{x}(\alpha))=0\right\}:=A^L_t,
\end{align}
where $B_0^L$ was defined in \eqref{eq:AB0L}. This implies that for any $t>0$, $\nu_{pp}(t)$ and $\nu_{sc}(t)$ are determined by the absolutely continuous part of $\bar{\mu}_{ac}$, i.e., determined by $\bar{u}_x$ and $\bar{\rho}$. 

\item [(iii)] There are at most countably many time $t\geq0$  such that either the pure point part or the singular continuous part of $\nu(t)$ (or $\mu(t)$ for $\lambda\neq 1$) is not zero; in other words, the set
\begin{align}\label{eq:singulartime}
T_s=T_{pp}\cup T_{sc},\quad T_{pp}:=\{t:~\nu_{pp}(t)\neq 0\},~~	
T_{sc}:=\{t:~\nu_{sc}(t)\neq 0\}.
\end{align}
is countable.

\item [(iv)] We also have
\begin{equation}\label{eq:continu1}
\mu(t+)=\mu(t)=\mu(t-)-\lambda\mu_s(t-)=\mu(t-)-\lambda\nu_s(t)=\mu_{ac}(t-)+(1-\lambda)\mu_s(t-),
\end{equation}
and
\begin{equation}\label{eq:continu2}
\lim_{s\to t-}\mu(s)=\mu(t-),\quad \lim_{s\to t+}\mu(s)=\mu(t+),\quad t>0.
\end{equation}
Moreover,
\begin{align}\label{eq:continuoustime}
\lim_{s\to t}\mu(s) = \mu(t)=\mu_{ac}(t),\quad t\in [0,+\infty)\setminus T_s.
\end{align}

\item[(v)] For every $t>0$, the dissipated part of the energy has the decomposition
\[
\di(\nu-\mu)(t)=h(x,t)\di x+\lambda\di\nu_s(t),
\qquad
h(x,t)=\frac{\lambda}{1-\lambda}u_x^2(x,t)
\mathbf1_{\{\rho(x,t)=0,\;u_x(x,t)\geq2/t\}}.
\]
Moreover, with the convention $J_0=A_0^L$, for every $0\leq\tilde t<t$,
\begin{equation}\label{eq:outgoingcusp}
h(x,t)\leq\frac{4\lambda}{(1-\lambda)(t-\tilde t)^2},
\qquad
\text{for a.e. }x\in y(J_{\tilde t},t).
\end{equation}

\

\item [$\bullet$] \textbf{Properties of $\rho(x,t)$ and $u(x,t)$:}

\item [(vi)] For every $t\geq0$, the function $u(\cdot,t)$ is globally absolutely continuous, $\rho(\cdot,t)\in L^2(\mathbb R)$, and
\begin{gather}\label{eq:measuresolutionrho}
\rho(x,t)=\left\{
\begin{split}
&\rho(y(\alpha,t),t)=\frac{\bar{\rho}(\bar{x}(\alpha))\bar{x}'(\alpha)}{y_\alpha(\alpha,t)}~\textrm{ for }~ x=y(\alpha,t),~~\alpha\in B_t^L,\\
&0~\textrm{ for }~ x=y(\alpha,t),\quad \alpha\in A_t^L.
\end{split}
\right.
\end{gather}
Moreover, we have 
\begin{align}\label{eq:acpart}
\di\mu_{ac}(t)=(u_x^2+\rho^2)(x,t)\di x.
\end{align}
We also have
\begin{equation}\label{eq:rhoproperties}
\rho\in L^\infty(0,\infty;L^2(\mathbb{R}))\cap C_w([0,\infty);L^2(\mathbb{R})),\quad  \|\rho(\cdot,t)\|_{L^2}^2\leq\bar{\mu}(\mathbb{R}).
\end{equation}
and
\begin{equation}\label{eq:propertiesu}
u\in C([0,\infty);C_b(\mathbb{R}))\cap C^{1/2}_{loc}(\mathbb{R}\times [0,\infty)),~~ u_x\in L^\infty(0,\infty;L^2(\mathbb{R})),~~ u_t\in L_{loc}^2(\mathbb{R}\times [0,\infty)).
\end{equation}

\item [(vii)] The Radon--Nikodym derivative satisfies
$\frac{\di\mu(t)}{\di\nu(t)}(x)\in\{1,1-\lambda\}$, and
$\frac{\di\mu(t)}{\di\nu(t)}=1$ for a.e. $x\in\mathbb R$ satisfying
$u_x(x,t)<0$ or $\rho(x,t)\neq0$.

\end{enumerate}
\end{theorem}
\begin{proof}
All identities involving derivatives with respect to $\alpha$ are understood for almost every $\alpha$. Since $A_0^L\subset J_t$ and $A_0^L\cup B_0^L=\mathbb R$, one has $J_t^c\subset B_0^L$. Differentiating \eqref{eq:lambda_dissipation} with respect to $\alpha$ and using \eqref{eq:flambda}, we obtain
\begin{equation}\label{eq:separation}
\begin{aligned}
y_{\alpha}(\alpha,t)&=\left\{
\begin{aligned}
&(1-\lambda)\bar{x}'(\alpha)\left[1+\frac{t}{2}\bar{u}_x(\bar{x}(\alpha))\right]^2,
&& \alpha\in B_0^L\cap J_t,\\
&\bar{x}'(\alpha)\left(\left[1+\frac{t}{2}\bar{u}_x(\bar{x}(\alpha))\right]^2
 +\frac{t^2}{4}\bar{\rho}^2(\bar{x}(\alpha))\right),
&& \alpha\in J_t^c\subset B_0^L,\\
&(1-\lambda)\frac{t^2}{4},
&& \alpha\in A_0^L,
\end{aligned}
\right.\\
&=\left\{
\begin{aligned}
&[1-\lambda\mathbf{1}_{J_t}(\alpha)]\bar{x}'(\alpha)
\left(\left[1+\frac{t}{2}\bar{u}_x(\bar{x}(\alpha))\right]^2
+\frac{t^2}{4}\bar{\rho}^2(\bar{x}(\alpha))\right),
&& \alpha\in B_0^L,\\
&[1-\lambda\mathbf{1}_{J_t}(\alpha)]\frac{t^2}{4}
=(1-\lambda)\frac{t^2}{4},
&& \alpha\in A_0^L.
\end{aligned}
\right.
\end{aligned}
\end{equation}
In particular, $y_\alpha(\alpha,t)\geq0$. Thus $y(\cdot,t)$ is nondecreasing. Moreover, the explicit formula \eqref{eq:lambda_dissipation} shows that $y(\alpha,t)-\alpha$ remains bounded as $\alpha\to\pm\infty$ on every bounded time interval; hence $y(\cdot,t)$ is surjective.

\medskip
\noindent\textit{Proof of (i).}
Let $\phi\in C_b(\mathbb R)$. Since $f\in L^1(\mathbb R)$ and $s\mapsto y(\alpha,s)$ is continuous for every $\alpha$, the dominated convergence theorem and \eqref{eq:solution2}$_4$ give
\[
\lim_{s\to t}\int_{\mathbb R}\phi(x)\di\nu(s)
=\lim_{s\to t}\int_{\mathbb R}\phi(y(\alpha,s))f(\alpha)\di\alpha
=\int_{\mathbb R}\phi(y(\alpha,t))f(\alpha)\di\alpha
=\int_{\mathbb R}\phi(x)\di\nu(t).
\]
Consequently, $\nu\in C([0,\infty);\mathcal M_+(\mathbb R))$. Taking $\phi\equiv1$ yields
\[
\nu(t)(\mathbb R)=\int_{\mathbb R}f(\alpha)\di\alpha=\bar\nu(\mathbb R),
\]
which is \eqref{eq:conser}.

\medskip
\noindent\textit{Proof of (ii).}
For every Borel set $E\subset\mathbb R$, \eqref{eq:solution2} gives
\[
0\leq\mu(t)(E)
=\int_{y^{-1}(E,t)}f(\alpha)[1-\lambda\mathbf 1_{J_t}(\alpha)]\di\alpha
\leq\int_{y^{-1}(E,t)}f(\alpha)\di\alpha
=\nu(t)(E),
\]
so $\mu(t)\ll\nu(t)$. Applying Lemma~\ref{lmm:A2} to the two push-forward representations in \eqref{eq:solution2} gives \eqref{eq:decompmutmu} and \eqref{eq:decompmutnu}.

It remains to identify the zero set of $y_\alpha(\cdot,t)$. If $t>0$ and $\alpha\in A_0^L$, then \eqref{eq:separation} gives $y_\alpha(\alpha,t)=(1-\lambda)t^2/4>0$. If $\alpha\in J_t^c$, then the second expression in \eqref{eq:separation} is strictly positive. Finally, for $\alpha\in B_0^L\cap J_t$, the first expression in \eqref{eq:separation} vanishes precisely when
\[
\bar u_x(\bar x(\alpha))=-\frac2t,
\qquad
\bar\rho(\bar x(\alpha))=0.
\]
This proves \eqref{eq:singulart}. In particular, all singular parts formed at a positive time are generated by the absolutely continuous initial energy, and therefore are determined by $\bar u_x$ and $\bar\rho$.

\medskip
\noindent\textit{Proof of (iii).}
For distinct positive times $t_1\neq t_2$, relation \eqref{eq:singulart} implies
\[
A_{t_1}^L\cap A_{t_2}^L=\emptyset.
\]
Let $\sigma$ be the finite measure $\di\sigma=f(\alpha)\di\alpha$ on the label space. If $t\in T_{pp}$, then by \eqref{eq:decompmutnu},
\[
\sigma(A_t^{L,pp})=\nu_{pp}(t)(\mathbb R)>0.
\]
The sets $A_t^{L,pp}$ are pairwise disjoint as $t$ varies. A finite measure space contains at most countably many pairwise disjoint measurable sets of positive measure; hence $T_{pp}$ is countable. The same argument, using
\[
\sigma(A_t^{L,sc})=\nu_{sc}(t)(\mathbb R)>0,
\]
shows that $T_{sc}$ is countable. The possible singular time $t=0$ adds at most one element. Therefore $T_s=T_{pp}\cup T_{sc}$ is countable, proving \eqref{eq:singulartime}.

\medskip
\noindent\textit{Proof of (iv).}
For $t>0$, set
\[
J_{t-}:=A_0^L\cup\bigcup_{0<s<t}J_s.
\]
Since $J_s$ is increasing in $s$, one has pointwise in $\alpha$
\[
\lim_{s\to t-}\mathbf 1_{J_s}(\alpha)=\mathbf 1_{J_{t-}}(\alpha),
\qquad
\lim_{s\to t+}\mathbf 1_{J_s}(\alpha)=\mathbf 1_{J_t}(\alpha).
\]
Using again the continuity of $y(\alpha,s)$ and dominated convergence, we obtain the weak-star limits
\[
\mu(t-)=y(\cdot,t)\#[f(1-\lambda\mathbf 1_{J_{t-}})\di\alpha],
\qquad
\mu(t+)=\mu(t)=y(\cdot,t)\#[f(1-\lambda\mathbf 1_{J_t})\di\alpha],
\]
which proves \eqref{eq:continu2}. Moreover,
\[
J_t=J_{t-}\mathbin{\dot\cup}A_t^L.
\]
Since $\mathbf 1_{J_{t-}}=0$ on $A_t^L$, Lemma~\ref{lmm:A2} and \eqref{eq:decompmutnu} yield
\[
\mu_s(t-)=y(\cdot,t)\#[f|_{A_t^L}\di\alpha]=\nu_s(t).
\]
Consequently,
\[
\begin{aligned}
\mu(t)
&=y(\cdot,t)\#[f(1-\lambda\mathbf 1_{J_{t-}})\di\alpha]
  -\lambda y(\cdot,t)\#[f\mathbf 1_{A_t^L}\di\alpha]\\
&=\mu(t-)-\lambda\mu_s(t-)
 =\mu(t-)-\lambda\nu_s(t)\\
&=\mu_{ac}(t-)+(1-\lambda)\mu_s(t-),
\end{aligned}
\]
which is \eqref{eq:continu1}. If $t\notin T_s$, then $\nu_s(t)=0$, and hence $\mu_s(t)=0$ by \eqref{eq:decompmutmu}. The left and right limits therefore agree with $\mu(t)=\mu_{ac}(t)$, proving \eqref{eq:continuoustime} for $t>0$. If $0\notin T_s$, the same right-limit argument at the initial time gives \eqref{eq:continuoustime} for $t=0$.

\medskip
We next establish the spatial identities used in both (v) and (vi). Fix $t>0$. From \eqref{eq:solution2}$_2$, interpreted as an identity of locally finite signed measures, and from the fact that
\[
\bar\rho(\bar x(\alpha))\bar x'(\alpha)=0
\qquad\text{for a.e. }\alpha\in A_t^L,
\]
Applying Lemma~\ref{lmm:A2} separately to the positive and negative parts, together with the monotone change-of-variables formula, gives \eqref{eq:measuresolutionrho}. In particular, after choosing the representative that vanishes on the Lebesgue-null set $y(A_t^L,t)$,
\begin{equation}\label{eq:rholambdat}
\begin{aligned}
\rho(y(\alpha,t),t)y_\alpha(\alpha,t)
&=\bar\rho(\bar x(\alpha))\bar x'(\alpha),
&& \alpha\in\mathbb R,\\
\rho(y(\alpha,t),t)&=0,
&& \alpha\in J_t.
\end{aligned}
\end{equation}

The label function $\alpha\mapsto y_t(\alpha,t)$ is locally absolutely continuous. Differentiating its explicit expression in \eqref{eq:solution2}$_1$ gives
\[
\partial_\alpha y_t(\alpha,t)
=\bar u_x(\bar x(\alpha))\bar x'(\alpha)
+\frac12\int_0^t f(\alpha)[1-\lambda\mathbf 1_{J_s}(\alpha)]\di s.
\]
For $\alpha\in A_t^L$, one has $\alpha\in B_0^L$, $\bar\rho(\bar x(\alpha))=0$, $\bar u_x(\bar x(\alpha))=-2/t$, and $\mathbf 1_{J_s}(\alpha)=0$ for $0\leq s<t$. Hence, by \eqref{eq:flambda},
\begin{equation}\label{eq:importantrelation1}
\begin{aligned}
&\bar u_x(\bar x(\alpha))\bar x'(\alpha)
+\frac12\int_0^t f(\alpha)[1-\lambda\mathbf 1_{J_s}(\alpha)]\di s\\
&\qquad
=\bar u_x(\bar x(\alpha))\bar x'(\alpha)+\frac t2f(\alpha)
=\bar u_x(\bar x(\alpha))\bar x'(\alpha)
  \left[1+\frac t2\bar u_x(\bar x(\alpha))\right]
=0.
\end{aligned}
\end{equation}
Thus $y_t(\alpha,t)$ is constant on every interval on which $y(\cdot,t)$ is constant. Consequently, the prescription $u(y(\alpha,t),t)=y_t(\alpha,t)$ in \eqref{eq:solution2}$_1$ defines an unambiguous Eulerian function.

We now prove its local absolute continuity without using an inverse-function estimate. Let $\phi\in C_c^1(\mathbb R)$. By the area formula for the monotone map $y(\cdot,t)$, followed by integration by parts in $\alpha$, we have
\begin{align}\label{eq:AClebesgue}
\int_{\mathbb R}u(x,t)\phi_x(x)\di x
&=\int_{B_t^L}u(y(\alpha,t),t)\phi_x(y(\alpha,t))y_\alpha(\alpha,t)\di\alpha\notag\\
&=-\int_{\mathbb R}\phi(y(\alpha,t))
\left[\bar u_x(\bar x(\alpha))\bar x'(\alpha)
+\frac12\int_0^t f(\alpha)[1-\lambda\mathbf 1_{J_s}(\alpha)]\di s\right]\di\alpha.
\end{align}
The expression in brackets vanishes on $A_t^L$ by \eqref{eq:importantrelation1}. Hence, after restricting to the inverse image of an arbitrary compact set and applying Lemma~\ref{lmm:A2} separately to the positive and negative parts, the signed measure on the right-hand side of \eqref{eq:AClebesgue} is locally absolutely continuous with respect to Lebesgue measure. It follows that $u(\cdot,t)\in W^{1,1}_{loc}(\mathbb R)$ and therefore is locally absolutely continuous. The chain rule now yields
\begin{equation}\label{eq:ulambdat}
u_x(y(\alpha,t),t)y_\alpha(\alpha,t)
=\bar u_x(\bar x(\alpha))\bar x'(\alpha)
+\frac t2f(\alpha)
-\frac\lambda2\int_0^t f(\alpha)\mathbf 1_{J_s}(\alpha)\di s.
\end{equation}

We next compute the energy density on $B_t^L$. First let
$\alpha\in B_0^L\cap J_t\cap B_t^L$. Then
$\bar u_x(\bar x(\alpha))<0$, $\bar\rho(\bar x(\alpha))=0$, and the breaking time of this label is
$-2/\bar u_x(\bar x(\alpha))<t$. Therefore \eqref{eq:ulambdat}, \eqref{eq:flambda}, and \eqref{eq:rholambdat} give
\begin{equation}\label{eq:uxt}
\begin{aligned}
u_x(y(\alpha,t),t)y_\alpha(\alpha,t)
&=\bar u_x(\bar x(\alpha))\bar x'(\alpha)+\frac t2f(\alpha)
-\frac\lambda2\int_{-\frac{2}{\bar u_x(\bar x(\alpha))}}^t f(\alpha)\di s\\
&=(1-\lambda)\bar u_x(\bar x(\alpha))\bar x'(\alpha)
\left[1+\frac t2\bar u_x(\bar x(\alpha))\right].
\end{aligned}
\end{equation}
Combining this identity with \eqref{eq:separation} and \eqref{eq:rholambdat}, we obtain
\begin{equation}\label{eq:B0LJt}
(u_x^2+\rho^2)(y(\alpha,t),t)y_\alpha(\alpha,t)
=(1-\lambda)(\bar u_x^2+\bar\rho^2)(\bar x(\alpha))\bar x'(\alpha),
\quad \alpha\in B_0^L\cap J_t\cap B_t^L.
\end{equation}

If $\alpha\in J_t^c\cap B_t^L$, then $\mathbf 1_{J_s}(\alpha)=0$ for every $0\leq s\leq t$. Hence \eqref{eq:ulambdat} becomes
\[
u_x(y(\alpha,t),t)y_\alpha(\alpha,t)
=\bar x'(\alpha)\left[
\bar u_x(\bar x(\alpha))\left(1+\frac t2\bar u_x(\bar x(\alpha))\right)
+\frac t2\bar\rho^2(\bar x(\alpha))\right].
\]
Writing 
\[
w_0=\bar u_x(\bar x(\alpha)),\quad r_0=\bar\rho(\bar x(\alpha)),\quad
D=\left(1+\frac t2w_0\right)^2+\frac{t^2}{4}r_0^2,
\]
then direct computations show that
\[
\left[w_0\left(1+\frac{t}{2}w_0\right)+\frac{t}{2}r_0^2\right]^2+r_0^2
=D(w_0^2+r_0^2).
\]
Consequently, \eqref{eq:rholambdat} and \eqref{eq:separation} give
\begin{equation}\label{eq:B0LJtc}
(u_x^2+\rho^2)(y(\alpha,t),t)y_\alpha(\alpha,t)
=(\bar u_x^2+\bar\rho^2)(\bar x(\alpha))\bar x'(\alpha),
\quad \alpha\in J_t^c\cap B_t^L.
\end{equation}
Combining \eqref{eq:B0LJt}, \eqref{eq:B0LJtc}, and \eqref{eq:flambda}, we conclude that
\begin{equation}\label{eq:B0L}
\begin{aligned}
(u_x^2+\rho^2)(y(\alpha,t),t)y_\alpha(\alpha,t)
&=[1-\lambda\mathbf 1_{J_t}(\alpha)]
(\bar u_x^2+\bar\rho^2)(\bar x(\alpha))\bar x'(\alpha)\\
&=[1-\lambda\mathbf 1_{J_t}(\alpha)]f(\alpha),
\quad \alpha\in B_0^L\cap B_t^L.
\end{aligned}
\end{equation}

Finally, let $\alpha\in A_0^L$. Then $\bar x'(\alpha)=0$, $f(\alpha)=1$, $\mathbf 1_{J_t}(\alpha)=1$, and, by \eqref{eq:separation}, $y_\alpha(\alpha,t)=(1-\lambda)t^2/4>0$. Equations \eqref{eq:ulambdat} and \eqref{eq:rholambdat} yield
\begin{equation}\label{eq:uxestimate}
u_x(y(\alpha,t),t)y_\alpha(\alpha,t)=(1-\lambda)\frac t2,
\quad \alpha\in A_0^L,
\end{equation}
and hence
\begin{equation}\label{eq:A0L}
(u_x^2+\rho^2)(y(\alpha,t),t)y_\alpha(\alpha,t)
=[1-\lambda\mathbf 1_{J_t}(\alpha)]f(\alpha),
\quad \alpha\in A_0^L.
\end{equation}
Equations \eqref{eq:B0L} and \eqref{eq:A0L} give the common identity
\begin{align}\label{eq:acpart2}
(u_x^2+\rho^2)(y(\alpha,t),t)y_\alpha(\alpha,t)
=[1-\lambda\mathbf 1_{J_t}(\alpha)]f(\alpha),
\qquad \alpha\in B_t^L.
\end{align}
We first prove (vi).

\noindent\textit{Completion of the proof of (vi).}
By \eqref{eq:decompmutmu}, the absolutely continuous part of $\mu(t)$ is
\[
\mu_{ac}(t)=y(\cdot,t)\#[f(1-\lambda\mathbf 1_{J_t})|_{B_t^L}\di\alpha].
\]
Since $\mathcal L(\mathbb R\setminus y(B_t^L,t))=0$, the change-of-variables formula and \eqref{eq:acpart2} imply \eqref{eq:acpart}.

Furthermore, $J_t\supset A_0^L$ and $J_t$ is increasing, so
\[
\begin{aligned}
\|u_x(\cdot,t)\|_{L^2}^2+\|\rho(\cdot,t)\|_{L^2}^2
&=\mu_{ac}(t)(\mathbb R)\leq\mu(t)(\mathbb R)\\
&=\int_{\mathbb R}f(\alpha)[1-\lambda\mathbf 1_{J_t}(\alpha)]\di\alpha\leq\int_{\mathbb R}f(\alpha)[1-\lambda\mathbf 1_{A_0^L}(\alpha)]\di\alpha
=\bar\mu(\mathbb R).
\end{aligned}
\]
Thus $u_x,\rho\in L^\infty(0,\infty;L^2(\mathbb R))$. Since $u(\cdot,t)$ is locally absolutely continuous and $u_x(\cdot,t)\in L^2(\mathbb R)$, Lemma~\ref{lem:A3} shows that $u(\cdot,t)$ is globally absolutely continuous.

To prove weak continuity of $\rho$, let $\phi\in C_c^\infty(\mathbb R)$. The locally finite signed-measure identity in \eqref{eq:solution2}$_2$ gives
\[
\int_{\mathbb R}\phi(x)\rho(x,t)\di x
=\int_{\mathbb R}\phi(y(\alpha,t))
\bar\rho(\bar x(\alpha))\bar x'(\alpha)\di\alpha.
\]
On every bounded time interval, the inverse image under $y(\cdot,t)$ of $\operatorname{supp}\phi$ remains in a fixed bounded interval of labels. On such intervals, $(\bar\rho\circ\bar x)\bar x'$ belongs to $L^1$ by the Cauchy--Schwarz inequality and the monotone change of variables. Hence the right-hand side is continuous in $t$ by local dominated convergence. The uniform $L^2$ bound and the density of $C_c^\infty(\mathbb R)$ in $L^2(\mathbb R)$ yield
$\rho\in C_w([0,\infty);L^2(\mathbb R))$, proving \eqref{eq:rhoproperties}.

The spatial estimate
\[
|u(x,t)-u(z,t)|\leq\|u_x(\cdot,t)\|_{L^2}|x-z|^{1/2}
\]
is uniform in $t$. Moreover, \eqref{eq:solution2} and the bound
$0\leq\alpha-\bar x(\alpha)\leq\bar\nu(\mathbb R)$ show that both
$y_t(\alpha,t)$ and $y_{tt}(\alpha,t)$ are bounded uniformly in $\alpha$ on every bounded time interval. Given $x=y(\alpha,t)$, we therefore obtain, for $s$ close to $t$,
\[
\begin{aligned}
|u(x,t)-u(x,s)|
&\leq |u(y(\alpha,t),t)-u(y(\alpha,s),s)|
   +|u(y(\alpha,s),s)-u(y(\alpha,t),s)|\\
&\leq C_T|t-s|+C_T|y(\alpha,t)-y(\alpha,s)|^{1/2}
\leq C_T\bigl(|t-s|+|t-s|^{1/2}\bigr).
\end{aligned}
\]
It follows that $u\in C([0,\infty);C_b(\mathbb R))\cap C^{1/2}_{loc}(\mathbb R\times[0,\infty))$.

It remains to prove the asserted regularity of $u_t$ without assuming the field equation in advance. Differentiating \eqref{eq:lambda_dissipation} twice with respect to $t$ gives
\[
y_{tt}(\alpha,t)=\frac12\left[\int_{-\infty}^{\alpha}-\kappa\int_{\mathbb R}\right]
f(\eta)[1-\lambda\mathbf 1_{J_t}(\eta)]\di\eta.
\]
On the other hand, the push-forward formula for $\mu(t)$ in \eqref{eq:solution2} implies, for almost every $\alpha\in B_t^L$,
\[
F(y(\alpha,t),t)
=\frac12\left[\int_{-\infty}^{\alpha}-\kappa\int_{\mathbb R}\right]
f(\eta)[1-\lambda\mathbf 1_{J_t}(\eta)]\di\eta
=y_{tt}(\alpha,t).
\]
Here the possible ambiguity between open and closed cumulative intervals occurs only on the Lebesgue-null image of $A_t^L$ and is therefore irrelevant. The preceding expression also shows that $F$ is bounded on every bounded time interval.

The explicit formulas give $(y_\alpha)_t=(y_t)_\alpha$ for almost every $(\alpha,t)$. Hence, for every $\phi\in C_c^1(\mathbb R\times[0,\infty))$,
\[
\begin{aligned}
&\partial_t\bigl[y_t(\alpha,t)\phi(y(\alpha,t),t)y_\alpha(\alpha,t)\bigr]
-\partial_\alpha\left[\frac12y_t^2(\alpha,t)\phi(y(\alpha,t),t)\right]\\
&\qquad
=\left[y_t(\alpha,t)\phi_t(y(\alpha,t),t)
+\frac12y_t^2(\alpha,t)\phi_x(y(\alpha,t),t)
+F(y(\alpha,t),t)\phi(y(\alpha,t),t)\right]y_\alpha(\alpha,t)
\end{aligned}
\]
for almost every $(\alpha,t)$. Integrating in $(\alpha,t)$, using the area formula and the compact support of $\phi$, gives
\[
\int_0^\infty\!\int_{\mathbb R}
\left(u\phi_t+\frac12u^2\phi_x+F\phi\right)\di x\di t
+\int_{\mathbb R}\bar u(x)\phi(x,0)\di x=0.
\]
Thus
\[
u_t+\left(\frac{u^2}{2}\right)_x=F
\]
in the sense of distributions. Since $u(\cdot,t)$ is locally absolutely continuous, this is
$u_t=F-uu_x$. The boundedness of $F$ and $u$, together with
$u_x\in L^\infty(0,\infty;L^2(\mathbb R))$, yields
$u_t\in L^2_{loc}(\mathbb R\times[0,\infty))$. This proves \eqref{eq:propertiesu}. The assertions at $t=0$ follow from the initial data and \eqref{eq:initialdecay}.

\medskip
\noindent\textit{Proof of (v).}
From \eqref{eq:solution2},
\[
(\nu-\mu)(t)=\lambda y(\cdot,t)\#[f\mathbf 1_{J_t}\di\alpha].
\]
Since $A_t^L\subset J_t$, Lemma~\ref{lmm:A2} gives
\[
(\nu-\mu)_s(t)=\lambda y(\cdot,t)\#[f|_{A_t^L}\di\alpha]
=\lambda\nu_s(t).
\]
Write $\di(\nu-\mu)_{ac}(t)=h(x,t)\di x$. The change-of-variables formula on $B_t^L$ yields
\begin{equation}\label{eq:hyalpha1}
h(y(\alpha,t),t)y_\alpha(\alpha,t)
=\lambda f(\alpha)\mathbf 1_{J_t}(\alpha),
\qquad \alpha\in B_t^L.
\end{equation}

We now identify the Eulerian image of the broken labels. 

Let
$\alpha\in B_0^L\cap J_t\cap B_t^L$ and set
$\tau=-2/\bar u_x(\bar x(\alpha))$. Then
\begin{equation}\label{eq:inequality1}
0<\tau<t,
\qquad
\bar\rho(\bar x(\alpha))=0.
\end{equation}
By \eqref{eq:flambda} and \eqref{eq:separation},
\begin{equation}\label{eq:hyalpha2}
y_\alpha(\alpha,t)=\left\{
\begin{aligned}
&(1-\lambda)\bar x'(\alpha)
\left[1+\frac t2\bar u_x(\bar x(\alpha))\right]^2
=\frac{1-\lambda}{4}f(\alpha)(t-\tau)^2,
&& \alpha\in B_0^L\cap J_t\cap B_t^L,\\
&(1-\lambda)\frac{t^2}{4},
&& \alpha\in A_0^L.
\end{aligned}
\right.
\end{equation}
Using \eqref{eq:rholambdat}, \eqref{eq:uxt}, and \eqref{eq:uxestimate}, we obtain
\[
\rho(y(\alpha,t),t)=0,
\qquad
u_x(y(\alpha,t),t)=\left\{
\begin{aligned}
&\dfrac2{t-\tau}>\dfrac2t,
&& \alpha\in B_0^L\cap J_t\cap B_t^L,\\
&\dfrac2t,
&& \alpha\in A_0^L.
\end{aligned}
\right.
\]

Let $\alpha\in J_t^c$. Recall
\[
w_0=\bar u_x(\bar x(\alpha)),
\qquad
r_0=\bar\rho(\bar x(\alpha)),
\qquad
D=\left(1+\frac t2w_0\right)^2+\frac{t^2}{4}r_0^2.
\]
Then $D>0$, and \eqref{eq:flambda}, \eqref{eq:separation}, \eqref{eq:rholambdat}, and \eqref{eq:ulambdat} imply
\[
u_x(y(\alpha,t),t)
=\frac{w_0(1+tw_0/2)+(t/2)r_0^2}{D},
\qquad
\rho(y(\alpha,t),t)=\frac{r_0}{D}.
\]
For $\alpha\in J_t^c$, if $\rho(y(\alpha,t),t)=0$, then $r_0=0$. In this case, if $w_0<0$, the condition $\alpha\notin J_t$ gives $1+tw_0/2>0$, and hence $u_x(y(\alpha,t),t)<0$. If $w_0\geq0$, then
\[
0\leq u_x(y(\alpha,t),t)=\frac{w_0}{1+tw_0/2}<\frac2t.
\]

Thus, up to a null set,
\[
\mathbf 1_{\{\rho(\cdot,t)=0,\;u_x(\cdot,t)\geq2/t\}}(y(\alpha,t))
=\mathbf 1_{J_t}(\alpha),
\qquad \alpha\in B_t^L.
\]

On $J_t\cap B_t^L$, identity \eqref{eq:acpart2} reduces to
\[
u_x^2(y(\alpha,t),t)y_\alpha(\alpha,t)=(1-\lambda)f(\alpha).
\]
Combining this relation with \eqref{eq:hyalpha1} and the preceding phase characterization gives
\[
h(x,t)=\frac{\lambda}{1-\lambda}u_x^2(x,t)
\mathbf 1_{\{\rho(x,t)=0,\;u_x(x,t)\geq2/t\}}
\quad\text{for a.e. }x\in\mathbb R.
\]
More explicitly, \eqref{eq:hyalpha1} and \eqref{eq:hyalpha2} yield
\begin{equation}\label{eq:hyalpha3}
h(y(\alpha,t),t)=\left\{
\begin{aligned}
&\dfrac{\lambda}{1-\lambda}
\left[\dfrac{2\bar u_x(\bar x(\alpha))}
{2+t\bar u_x(\bar x(\alpha))}\right]^2
=\dfrac{4\lambda}{(1-\lambda)(t-\tau)^2},
&& \alpha\in B_0^L\cap J_t\cap B_t^L,\\
&\dfrac{4\lambda}{(1-\lambda)t^2},
&& \alpha\in A_0^L.
\end{aligned}
\right.
\end{equation}
This proves the asserted decomposition of $\nu(t)-\mu(t)$.

Finally, fix $0\leq\tilde t<t$, with the convention $J_0=A_0^L$. For almost every $x\in y(J_{\tilde t},t)$, write $x=y(\alpha,t)$ with $\alpha\in J_{\tilde t}$. If $\alpha\in A_0^L$, assign $\tau=0$; otherwise its breaking time satisfies $0<\tau\leq\tilde t$. In both cases, \eqref{eq:hyalpha3} gives
\[
h(x,t)=\frac{4\lambda}{(1-\lambda)(t-\tau)^2}
\leq\frac{4\lambda}{(1-\lambda)(t-\tilde t)^2},
\]
which proves \eqref{eq:outgoingcusp}.

\medskip
\noindent\textit{Proof of (vii).}
For $t>0$, the result follows directly from (v) and \eqref{eq:acpart}. Indeed, on the absolutely continuous outgoing region
$\{\rho=0,\;u_x\geq2/t\}$,
\[
\di\nu_{ac}(t)
=\left(1+\frac\lambda{1-\lambda}\right)u_x^2\di x
=\frac1{1-\lambda}\di\mu_{ac}(t),
\]
so $\di\mu/\di\nu=1-\lambda$ there. On the complement of that region, $h=0$, hence $\mu_{ac}=\nu_{ac}$ and $\di\mu/\di\nu=1$. On the singular part,
\[
\mu_s(t)=(1-\lambda)\nu_s(t)
\]
by \eqref{eq:decompmutmu}--\eqref{eq:decompmutnu}. Therefore
\[
\frac{\di\mu(t)}{\di\nu(t)}\in\{1,1-\lambda\}.
\]
If $\rho(x,t)\neq0$ or $u_x(x,t)<0$, then $x$ does not belong to the outgoing region, and hence $\di\mu(t)/\di\nu(t)=1$ for almost every such $x$. The assertion at $t=0$ follows from \eqref{eq:initialdecay}.
\end{proof}

Use the characteristics method provided above to calculate the explicit formula of solution to Example \ref{example} for $t>2$, and one obtains:
\begin{example}[Solution to Example \ref{example} for $t>2$]\label{example2}
According to \eqref{eq:solution2},
the corresponding solutions to Example \eqref{example} for $t>2$ is given by 
\[
u(x,t)=
\begin{cases}
0, & x\le 0,\\[8pt]
\dfrac{2x}{t-2}, 
& 0<x<(1-\lambda)\left(\frac t2-1\right)^2,\\[10pt]
\dfrac{\left(\frac t2-1\right)(x+1-\lambda)}{1+\left(\frac t2-1\right)^2},
& (1-\lambda)\left(\frac t2-1\right)^2<x<1+(2-\lambda)\left(\frac t2-1\right)^2,\\[12pt]
(2-\lambda)\left(\frac t2-1\right), 
& x\ge 1+(2-\lambda)\left(\frac t2-1\right)^2,
\end{cases}
\]
and
\[
\rho(x,t)=
\begin{cases}
\dfrac{1}{1+\left(\frac t2-1\right)^2},
& (1-\lambda)\left(\frac t2-1\right)^2<x<1+(2-\lambda)\left(\frac t2-1\right)^2,\\[10pt]
0, & \text{otherwise}.
\end{cases}
\]
Moreover, for $t>2$, we have
\[
u_x(x,t)=
\begin{cases}
0, & x\le 0,\\[6pt]
\dfrac{2}{t-2}, 
& 0<x<(1-\lambda)\left(\dfrac{t}{2}-1\right)^2,\\[10pt]
\dfrac{\dfrac{t}{2}-1}{1+\left(\dfrac{t}{2}-1\right)^2}, 
& (1-\lambda)\left(\dfrac{t}{2}-1\right)^2<x<1+(2-\lambda)\left(\dfrac{t}{2}-1\right)^2,\\[12pt]
0, & x\ge 1+(2-\lambda)\left(\dfrac{t}{2}-1\right)^2.
\end{cases}
\]
For \(t>2\), if \(0\le \lambda<1\), an outgoing cusp occurs at \((x,t)=(0,2)\).
As \(t\to 2^+\), the interval
\[
0<x<(1-\lambda)\left(\frac{t}{2}-1\right)^2
\]
opens up from \(x=0\), and
\[
u_x(x,t)=\frac{2}{t-2}\to +\infty.
\]
Thus characteristics emanate from \((0,2)\), forming an outgoing cusp.
If \(\lambda=1\), this interval disappears, and the outgoing cusp does not occur.

\end{example}

\subsection{Existence}\label{subsec:existence}
In this subsection, we show that the formula \eqref{eq:solution2} defines global $\lambda$-dissipative solutions to the 2HS system. We have the following existence theorem:
\begin{theorem}[Existence]\label{thm:mainexistence}
Let $(\bar{u},\bar{\rho},\bar{\mu},\bar{\nu})\in \mathcal{D}_{\lambda}$ satisfy \eqref{eq:initialdecay}, and $(u(t),\rho(t),\mu(t),\nu(t))$ be defined by \eqref{eq:solution2}. Then $(u(t),\rho(t),\mu(t),\nu(t))$ is a global-in-time $\lambda$-dissipative solution to the 2HS system in the sense of Definition~\ref{def:weak} with initial datum $(\bar{u},\bar{\rho},\bar{\mu},\bar{\nu})$. Moreover, the functions $u$, $\rho$ and energy measures $\mu$, $\nu$ satisfy all the properties in Theorem \ref{thm:introduce}.
\end{theorem}

\begin{proof}
Theorem~\ref{thm:measure} gives conditions (i), (ii), (v), and (vi) in Definition~\ref{def:weak}. We verify the weak equations and the two measure balance laws.

Let
\[
U(\alpha,t):=u(y(\alpha,t),t)=y_t(\alpha,t),
\qquad
q(\alpha,t):=y_\alpha(\alpha,t).
\]
For a.e. $t>0$, differentiation of \eqref{eq:solution2}$_1$ gives
\[
U_t(\alpha,t)
=\frac12\left[\int_{-\infty}^{\alpha}-\kappa\int_{\mathbb R}\right]
f(\eta)[1-\lambda\mathbf1_{J_t}(\eta)]\di\eta.
\]
Moreover, differentiating $y_t=U$ in the label variable gives
$q_t=U_\alpha$ in the sense of distributions, hence a.e. in
$(\alpha,t)$. At every time outside the countable set $T_s$, the measure
$\mu(t)$ is absolutely continuous. Since
\[
\mu(t)=y(\cdot,t)\#[f(1-\lambda\mathbf1_{J_t})\di\alpha],
\]
the monotonicity of $y(\cdot,t)$ and the area formula imply
\begin{equation}\label{eq:calF}
[U_t(\alpha,t)-F(y(\alpha,t),t)]q(\alpha,t)=0
\qquad\text{for a.e. }(\alpha,t)\in\mathbb R\times(0,\infty).
\end{equation}
Indeed, whenever $q(\alpha,t)>0$, the label $\alpha$ is a Lebesgue point of
all relevant quantities for a.e. $(\alpha,t)$ and the cumulative mass of
$\mu(t)$ to the left of $y(\alpha,t)$ is the corresponding label integral;
when $q=0$, the displayed identity is automatic.

We now verify the first weak equation without taking traces of $u_t$ along a
single characteristic. By the monotone area formula,
\[
\begin{aligned}
\int_0^\infty\!\!\int_{\mathbb R}
\left[u\phi_t+\frac{u^2}{2}\phi_x+F\phi\right]\di x\di t=\int_0^\infty\!\!\int_{\mathbb R}
\left[Uq\phi_t(y,t)+\frac{U^2}{2}q\phi_x(y,t)
+F(y,t)q\phi(y,t)\right]\di\alpha\di t.
\end{aligned}
\]
Using $y_t=U$, $q_t=U_\alpha$, and \eqref{eq:calF}, the integrand on the
right equals
\[
\frac{\di}{\di t}[Uq\phi(y,t)]
-\frac12\frac{\partial}{\partial\alpha}[U^2\phi(y,t)]
\]
in the sense of distributions. The label boundary term vanishes because
$\phi$ has compact spatial support and $y(\cdot,t)$ is surjective. Integrating
in time and using
$U(\alpha,0)=\bar u(\bar x(\alpha))$ and
$q(\alpha,0)=\bar x'(\alpha)$ gives
\[
\int_0^\infty\!\!\int_{\mathbb R}
\left[u\phi_t+\frac{u^2}{2}\phi_x+F\phi\right]\di x\di t
=-\int_{\mathbb R}\bar u(x)\phi(x,0)\di x.
\]
Since $u(\cdot,t)$ is absolutely continuous,
$\int (u^2/2)\phi_x=-\int uu_x\phi$, and \eqref{eq:weakformula} follows.

Next, the identity \eqref{eq:rholambdat} and
$y_t=u(y,t)$ give
\[
\begin{aligned}
\int_0^\infty\int_{\mathbb R}
\rho(\phi_t+u\phi_x)\di x\di t
=\int_0^\infty\int_{\mathbb R}
\bar\rho(\bar x(\alpha))\bar x'(\alpha)
\frac{\di}{\di t}\phi(y(\alpha,t),t)\di\alpha\di t=-\int_{\mathbb R}\bar\rho(x)\phi(x,0)\di x,
\end{aligned}
\]
which is \eqref{eq:weakformula1}.

The transport identity for $\nu$ follows directly from
$\nu(t)=y(\cdot,t)\#[f\di\alpha]$:
\[
\begin{aligned}
\int_0^\infty\int_{\mathbb R}
(\phi_t+u\phi_x)\di\nu(t)\di t
=\int_{\mathbb R}f(\alpha)
\int_0^\infty\frac{\di}{\di t}
\phi(y(\alpha,t),t)\di t\di\alpha=-\int_{\mathbb R}\phi(x,0)\di\bar\nu.
\end{aligned}
\]
Thus \eqref{eq:energyconservenu} holds.

It remains to identify the defect measure in the balance law for $\mu$. Define the first breaking time by
\[
\tau(\alpha):=
\begin{cases}
0,&\alpha\in A_0^L,\\[1mm]
-\dfrac{2}{\bar u_x(\bar x(\alpha))},
&\alpha\in B_0^L,\quad
\bar u_x(\bar x(\alpha))<0,\quad
\bar\rho(\bar x(\alpha))=0,\\[2mm]
+\infty,&\text{otherwise}.
\end{cases}
\]
Then $J_t=\{\alpha:\tau(\alpha)\leq t\}$, with the convention
$J_0=A_0^L$. Define the finite nonnegative Radon measure
\[
\mathfrak D
:=\lambda
\left(\alpha\mapsto
(y(\alpha,\tau(\alpha)),\tau(\alpha))\right)\#
\left[f(\alpha)\mathbf1_{\{0<\tau(\alpha)<\infty\}}\di\alpha\right]
\]
on $\mathbb R\times(0,\infty)$. For
$\psi_\alpha(t)=\phi(y(\alpha,t),t)$, Fubini's theorem and the elementary one-dimensional integration-by-parts formula
\[
\int_0^\infty\psi_\alpha'(t)
[1-\lambda\mathbf1_{\{\tau(\alpha)\leq t\}}]\di t
= -[1-\lambda\mathbf1_{A_0^L}(\alpha)]\psi_\alpha(0)
+\lambda\mathbf1_{\{0<\tau(\alpha)<\infty\}}
\psi_\alpha(\tau(\alpha))
\]
give
\begin{equation}\label{eq:conservation of mass}
\begin{aligned}
\int_0^\infty\int_{\mathbb R}
(\phi_t+u\phi_x)\di\mu(t)\di t
+\int_{\mathbb R}\phi(x,0)\di\bar\mu
&=\lambda\int_{\{0<\tau(\alpha)<\infty\}}
\phi(y(\alpha,\tau(\alpha)),\tau(\alpha))
f(\alpha)\di\alpha\\
&
=\int_{\mathbb R\times(0,\infty)}
\phi(x,t)\di\mathfrak D(x,t).
\end{aligned}
\end{equation}
Here we used
$\bar\mu=\bar x\#[f(1-\lambda\mathbf1_{A_0^L})\di\alpha]$, which is equivalent to \eqref{eq:initialdecay}. This proves \eqref{eq:energydissipation}. Since $\mathfrak D$ is nonnegative, \eqref{eq:fourth} follows. For $\lambda=0$, the measure $\mathfrak D$ vanishes.
\end{proof}

\begin{remark}\label{rmk:testfunction}
The cutoff argument in the proof shows that the measure identities
\eqref{eq:energyconservenu}--\eqref{eq:fourth} remain valid for bounded
$C^1$ test functions with bounded derivatives whose time support is contained
in $[0,M)$ for some $M>0$. For the Lebesgue-space weak equations
\eqref{eq:weakformula} and \eqref{eq:weakformula1}, the spatial compact-support
assumption is retained unless additional spatial integrability is imposed.
\end{remark}

\section{Uniqueness}\label{sec:uniqueness}
In this section, we prove uniqueness of the $\lambda$-dissipative solutions to the 2HS system. 
We show that every solution in the sense of Definition~\ref{def:weak} coincides with the solution given by \eqref{eq:solution2} for the same initial data.

The proof is divided into three parts. In Subsection~\ref{subsec:charac}, we construct the canonical characteristics associated with the transported measure $\tilde\nu$. In Subsection~\ref{subsec:singularset}, we determine the first breaking time and derive the ordinary differential equations on every component on which $\tilde y_\alpha>0$. In Subsection~\ref{subsec:accumultation}, we identify the Eulerian outgoing region, recover the Lagrangian representation of $\tilde\mu$, and conclude uniqueness.

\subsection{Construction of characteristics}\label{subsec:charac}
We first record the characteristic lemma used below.

\begin{lemma}\label{lmm:dafermos}
Let $u\in C(\mathbb R\times[0,T])$ satisfy
$u_x\in L^\infty(0,T;L^2(\mathbb R))$ and be a weak solution of
\begin{equation}\label{eq:equF}
u_t+\left(\frac{u^2}{2}\right)_x=F(x,t)
\end{equation}
on $\mathbb R\times(0,T)$. Assume that $F\in L^\infty(\mathbb R\times(0,T))$ and that, for a.e. $t\in(0,T)$, $F(\cdot,t)$ has a continuous representative. If $z\in C^1([a,b])$, $[a,b]\subset[0,T]$, satisfies
\[
\dot z(t)=u(z(t),t),
\]
then $v(t):=u(z(t),t)$ is absolutely continuous and
\[
\dot v(t)=F(z(t),t)
\]
for a.e. $t\in(a,b)$. In particular, if $F$ is bounded, then $v$ is Lipschitz continuous and $z'$ is Lipschitz continuous.
\end{lemma}

\begin{proof}
Let $\varrho\in C_c^\infty((-1,1))$ be nonnegative with
$\int_{\mathbb R}\varrho=1$, and set
\[
\varrho_\e(x)=\e^{-1}\varrho(x/\e),
\qquad
v_\e(t)=\int_{\mathbb R}u(x,t)\varrho_\e(x-z(t))\di x.
\]
Testing \eqref{eq:equF} against a smooth time cutoff multiplied by
$\varrho_\e(x-z(t))$ gives, in the sense of distributions in time,
\[
\dot v_\e(t)
=\int_{\mathbb R}F(x,t)\varrho_\e(x-z(t))\di x
+\int_{\mathbb R}\left[\frac{u^2(x,t)}2-u(z(t),t)u(x,t)\right]
\varrho_\e'(x-z(t))\di x.
\]
Since $\int\varrho_\e'=0$, the last integral equals
\[
\frac12\int_{\mathbb R}[u(x,t)-u(z(t),t)]^2
\varrho_\e'(x-z(t))\di x.
\]
The Cauchy-Schwarz inequality gives
\[
|u(x,t)-u(z(t),t)|^2
\leq |x-z(t)|\int_{[x,z(t)]}|u_x(\xi,t)|^2\di\xi.
\]
It follows that
\[
\left|\int_{\mathbb R}[u(x,t)-u(z(t),t)]^2
\varrho_\e'(x-z(t))\di x\right|
\leq C\int_{z(t)-\e}^{z(t)+\e}|u_x(x,t)|^2\di x.
\]
The right-hand side tends to zero for a.e. $t$ and is dominated by
$C\|u_x(\cdot,t)\|_{L^2}^2\in L^1(0,T)$. Moreover,
$v_\e(t)\to u(z(t),t)$ uniformly on compact time intervals, and the averaged source converges to $F(z(t),t)$ for a.e. $t$ by the assumed continuity of $F(\cdot,t)$. Passing to the limit in the integrated identity yields
\[
u(z(t),t)-u(z(s),s)=\int_s^tF(z(\tau),\tau)\di\tau,
\]
which proves the result.
\end{proof}

Let $(u,\rho,\mu,\nu)$ be a $\lambda$-dissipative solution and write
\[
\di\nu_{ac}(t)=g(x,t)\di x.
\]
Since $\mu(t)\ll\nu(t)$ and $\di\mu_{ac}(t)=(u_x^2+\rho^2)\di x$, one has, for a.e. $t>0$,
\[
0\leq u_x^2(x,t)+\rho^2(x,t)\leq g(x,t),
\qquad
\|g(\cdot,t)\|_{L^1}=\bar\nu(\mathbb R).
\]
We next establish finite propagation for the transported measure $\nu$.

\begin{lemma}\label{lmm:keylemma1}
Let $(u,\rho,\mu,\nu)$ be a $\lambda$-dissipative solution in the sense of Definition~\ref{def:weak}. Let $y\in\mathbb R$, $0<\tau<t$, and $\e_0>0$, and set
\[
a_\pm:=u(y,\tau)\pm\e_0.
\]
If $t-\tau>0$ is sufficiently small, then, whenever $\nu(\tau)$ and $\nu(t)$ are absolutely continuous,
\begin{equation}\label{eq:claim31}
\int_{-\infty}^{y+a_-(t-\tau)}g(x,t)\di x
\leq\int_{-\infty}^{y}g(x,\tau)\di x
\leq\int_{-\infty}^{y+a_+(t-\tau)}g(x,t)\di x.
\end{equation}
For arbitrary $\tau$ and $t$, the corresponding measure inequalities are
\begin{equation}\label{eq:claim3}
\nu(t)((-\infty,y+a_-(t-\tau)])
\leq\nu(\tau)((-\infty,y]),
\qquad
\nu(\tau)((-\infty,y])
\leq\nu(t)((-\infty,y+a_+(t-\tau)]).
\end{equation}
Moreover, for every $T>0$, every $0\leq\tau<t\leq T$, and every $C_T$ satisfying
$\|u\|_{C_b(\mathbb R\times[0,T])}\leq C_T$, one has, at absolutely continuous times,
\begin{equation}\label{eq:claim1}
\int_{-\infty}^{y-C_T(t-\tau)}g(x,t)\di x
\leq\int_{-\infty}^{y}g(x,\tau)\di x
\leq\int_{-\infty}^{y+C_T(t-\tau)}g(x,t)\di x,
\end{equation}
and, at arbitrary times,
\begin{equation}\label{eq:claim}
\nu(t)((-\infty,y-C_T(t-\tau)])
\leq\nu(\tau)((-\infty,y]),
\qquad
\nu(\tau)((-\infty,y])
\leq\nu(t)((-\infty,y+C_T(t-\tau)]).
\end{equation}
\end{lemma}

\begin{proof}
Choose a nonincreasing function $\theta\in C^\infty(\mathbb R)$ satisfying
\begin{equation}\label{eq:theta}
0\leq\theta\leq1,
\qquad
\theta(r)=1\quad(r\leq0),
\qquad
\theta(r)=0\quad(r\geq1),
\qquad
\theta'\leq0,
\end{equation}
and put $\theta_\e(r)=\theta(r/\e)$. We first multiply the test functions below by a spatial cutoff $\zeta_R$ which equals one on $[-R,R]$. Letting $R\to\infty$ produces no error because $u$ is bounded on finite time intervals and $\nu(s)(\mathbb R)=\bar\nu(\mathbb R)$.

For $a\in\mathbb R$ and $\delta>0$ sufficiently small, choose
$\chi_\delta\in C_c^\infty((0,\infty))$ such that
\[
0\leq\chi_\delta\leq1,\qquad
\chi_\delta=1\quad\text{on }[\tau,t],
\qquad
\operatorname{supp}\chi_\delta
\subset(\tau-\delta,t+\delta),
\]
and
\[
\chi_\delta'\geq0\quad\text{on }(\tau-\delta,\tau),
\qquad
\chi_\delta'\leq0\quad\text{on }(t,t+\delta).
\]
Let $\zeta_R\in C_c^\infty(\mathbb R)$ satisfy
\[
0\leq\zeta_R\leq1,\qquad
\zeta_R=1\quad\text{on }[-R,R],
\qquad
|\zeta_R'|\leq \frac{C}{R}.
\]
We use
\[
\phi_{\e,\delta,R}(x,s)
=
\theta_\e(x-y-a(s-\tau))
\chi_\delta(s)\zeta_R(x)
\]
as a test function in \eqref{eq:energyconservenu}. Since
\[
\begin{aligned}
(\partial_s+u\partial_x)\phi_{\e,\delta,R}
={}&
\chi_\delta'(s)
\theta_\e(x-y-a(s-\tau))\zeta_R(x)
\\
&+\chi_\delta(s)[u(x,s)-a]
\theta_\e'(x-y-a(s-\tau))\zeta_R(x)
\\
&+\chi_\delta(s)u(x,s)
\theta_\e(x-y-a(s-\tau))\zeta_R'(x),
\end{aligned}
\]
we obtain
\[
\begin{aligned}
0={}&
\int_0^\infty\int_{\mathbb R}
\chi_\delta'(s)
\theta_\e(x-y-a(s-\tau))\zeta_R(x)
\di\nu(s)\di s
\\
&+
\int_0^\infty\int_{\mathbb R}
\chi_\delta(s)[u(x,s)-a]
\theta_\e'(x-y-a(s-\tau))\zeta_R(x)
\di\nu(s)\di s
\\
&+
\int_0^\infty\int_{\mathbb R}
\chi_\delta(s)u(x,s)
\theta_\e(x-y-a(s-\tau))\zeta_R'(x)
\di\nu(s)\di s.
\end{aligned}
\]
First letting $\delta\downarrow0$ and then $R\to\infty$, using the
weak-star continuity of $\nu$, the boundedness of $u$ on finite time
intervals, and $\nu(s)(\mathbb R)=\bar\nu(\mathbb R)$, we obtain
\begin{equation}\label{eq:regularity1}
\begin{aligned}
&\int_{\mathbb R}
\theta_\e(x-y-a(t-\tau))\di\nu(t)
-\int_{\mathbb R}\theta_\e(x-y)\di\nu(\tau)
\\
&\qquad
=
\int_\tau^t\int_{\mathbb R}
[u(x,s)-a]\theta_\e'(x-y-a(s-\tau))
\di\nu(s)\di s.
\end{aligned}
\end{equation}

For $a=a_+$, continuity of $u$ implies that $u-a_+\leq0$ on the support of $\theta_\e'$ when $t-\tau$ and $\e$ are sufficiently small. Since $\theta_\e'\leq0$, \eqref{eq:regularity1} gives
\begin{equation}\label{eq:regu}
\int_{\mathbb R}\theta_\e(x-y-a_+(t-\tau))\di\nu(t)
\geq\int_{\mathbb R}\theta_\e(x-y)\di\nu(\tau).
\end{equation}
For $a=a_-$, the opposite sign gives
\begin{equation}\label{eq:regu1}
\int_{\mathbb R}\theta_\e(x-y-a_-(t-\tau))\di\nu(t)
\leq\int_{\mathbb R}\theta_\e(x-y)\di\nu(\tau).
\end{equation}
Letting $\e\downarrow0$ proves \eqref{eq:claim3}; if the endpoint measures are absolutely continuous, this is \eqref{eq:claim31}. Taking $a_+=C_T$ and $a_-=-C_T$ removes the smallness restriction and gives \eqref{eq:claim} and \eqref{eq:claim1}.
\end{proof}

\begin{theorem}\label{thm:uniqueness1}
Let $0\leq\lambda<1$, and let
$(\tilde u,\tilde\rho,\tilde\mu,\tilde\nu)$ be a $\lambda$-dissipative solution in the sense of Definition~\ref{def:weak} with initial datum
$(\bar u,\bar\rho,\bar\mu,\bar\nu)\in\mathcal D_\lambda$ satisfying \eqref{eq:initialdecay}. Then there exists a nondecreasing, surjective characteristic $\tilde y(\alpha,t)$ satisfying
\begin{align}\label{eq:flowmap}
\frac{\partial}{\partial t}\tilde y(\alpha,t)
=\tilde u(\tilde y(\alpha,t),t),
\qquad
\tilde y(\alpha,0)=\bar x(\alpha),
\end{align}
\begin{equation}\label{eq:Liptildey}
|\tilde y(\beta,t)-\tilde y(\eta,t)|
\leq e^{t/2}|\beta-\eta|,
\qquad \beta,\eta\in\mathbb R,
\end{equation}
and
\begin{align}\label{eq:energyconserved}
\tilde\nu(t)((-\infty,\tilde y(\alpha,t)))
\leq\alpha-\bar x(\alpha)
\leq\tilde\nu(t)((-\infty,\tilde y(\alpha,t)]).
\end{align}
Moreover,
\[
\tilde\nu(t)=\tilde y(\cdot,t)\#[f(\alpha)\di\alpha].
\]
For $\lambda=0$, the solution coincides with the solution given by \eqref{eq:solution2}; hence the conservative solution is unique.
\end{theorem}

\begin{proof}
We divide the proof into four steps.

\textbf{Step 1. Construction and regularity of the generalized inverse.}
For $\beta\in\mathbb R$ and $t\geq0$, define $x_1(\beta,t)$ by
\begin{align}\label{eq:x1betat}
x_1(\beta,t)+\tilde\nu(t)((-\infty,x_1(\beta,t)))
\leq\beta
\leq x_1(\beta,t)+\tilde\nu(t)((-\infty,x_1(\beta,t)]).
\end{align}
This generalized inverse is uniquely defined and is nondecreasing in $\beta$. If $\eta<\beta$, then
\[
\begin{aligned}
\beta-\eta
\geq x_1(\beta,t)-x_1(\eta,t)
+\tilde\nu(t)((x_1(\eta,t),x_1(\beta,t)))\geq x_1(\beta,t)-x_1(\eta,t),
\end{aligned}
\]
so
\begin{equation}\label{eq:Lipx1}
|x_1(\beta,t)-x_1(\eta,t)|\leq|\beta-\eta|.
\end{equation}
Since $\tilde\mu(t)\leq\tilde\nu(t)$ and $\tilde u(\cdot,t)$ is absolutely continuous,
\begin{equation}\label{eq:vxbeta}
\begin{aligned}
|\tilde u(x_1(\beta,t),t)-\tilde u(x_1(\eta,t),t)|&\leq\int_{x_1(\eta,t)}^{x_1(\beta,t)}|\tilde u_x(x,t)|\di x\\
&\quad\leq\frac12\left[x_1(\beta,t)-x_1(\eta,t)
+\tilde\mu(t)((x_1(\eta,t),x_1(\beta,t)))\right]
\leq\frac12(\beta-\eta).
\end{aligned}
\end{equation}

Let $T>0$ and $C_T\geq\|\tilde u\|_{C_b(\mathbb R\times[0,T])}$. At times $s,t\notin T_s$, the measures $\tilde\nu(s)$ and $\tilde\nu(t)$ are absolutely continuous, and \eqref{eq:claim1} together with \eqref{eq:x1betat} gives
\[
|x_1(\beta,t)-x_1(\beta,s)|\leq C_T|t-s|.
\]
We record the stability of the generalized inverse needed to extend this estimate. Suppose that
$\nu_n\overset{*}{\rightharpoonup}\nu$, that
$\nu_n(\mathbb R)=\nu(\mathbb R)$, and that $x_n$ and $x$ are defined by
\eqref{eq:x1betat} for the same value of $\beta$. Then $x_n\to x$.
Indeed, if $M=\nu(\mathbb R)$, then
$\beta-M\leq x_n\leq\beta$. Let a subsequence, not relabeled, satisfy
$x_n\to z$. For every continuity point $a<z$ of the distribution function of
$\nu$, one has $a<x_n$ for all sufficiently large $n$, and hence
\[
\beta\geq x_n+\nu_n(( -\infty,x_n))
\geq x_n+\nu_n(( -\infty,a]).
\]
Passing to the limit and then taking continuity points $a\uparrow z$ gives
\[
z+\nu(( -\infty,z))\leq\beta.
\]
Similarly, for every continuity point $b>z$, one has $x_n<b$ for all sufficiently
large $n$, and
\[
\beta\leq x_n+\nu_n(( -\infty,x_n])
\leq x_n+\nu_n(( -\infty,b)).
\]
Passing to the limit and then taking continuity points $b\downarrow z$ gives
\[
\beta\leq z+\nu((-\infty,z]).
\]
Thus $z$ satisfies the defining inequalities \eqref{eq:x1betat}; their solution is
unique because the identity term makes the corresponding graph strictly
increasing. Hence $z=x$, and the whole sequence converges.

Since $\tilde\nu$ is weak-star continuous and
$[0,T]\setminus T_s$ is dense, choose sequences of absolutely continuous times
$s_n\to s$ and $t_n\to t$. Applying the preceding stability result and passing to
the limit in
\[
|x_1(\beta,t_n)-x_1(\beta,s_n)|\leq C_T|t_n-s_n|
\]
yields the same estimate for arbitrary $s,t\in[0,T]$. Thus $x_1$ is jointly
continuous and locally Lipschitz in time.

\textbf{Step 2. Construction of the canonical flow.}
Consider
\begin{equation}\label{eq:ode1}
\theta(\alpha,t)=\alpha+
\int_0^t\tilde u(x_1(\theta(\alpha,s),s),s)\di s.
\end{equation}
By \eqref{eq:vxbeta}, this equation has a unique global solution. Uniqueness of the scalar ODE preserves the order of the initial labels, and Gr\"{o}nwall's inequality gives
\[
|\theta(\beta,t)-\theta(\eta,t)|\leq e^{t/2}|\beta-\eta|.
\]
Define
\begin{align}\label{eq:charac}
\tilde y(\alpha,t):=x_1(\theta(\alpha,t),t).
\end{align}
Then $\tilde y(\cdot,t)$ is nondecreasing, \eqref{eq:Liptildey} holds, and $\tilde y(\alpha,0)=\bar x(\alpha)$. Moreover,
\begin{align}\label{eq:combine}
\tilde y(\alpha,t)+\tilde\nu(t)((-\infty,\tilde y(\alpha,t)))
\leq\alpha+\int_0^t\tilde u(\tilde y(\alpha,s),s)\di s
\leq\tilde y(\alpha,t)+\tilde\nu(t)((-\infty,\tilde y(\alpha,t)]).
\end{align}

We prove \eqref{eq:flowmap}. Fix $\alpha$ and a time $\tau\notin T_s$ at which $\tilde y(\alpha,\cdot)$ is differentiable. Suppose first that, for some $\e_0>0$,
\begin{align}\label{eq:contradiction}
\partial_t\tilde y(\alpha,\tau)
\leq\tilde u(\tilde y(\alpha,\tau),\tau)-2\e_0.
\end{align}
Set $y=\tilde y(\alpha,\tau)$ and $a_-=\tilde u(y,\tau)-\e_0$. Choose a sequence $t_n\downarrow\tau$ with $t_n\notin T_s$. For $n$ large,
$\tilde y(\alpha,t_n)<y+a_-(t_n-\tau)$. Since the measures at $\tau$ and $t_n$ are absolutely continuous, \eqref{eq:claim31} and \eqref{eq:combine} give
\begin{equation}\label{eq:two}
\theta(\alpha,t_n)
<\theta(\alpha,\tau)
+[\tilde u(\tilde y(\alpha,\tau),\tau)-\e_0](t_n-\tau).
\end{equation}
Dividing by $t_n-\tau$ and using \eqref{eq:ode1} gives a contradiction. The alternative
\begin{align}\label{eq:contradiction2}
\partial_t\tilde y(\alpha,\tau)
\geq\tilde u(\tilde y(\alpha,\tau),\tau)+2\e_0
\end{align}
is ruled out in the same way by the second inequality in \eqref{eq:claim31}. Therefore \eqref{eq:flowmap} holds for a.e. $t$. Since both sides are locally integrable in time,
\begin{equation}\label{eq:integraltildey}
\tilde y(\alpha,t)=\bar x(\alpha)+
\int_0^t\tilde u(\tilde y(\alpha,s),s)\di s.
\end{equation}
Substituting this identity into \eqref{eq:combine} yields \eqref{eq:energyconserved}. Because
\[
\alpha-\bar x(\alpha)=\int_{-\infty}^{\alpha}f(\eta)\di\eta,
\]
\eqref{eq:energyconserved} is precisely the quantile characterization of
\[
\tilde\nu(t)=\tilde y(\cdot,t)\#[f\di\alpha].
\]
Moreover, if $M=\bar\nu(\mathbb R)$, then \eqref{eq:x1betat} gives
$\beta-M\leq x_1(\beta,t)\leq\beta$, while \eqref{eq:ode1} gives
$|\theta(\alpha,t)-\alpha|\leq t\|\tilde u\|_{C_b(\mathbb R\times[0,t])}$. Hence
$\tilde y(\alpha,t)\to\pm\infty$ as $\alpha\to\pm\infty$, and the continuous nondecreasing map $\tilde y(\cdot,t)$ is surjective.

\textbf{Step 3. Transport of the signed density $\tilde\rho$.}
Fix $\alpha_1<\alpha_2$, write $y_i(t)=\tilde y(\alpha_i,t)$, and define
\begin{equation}\label{eq:massbetweentwo}
\Theta(\alpha_1,\alpha_2,t)
:=\int_{y_1(t)}^{y_2(t)}\tilde\rho(x,t)\di x.
\end{equation}
We claim that
\begin{equation}\label{eq:mass}
\Theta(\alpha_1,\alpha_2,t)
=\Theta(\alpha_1,\alpha_2,0),
\qquad t\geq0.
\end{equation}
Let $\psi\in C^\infty(\mathbb R)$ be nonincreasing, with $\psi=1$ on
$(-\infty,0]$ and $\psi=0$ on $[1,\infty)$, and set
\[
\chi_\e(x,t)
=\psi\left(\frac{x-y_2(t)}\e\right)
-\psi\left(\frac{x-y_1(t)}\e\right).
\]
For $\eta\in C_c^\infty([0,\infty))$, insert
$\phi_\e(x,t)=\eta(t)\chi_\e(x,t)$ into \eqref{eq:weakformula1}. This gives
\begin{equation}\label{eq:testrho}
\int_0^\infty\!\!\int_{\mathbb R}
\tilde\rho\left[\eta'(t)\chi_\e
+\eta(t)(\chi_{\e,t}+\tilde u\chi_{\e,x})\right]\di x\di t
=-\eta(0)\int_{\mathbb R}\bar\rho(x)\chi_\e(x,0)\di x.
\end{equation}
Using $\dot y_i(t)=\tilde u(y_i(t),t)$, the local $1/2$-H\"older continuity of
$\tilde u$, and Cauchy--Schwarz, one obtains, on every bounded time interval,
\begin{equation}\label{eq:testrhoterm}
\left|\int_{\mathbb R}\tilde\rho(x,t)
\frac1\e\psi'\left(\frac{x-y_i(t)}\e\right)
[\tilde u(x,t)-\tilde u(y_i(t),t)]\di x\right|
\leq C_T\|\psi'\|_{L^2}
\|\tilde\rho(\cdot,t)\|_{L^2(y_i(t),y_i(t)+\e)}.
\end{equation}
The right-hand side tends to zero for a.e. $t$ and is bounded in
$L^1(0,T)$ uniformly in $\e$. Dominated convergence in \eqref{eq:testrho}
therefore shows that $\Theta(\alpha_1,\alpha_2,\cdot)$ has zero distributional
derivative.

To complete the proof of \eqref{eq:mass}, we show that $\Theta$ is continuous. If $t_n\to t$, then, with
$I_n=(y_1(t_n),y_2(t_n))$ and $I=(y_1(t),y_2(t))$,
\[
\begin{aligned}
|\Theta(\alpha_1,\alpha_2,t_n)-\Theta(\alpha_1,\alpha_2,t)|
\leq |\langle\tilde\rho(\cdot,t_n)-\tilde\rho(\cdot,t),\mathbf1_I\rangle|+\|\tilde\rho(\cdot,t_n)\|_{L^2}
\|\mathbf1_{I_n}-\mathbf1_I\|_{L^2},
\end{aligned}
\]
and both terms tend to zero by weak continuity in $L^2$, the uniform $L^2$
bound, and the continuity of $y_i$. Hence \eqref{eq:mass} holds for every
$t\geq0$.

Apply the monotone change-of-variables formula at time $t$ and at time zero.
For every pair of rational numbers $\alpha_1<\alpha_2$, \eqref{eq:mass} gives
\[
\int_{\alpha_1}^{\alpha_2}
\tilde\rho(\tilde y(\alpha,t),t)\tilde y_\alpha(\alpha,t)\di\alpha
=\int_{\alpha_1}^{\alpha_2}
\bar\rho(\bar x(\alpha))\bar x'(\alpha)\di\alpha.
\]
Both integrands are locally integrable in the label variable. Since the equality
holds simultaneously on all intervals with rational endpoints, for each fixed
$t\geq0$ it follows that
\begin{equation}\label{eq:tilderhoidentity}
\tilde\rho(\tilde y(\alpha,t),t)\tilde y_\alpha(\alpha,t)
=\bar\rho(\bar x(\alpha))\bar x'(\alpha)
\end{equation}
for a.e. $\alpha\in\mathbb R$. Equivalently, \eqref{eq:tilderhoidentity}
holds for a.e. $(\alpha,t)\in\mathbb R\times(0,\infty)$; by Fubini, it holds
for a.e. $t$ along a.e. fixed label. This spacetime interpretation is the one
used below in deriving the ODE system.

The corresponding push-forward identity is understood locally, or equivalently
in $\mathcal D'(\mathbb R)$: for every $\phi\in C_c^\infty(\mathbb R)$ and every
$t\geq0$,
\[
\int_{\mathbb R}\phi(x)\tilde\rho(x,t)\di x
=\int_{\mathbb R}\phi(\tilde y(\alpha,t))
(\bar\rho\circ\bar x)(\alpha)\bar x'(\alpha)\di\alpha.
\]
The integral on the right is finite because the inverse image of a compact set
under $\tilde y(\cdot,t)$ is bounded and
$(\bar\rho\circ\bar x)\bar x'\in L^1_{\rm loc}(\mathbb R)$. In this precise
sense,
\[
\tilde\rho(x,t)\di x
=\tilde y(\cdot,t)\#[(\bar\rho\circ\bar x)\bar x'\di\alpha]
\]
for every $t\geq0$.

\textbf{Step 4. The conservative case.}
Assume $\lambda=0$. Then $\tilde\mu=\tilde\nu$. At a.e. time, $\tilde\mu(t)$ is absolutely continuous, so Lemma~\ref{lmm:dafermos}, \eqref{eq:energyconserved}, and the push-forward formula for $\tilde\nu$ give
\begin{equation}\label{eq:accelaration}
\begin{aligned}
\frac{\partial^2}{\partial t^2}\tilde y(\alpha,t)
=\frac12\left[\int_{-\infty}^{\tilde y(\alpha,t)}
-\kappa\int_{\mathbb R}\right]\di\tilde\mu(t)=\frac12\left[\int_{-\infty}^{\alpha}
-\kappa\int_{\mathbb R}\right]f(\eta)\di\eta.
\end{aligned}
\end{equation}
Together with $\tilde y(\alpha,0)=\bar x(\alpha)$ and
$\partial_t\tilde y(\alpha,0)=\bar u(\bar x(\alpha))$, this gives \eqref{eq:lambda_dissipation} with $\lambda=0$. Hence $\tilde y=y$, and \eqref{eq:flowmap}, the two push-forward identities, and \eqref{eq:tilderhoidentity} give
$(\tilde u,\tilde\rho,\tilde\mu,\tilde\nu)=(u,\rho,\mu,\nu)$.
\end{proof}

In the remainder of this section, assume $0<\lambda<1$. Set
\[
\tilde U(\alpha,t):=\tilde u(\tilde y(\alpha,t),t),
\qquad
\tilde F(x,t):=\frac12\left[(1-\kappa)\int_{-\infty}^x
-\kappa\int_x^{+\infty}\right]\di\tilde\mu(t).
\]
By Lemma~\ref{lmm:dafermos},
\begin{equation}\label{eq:dtu0}
\frac{\di}{\di t}\tilde U(\alpha,t)
=\tilde F(\tilde y(\alpha,t),t),
\qquad
\tilde U(\alpha,0)=\bar u(\bar x(\alpha)),
\end{equation}
for a.e. $t>0$. At every time at which $\tilde\mu(t)$ is absolutely continuous,
\begin{equation}\label{eq:dtu}
\frac{\di}{\di t}\tilde U(\alpha,t)
=\frac12\left[D_\kappa^{-1}(\tilde u_x^2+\tilde\rho^2)\right](\tilde y(\alpha,t),t).
\end{equation}
To identify $\tilde y$ with \eqref{eq:lambda_dissipation}, it is enough to prove
\begin{equation}\label{eq:accelaration2}
\frac{\partial^2}{\partial t^2}\tilde y(\alpha,t)
=\frac12\left[\int_{-\infty}^{\alpha}-\kappa\int_{\mathbb R}\right]
f(\eta)[1-\lambda\mathbf1_{J_t}(\eta)]\di\eta
\end{equation}
for a.e. $t>0$. Equivalently, it is enough to establish
\begin{align}\label{eq:energyconserved11}
\tilde\mu(t)((-\infty,\tilde y(\alpha,t)))
\leq\int_{-\infty}^{\alpha}f(\eta)[1-\lambda\mathbf1_{J_t}(\eta)]\di\eta
\leq\tilde\mu(t)((-\infty,\tilde y(\alpha,t)]).
\end{align}

\subsection{Set of singularities}\label{subsec:singularset}
We next derive the equations satisfied by the spatial derivatives of the canonical
flow. For a.e. $(\alpha,t)$, set
\[
q(\alpha,t):=\tilde y_\alpha(\alpha,t),
\qquad
p(\alpha,t):=\tilde U_\alpha(\alpha,t),
\qquad
r(\alpha):=\bar\rho(\bar x(\alpha))\bar x'(\alpha).
\]
The proof of Lemma~\ref{lmm:ODEs} below constructs representatives of $q$ and
$p$ which are continuous, respectively locally Lipschitz, in time for a.e.
label $\alpha$; these representatives are used from that point onward. Whenever
$q(\alpha,t)>0$, define
\begin{equation}\label{eq:wrhoalpha}
w_\alpha(t):=\frac{p(\alpha,t)}{q(\alpha,t)}
=\tilde u_x(\tilde y(\alpha,t),t),
\qquad
\rho_\alpha(t):=\frac{r(\alpha)}{q(\alpha,t)}
=\tilde\rho(\tilde y(\alpha,t),t).
\end{equation}
For $\alpha\in B_0^L$, define
\begin{equation}\label{eq:blowuptime}
T_\alpha:=
\begin{cases}
-\dfrac{2}{\bar u_x(\bar x(\alpha))},
&\bar u_x(\bar x(\alpha))<0
\quad\text{and}\quad
\bar\rho(\bar x(\alpha))=0,\\[2mm]
+\infty,&\text{otherwise}.
\end{cases}
\end{equation}

\begin{lemma}\label{lmm:ODEs}
For a.e. $\alpha\in\mathbb R$, the functions $q(\alpha,\cdot)$ and
$p(\alpha,\cdot)$ admit representatives such that $q$ is locally absolutely
continuous, $p$ is locally Lipschitz, and $q(\alpha,t)\geq0$ for every $t\geq0$.
On every connected component of
\[
\{t>0:q(\alpha,t)>0\},
\]
the functions $w_\alpha$ and $\rho_\alpha$ are locally absolutely continuous and
\begin{equation}\label{eq:ODEs}
\left\{
\begin{aligned}
\dot w_\alpha(t)+\frac12w_\alpha^2(t)&=\frac12\rho_\alpha^2(t),\\
\dot\rho_\alpha(t)+\rho_\alpha(t)w_\alpha(t)&=0.
\end{aligned}
\right.
\end{equation}
For a.e. $\alpha\in B_0^L$, on the component whose closure contains $t=0$,
\begin{equation}\label{eq:ODEinitial}
\left\{
\begin{aligned}
w_\alpha(0)&=\bar u_x(\bar x(\alpha)),\\
\rho_\alpha(0)&=\bar\rho(\bar x(\alpha)).
\end{aligned}
\right.
\end{equation}
\end{lemma}

\begin{proof}
For each bounded time interval, the maps
$\alpha\mapsto\tilde y(\alpha,t)$ and
$\alpha\mapsto\tilde U(\alpha,t)$ are uniformly locally Lipschitz. The first
assertion follows from \eqref{eq:Liptildey}; for the second, \eqref{eq:vxbeta}
and the estimate for $\theta$ give
\[
|\tilde U(\beta,t)-\tilde U(\eta,t)|
\leq\frac12e^{t/2}|\beta-\eta|.
\]
Differentiating \eqref{eq:integraltildey} with respect to $\alpha$ in the sense
of distributions and applying Fubini gives, for a.e. $\alpha$,
\begin{equation}\label{eq:hp}
q(\alpha,t)=\bar x'(\alpha)+\int_0^t p(\alpha,s)\di s,
\qquad \partial_tq(\alpha,t)=p(\alpha,t).
\end{equation}
We use the first identity to define the time-continuous representative of $q$.
For every fixed $t$, it agrees a.e. in $\alpha$ with
$\tilde y_\alpha(\alpha,t)$. Intersecting the corresponding full-measure sets
at all rational times, and then using the continuity of
$t\mapsto q(\alpha,t)$, yields a single full-measure set of labels on which
$q(\alpha,t)\geq0$ for every $t\geq0$.

It remains to construct the time representative of $p$ and identify its
derivative. Fix $T>0$, choose once and for all a sequence $h_n\downarrow0$, and
put, for $h>0$,
\[
p_h(\alpha,t):=
\frac{\tilde U(\alpha+h,t)-\tilde U(\alpha,t)}{h}.
\]
For a.e. $t$, the measure $\tilde\mu(t)$ is absolutely continuous and
\eqref{eq:dtu0} yields
\[
\partial_t p_h(\alpha,t)
=\frac{1}{2h}
\int_{\tilde y(\alpha,t)}^{\tilde y(\alpha+h,t)}
(\tilde u_x^2+\tilde\rho^2)(x,t)\di x.
\]
The integrand is nonnegative. Moreover, using
$\tilde\mu\leq\tilde\nu$, the monotonicity of $\tilde y$, and
$\tilde\nu(t)=\tilde y(\cdot,t)\#[f\di\alpha]$, we obtain
\[
0\leq\partial_t p_h(\alpha,t)
\leq\frac{1}{2h}\int_\alpha^{\alpha+h}f(\eta)\di\eta
\leq\frac12
\]
for a.e. $t$, because $0\leq f\leq1$. Thus $p_h(\alpha,\cdot)$ is
$1/2$-Lipschitz, uniformly in $h$.

At every rational time $t\in[0,T]$, Lebesgue differentiation in the label
variable gives
$p_{h_n}(\alpha,t)\to\tilde U_\alpha(\alpha,t)$ for a.e. $\alpha$.
After intersecting these countably many full-measure sets, fix such an
$\alpha$. The uniform Lipschitz bound and convergence on the dense set of
rational times imply that $p_{h_n}(\alpha,\cdot)$ is uniformly Cauchy on
$[0,T]$; hence it converges uniformly to a $1/2$-Lipschitz function, denoted by
$p(\alpha,\cdot)$. For each fixed $t$, the forward difference quotients of the
Lipschitz function $\tilde U(\cdot,t)$ converge to
$\tilde U_\alpha(\cdot,t)$ in $L^1_{\rm loc}$ as $h_n\downarrow0$.
Since the same sequence converges pointwise to the preceding uniform limit on
the chosen full-measure set, uniqueness of the $L^1_{\rm loc}$ limit shows
that $p(\alpha,t)=\tilde U_\alpha(\alpha,t)$ for a.e. $\alpha$, for each fixed
$t\in[0,T]$. Taking $T\in\mathbb N$ and intersecting the resulting countable
full-measure sets constructs a common time representative on $[0,\infty)$.

Set $e(x,t)=\tilde u_x^2(x,t)+\tilde\rho^2(x,t)$. By Fubini's theorem, the
one-dimensional area formula, and the Lebesgue differentiation theorem, after
removing another null set of labels one has, for a.e. $t$ such that
$q(\alpha,t)>0$,
\[
\lim_{n\to\infty}\frac1{h_n}
\int_{\tilde y(\alpha,t)}^{\tilde y(\alpha+h_n,t)}e(x,t)\di x
=e(\tilde y(\alpha,t),t)q(\alpha,t).
\]
Indeed, the quotient of the lengths converges to $q(\alpha,t)$, and
$\tilde y(\alpha,t)$ is a one-sided Lebesgue point of $e(\cdot,t)$ for a.e.
$(\alpha,t)$ on $\{q>0\}$; the latter assertion follows from the area formula
because the set of non-Lebesgue points of $e(\cdot,t)$ has zero Lebesgue
measure.

Let $[c,d]$ be a compact subinterval of a connected component of
$\{t:q(\alpha,t)>0\}$. Integrating the equation for $p_{h_n}$ gives
\[
p_{h_n}(\alpha,d)-p_{h_n}(\alpha,c)
=\frac12\int_c^d\frac1{h_n}
\int_{\tilde y(\alpha,t)}^{\tilde y(\alpha+h_n,t)}e(x,t)\di x\di t.
\]
The integrand on the right is bounded by
$h_n^{-1}\int_\alpha^{\alpha+h_n}f\leq1$. Hence the locally uniform
convergence of $p_{h_n}$ and dominated convergence imply
\begin{equation}\label{eq:omegat}
\partial_t p(\alpha,t)
=\frac12e(\tilde y(\alpha,t),t)q(\alpha,t)
=\frac{p^2(\alpha,t)+r^2(\alpha)}{2q(\alpha,t)}
\quad\text{on }\{q>0\}.
\end{equation}
The last identity follows from the Sobolev chain rule
$p=\tilde u_x(\tilde y,t)q$ and from \eqref{eq:tilderhoidentity}. In particular,
$p$ is locally absolutely continuous on each positivity component. The
right-hand side is locally integrable; more precisely, for every bounded label
interval $K$ and every $T>0$, the area formula gives
\[
\int_0^T\int_K\mathbf1_{\{q>0\}}
\frac{p^2+r^2}{q}\di\alpha\di t
=\int_0^T\int_K\mathbf1_{\{q>0\}}e(\tilde y,t)q\di\alpha\di t
\leq T\bar\nu(\mathbb R).
\]

On a positivity component, divide \eqref{eq:hp} and \eqref{eq:omegat} by
$q$ and use \eqref{eq:wrhoalpha}. This gives \eqref{eq:ODEs}. For a.e.
$\alpha\in B_0^L$, on the initial positivity component,
\begin{equation}\label{eq:omegatbelow}
\dot w_\alpha(t)\geq-\frac12w_\alpha^2(t),
\qquad
w_\alpha(t)\geq\frac{2w_\alpha(0)}{2+t w_\alpha(0)}
\end{equation}
for as long as $2+t w_\alpha(0)>0$. Since $q_t=w_\alpha q$, integration yields
\begin{equation}\label{eq:lowertildey}
q(\alpha,t)
=q(\alpha,0)\exp\left(\int_0^tw_\alpha(s)\di s\right)
\geq\frac{q(\alpha,0)}4[2+t w_\alpha(0)]^2
\end{equation}
under the same condition.

Finally, for a.e. $\alpha\in B_0^L$,
\[
q(\alpha,0)=\bar x'(\alpha),
\qquad
p(\alpha,0)=(\bar u\circ\bar x)'(\alpha)
=\bar u_x(\bar x(\alpha))\bar x'(\alpha),
\]
and $r(\alpha)=\bar\rho(\bar x(\alpha))\bar x'(\alpha)$. Dividing by
$q(\alpha,0)>0$ proves \eqref{eq:ODEinitial}.
\end{proof}

With the preceding lemma, the first breaking time is determined entirely by the initial data.

\begin{proposition}\label{eq:Jt_tildeAt}
Let $(\tilde u,\tilde\rho,\tilde\mu,\tilde\nu)$ and $\tilde y$ be as in Theorem~\ref{thm:uniqueness1}. Then, with all sets understood modulo $f(\alpha)\di\alpha$-null sets,
\begin{equation}\label{eq:tildejt}
\tilde J_t
:=\{\alpha\in\mathbb R:q(\alpha,s)=0
\text{ for some }0\leq s\leq t\}
=J_t,
\end{equation}
where $J_t$ is defined by \eqref{eq:blowupsets}.
\end{proposition}

\begin{proof}
Fix a.e. $\alpha\in B_0^L$ and consider the initial positivity component. Introduce
\[
z_\alpha(t):=w_\alpha(t)-i\rho_\alpha(t).
\]
By \eqref{eq:ODEs},
\[
\dot z_\alpha+\frac12z_\alpha^2=0,
\qquad
z_\alpha(0)=w_\alpha(0)-i\rho_\alpha(0),
\]
and hence
\[
z_\alpha(t)=\frac{z_\alpha(0)}{1+\frac t2z_\alpha(0)}.
\]
Writing
\[
D_\alpha(t)
:=\left[1+\frac t2w_\alpha(0)\right]^2
+\frac{t^2}{4}\rho_\alpha^2(0),
\]
we obtain
\begin{equation}\label{eq:explicit-solution}
w_\alpha(t)
=\frac{w_\alpha(0)+\frac t2[w_\alpha^2(0)+\rho_\alpha^2(0)]}{D_\alpha(t)},
\qquad
\rho_\alpha(t)=\frac{\rho_\alpha(0)}{D_\alpha(t)}.
\end{equation}
Integrating $q_t=w_\alpha q$ gives
\begin{equation}\label{eq:tildeyalpha}
q(\alpha,t)
=\bar x'(\alpha)D_\alpha(t)
=\bar x'(\alpha)\left(
\left[1+\frac t2\bar u_x(\bar x(\alpha))\right]^2
+\frac{t^2}{4}\bar\rho^2(\bar x(\alpha))\right).
\end{equation}
This identity holds on the maximal initial positivity component. If that component had a finite right endpoint at which $D_\alpha>0$, the continuity of $q(\alpha,\cdot)$ would extend positivity beyond the endpoint, a contradiction. Therefore the first zero exists precisely when
\[
\bar u_x(\bar x(\alpha))<0,
\qquad
\bar\rho(\bar x(\alpha))=0,
\]
and then it occurs at $T_\alpha=-2/\bar u_x(\bar x(\alpha))$. Conversely, \eqref{eq:tildeyalpha} is strictly positive before $T_\alpha$.

For $\alpha\in A_0^L$, $q(\alpha,0)=\bar x'(\alpha)=0$, so such a label belongs to $\tilde J_t$ for every $t>0$. Combining these facts with \eqref{eq:blowupsets} proves \eqref{eq:tildejt}.
\end{proof}

For later use, define
\begin{equation}\label{eq:St}
S_t:=\{\alpha\in\mathbb R:q(\alpha,t)=0\},
\qquad
S_t\subset\tilde J_t=J_t.
\end{equation}

\subsection{Inequality \eqref{eq:energyconserved11} and uniqueness}\label{subsec:accumultation}
For $t>0$, define the Eulerian outgoing region
\[
\Omega_t:=\{x\in\mathbb R:
\tilde\rho(x,t)=0,
\ \tilde u_x(x,t)\geq2/t\}.
\]
The exact Eulerian dissipation rule in Definition~\ref{def:weak} reads
\begin{equation}\label{eq:upbound}
\di(\tilde\nu-\tilde\mu)(t)
=\frac{\lambda}{1-\lambda}\tilde u_x^2(x,t)
\mathbf1_{\Omega_t}(x)\di x
+\lambda\di\tilde\nu_s(t).
\end{equation}

\begin{proposition}\label{pro:B}
For a.e. $t>0$,
\begin{equation}\label{eq:weakmeas}
\tilde\nu(t)-\tilde\mu(t)
=\lambda\tilde y(\cdot,t)\#
[f(\alpha)\mathbf1_{\tilde J_t}(\alpha)\di\alpha].
\end{equation}

\end{proposition}

\begin{proof}
Fix a time $t>0$ in the full-measure set on which the identities in Lemma~\ref{lmm:ODEs}, the Sobolev chain rules, and the decomposition of the monotone push-forward hold. We first prove the phase identity
\begin{equation}\label{eq:key1}
\mathbf1_{\Omega_t}(\tilde y(\alpha,t))
=\mathbf1_{\tilde J_t}(\alpha)
\qquad
\text{for }f(\alpha)\di\alpha\text{-a.e. }\alpha\in S_t^c.
\end{equation}

Assume first that $\alpha\notin\tilde J_t$. Then $\alpha\in B_0^L$, the initial positivity component contains $[0,t]$, and \eqref{eq:explicit-solution}--\eqref{eq:tildeyalpha} give
\begin{equation}\label{eq:Tstime}
q(\alpha,t)=\bar x'(\alpha)D_\alpha(t),
\qquad
w_\alpha(t)
=\frac{w_\alpha(0)+\frac t2[w_\alpha^2(0)+\rho_\alpha^2(0)]}{D_\alpha(t)},
\qquad
\rho_\alpha(t)=\frac{\rho_\alpha(0)}{D_\alpha(t)}.
\end{equation}
Thus the following alternatives exhaust the possibilities:
\begin{equation}\label{eq:negative}
\rho_\alpha(t)\neq0,
\quad\text{or}\quad
w_\alpha(t)<0,
\quad\text{or}\quad
0\leq w_\alpha(t)<\frac2t.
\end{equation}
Indeed,
\begin{equation}\label{eq1}
\rho_\alpha(0)\neq0
\quad\Longrightarrow\quad
\rho_\alpha(t)\neq0,
\end{equation}
while, if $\rho_\alpha(0)=0$ and $w_\alpha(0)<0$, then
\begin{equation}\label{eq2}
w_\alpha(t)=\frac{2}{t-T_\alpha}<0
\qquad(t<T_\alpha).
\end{equation}
If $\rho_\alpha(0)=w_\alpha(0)=0$, then
\begin{equation}\label{eq3}
w_\alpha(t)=0,
\end{equation}
and if $\rho_\alpha(0)=0$ and $w_\alpha(0)>0$, then
\begin{equation}\label{eq4}
w_\alpha(t)
=\frac{2w_\alpha(0)}{2+t w_\alpha(0)}
<\frac2t.
\end{equation}
Consequently,
\begin{equation}\label{eq:Tstime1}
\alpha\notin\tilde J_t,
\quad q(\alpha,t)>0
\quad\Longrightarrow\quad
\tilde y(\alpha,t)\notin\Omega_t.
\end{equation}

Now assume $\alpha\in\tilde J_t\cap S_t^c$, and let $(a,b)$ be the connected component of
$\{s>0:q(\alpha,s)>0\}$ containing $t$. Since the continuous representative of
$q$ has already reached zero,
\begin{equation}\label{eq:12equal1}
0\leq a<t,
\qquad
q(\alpha,a)=0,
\qquad
q(\alpha,s)>0\quad(a<s<b).
\end{equation}
If $\alpha\in A_0^L$, then $r(\alpha)=0$. If $\alpha\in B_0^L$,
Proposition~\ref{eq:Jt_tildeAt} shows that breaking is possible only when
$\bar\rho(\bar x(\alpha))=0$, and again $r(\alpha)=0$. Hence
\begin{equation}\label{eq:12equal2}
\rho_\alpha(s)=0,
\qquad a<s<b.
\end{equation}
By \eqref{eq:hp} and \eqref{eq:omegat}, $q\in W^{2,1}_{\rm loc}(a,b)$ and
\[
q_{tt}=\frac{q_t^2}{2q}
\quad\text{a.e. on }(a,b).
\]
Consequently $Q=\sqrt q$ belongs to $W^{2,1}_{\rm loc}(a,b)$ and
$Q_{tt}=0$ a.e. Since $Q$ extends continuously to $a$, $Q(a)=0$, and
$Q>0$ on $(a,b)$, it follows that $Q(s)=c_\alpha(s-a)$ with
$c_\alpha>0$. Thus, with $C_\alpha=c_\alpha^2$,
\begin{equation}\label{eq:12equal3}
q(\alpha,s)=C_\alpha(s-a)^2,
\qquad
w_\alpha(s)=\frac{q_t(\alpha,s)}{q(\alpha,s)}
=\frac2{s-a}.
\end{equation}
In particular, such a positivity component cannot have a finite right endpoint,
and
\[
\rho_\alpha(t)=0,
\qquad
w_\alpha(t)=\frac2{t-a}\geq\frac2t.
\]
Combining this conclusion with \eqref{eq:Tstime1} proves
\begin{equation}\label{eq:Tstime2}
\mathbf1_{\Omega_t}(\tilde y(\alpha,t))
=\mathbf1_{\tilde J_t}(\alpha)
\qquad
\text{for }f(\alpha)\di\alpha\text{-a.e. }\alpha\in S_t^c,
\end{equation}
which is \eqref{eq:key1}.

By Theorem~\ref{thm:uniqueness1} and Lemma~\ref{lmm:A2},
\begin{equation}\label{eq:key2}
\tilde\nu_s(t)
=\tilde y(\cdot,t)\#[f\mathbf1_{S_t}\di\alpha],
\qquad
\tilde\nu_{ac}(t)
=\tilde y(\cdot,t)\#[f\mathbf1_{S_t^c}\di\alpha].
\end{equation}
On $\Omega_t$, \eqref{eq:upbound} and
$\di\tilde\mu_{ac}=(\tilde u_x^2+\tilde\rho^2)\di x$ give
\[
\di\tilde\nu_{ac}
=\frac{1}{1-\lambda}\tilde u_x^2\di x,
\qquad
\di(\tilde\nu-\tilde\mu)_{ac}
=\lambda\mathbf1_{\Omega_t}\di\tilde\nu_{ac}.
\]
Outside $\Omega_t$, the absolutely continuous part of $\tilde\nu-\tilde\mu$ vanishes. Hence, using \eqref{eq:key1} and \eqref{eq:key2},
\begin{equation}\label{eq:weak1}
(\tilde\nu-\tilde\mu)_{ac}(t)
=\lambda\tilde y(\cdot,t)\#
[f\mathbf1_{\tilde J_t\cap S_t^c}\di\alpha].
\end{equation}
For the singular part, \eqref{eq:upbound}, \eqref{eq:key2}, and $S_t\subset\tilde J_t$ yield
\[
(\tilde\nu-\tilde\mu)_s(t)
=\lambda\tilde\nu_s(t)
=\lambda\tilde y(\cdot,t)\#[f\mathbf1_{S_t}\di\alpha].
\]
Combining the two parts gives, for every Borel set $E\subset\mathbb R$,
\begin{equation}\label{eq:H_1alphat}
(\tilde\nu-\tilde\mu)(t)(E)
=\lambda\int_{\tilde y^{-1}(E,t)}
f(\alpha)\mathbf1_{\tilde J_t}(\alpha)\di\alpha.
\end{equation}
This is \eqref{eq:weakmeas}. 

\end{proof}

We can now complete the proof without any assumption on the accumulation of singular times and without the former pointwise growth bound for the density of $\tilde\nu-\tilde\mu$.

\begin{theorem}\label{thm:uniqueness3}
For $0\leq\lambda<1$, let
$(\tilde u,\tilde\rho,\tilde\mu,\tilde\nu)$ be a $\lambda$-dissipative solution in the sense of Definition~\ref{def:weak} with initial datum
$(\bar u,\bar\rho,\bar\mu,\bar\nu)\in\mathcal D_\lambda$ satisfying \eqref{eq:initialdecay}. Then
\[
\tilde y(\alpha,t)=y(\alpha,t),
\qquad \alpha\in\mathbb R,
\quad t\geq0,
\]
where $y$ is defined by \eqref{eq:lambda_dissipation}. Consequently,
\[
(\tilde u,\tilde\rho,\tilde\mu,\tilde\nu)
=(u,\rho,\mu,\nu),
\]
where $(u,\rho,\mu,\nu)$ is given by \eqref{eq:solution2}. In particular, the $\lambda$-dissipative solution is unique.
\end{theorem}

\begin{proof}
The case $\lambda=0$ was proved in Theorem~\ref{thm:uniqueness1}. Assume $0<\lambda<1$. By \eqref{eq:weakmeas} and \eqref{eq:tildejt}, for a.e. $t>0$,
\[
\tilde\mu(t)
=\tilde y(\cdot,t)\#
[f(\alpha)(1-\lambda\mathbf1_{J_t}(\alpha))\di\alpha].
\]
The monotonicity of $\tilde y(\cdot,t)$ gives \eqref{eq:energyconserved11}. At a.e. such time, $\tilde\mu(t)$ is absolutely continuous. Lemma~\ref{lmm:dafermos}, \eqref{eq:dtu0}, and \eqref{eq:energyconserved11} therefore imply \eqref{eq:accelaration2} for a.e. $t>0$. Integrating twice and using
\[
\tilde y(\alpha,0)=\bar x(\alpha),
\qquad
\partial_t\tilde y(\alpha,0)=\bar u(\bar x(\alpha)),
\]
we obtain
\[
\begin{aligned}
\tilde y(\alpha,t)
=\bar x(\alpha)+\bar u(\bar x(\alpha))t+\frac12\int_0^t(t-s)
\left[\int_{-\infty}^{\alpha}-\kappa\int_{\mathbb R}\right]
f(\eta)[1-\lambda\mathbf1_{J_s}(\eta)]\di\eta\di s.
\end{aligned}
\]
This is exactly \eqref{eq:lambda_dissipation}; hence $\tilde y=y$.

Equation \eqref{eq:flowmap} gives
$\tilde u(y(\alpha,t),t)=y_t(\alpha,t)$, and the surjectivity of $y(\cdot,t)$ yields $\tilde u=u$. The push-forward identity in Theorem~\ref{thm:uniqueness1} gives $\tilde\nu=\nu$. The representation of $\tilde\mu$ agrees with \eqref{eq:solution2} for a.e. $t$. Both sides are right-continuous: for the explicit side this follows from the continuity of $y(\alpha,t)$, the monotonicity of $J_t$, and dominated convergence. Thus $\tilde\mu=\mu$ for every $t\geq0$.

Finally, the local push-forward identity following
\eqref{eq:tilderhoidentity} and $\tilde y=y$ give, for every
$\phi\in C_c^\infty(\mathbb R)$ and every $t\geq0$,
\[
\int_{\mathbb R}\phi(x)\tilde\rho(x,t)\di x
=\int_{\mathbb R}\phi(y(\alpha,t))
(\bar\rho\circ\bar x)(\alpha)\bar x'(\alpha)\di\alpha.
\]
This is precisely the distributional, equivalently locally finite signed-measure,
interpretation of the second identity in \eqref{eq:solution2}. Hence
$\tilde\rho=\rho$ in $L^2(\mathbb R)$ for every $t$, and uniqueness of the
whole quadruple follows.
\end{proof}

\section{Asymptotic behaviours of $\lambda$-dissipative solutions}\label{sec:asymp}
In this section, we show the large time behavior of the $\lambda$-dissipation solutions to the 2HS system, i.e., Theorem \ref{thm:mainasym}. To this end, let us first study the scaling properties of the 2HS system.

For any global smooth solution $(u,\rho)$ to the 2HS equation \eqref{eq:HS}, and any non-zero constant $\gamma$,  the re-scaled functions $u_{\gamma}(x,t):=\frac{1}{\gamma}u(\gamma^2 x, \gamma t)$ and  $\rho_{\gamma}(x,t):= \gamma\rho(\gamma^2 x, \gamma t)$  are also solutions to the Hunter-Saxton equation \eqref{eq:HS}. Self-similar solutions are a special class of solutions that are invariant under this scaling property.
In particular, by choosing $\gamma:=2/t$, we have 
\[
u(x,t) = \frac{t}{2}u\left(\frac{4x}{t^2}, 2\right)=:\frac{t}{2}\tilde{u}\left(\frac{4x}{t^2}\right),\quad
\rho(x,t) = \frac{2}{t}\rho\left(\frac{4x}{t^2}, 2\right)=:\frac{2}{t}\tilde{\rho}\left(\frac{4x}{t^2}\right);
\] 
let  $y:=\frac{4x}{t^2}$, then 
\[
\tilde{u}(y)=\frac{2}{t}u\left(\frac{t^2}{4}y, t\right),\quad \tilde{\rho}(y)=\frac{t}{2}\rho\left(\frac{t^2}{4}y, t\right)
\] 
for these self-similar solutions.  It follows that the rescaled limits  
\[
\lim\limits_{t\to+\infty} \frac{2}{t}{u\left(\frac{t^2}{4}x,t\right)},\quad \lim\limits_{t\to+\infty} \frac{t}{2}{\rho\left(\frac{t^2}{4}x,t\right)}
\] 
will play an important role in the study of the asymptotic behavior of the solutions if the limits exist. Finding these limits is exactly part (i) in Theorem \ref{thm:mainasym}.
It should be noted that this rescaled limit  $\lim\limits_{t\to+\infty} \frac{2}{t}{u\left(\frac{t^2}{4}x,t\right)}$  is also invariant under the above scaling of the equation, provided it exists. Based on this limit, we give the leading order term in the asymptotic expansions, which is part (ii) in Theorem \ref{thm:mainasym}.
Because the proof of Theorem \ref{thm:mainasym} is similar to \cite[Theorem 1.1]{Gaoliu2025} for the Hunter-Saxton equation, we will only sketch the main idea and omit the details.

We begin with the following simple facts that will be frequently used later:
\begin{lemma}(\cite[Lemma 3.1]{Gaoliu2025}) \label{lem:Hospital}
\begin{enumerate}
\item [(i)] 
We have
\begin{equation}\label{eq:limit4}
\begin{aligned}
&\lim_{t\to+\infty}\left[\frac{\lambda}{t}\int_{0}^{t} \int_{-\infty}^{\alpha}f(\eta)\mathbf{1}_{J_s}(\eta)\di \eta\di s 
- \frac{2\lambda}{t^2}\int_{0}^{t} (t-s) \int_{-\infty}^{\alpha} f(\eta)\mathbf{1}_{J_s}(\eta)\di \eta\di s\right]=0 
\end{aligned}
\end{equation}
uniformly in  $\alpha \in \mathbb{R}.$ 
\end{enumerate}
In the following, let $\alpha(t)$ be a general function depending on $t$.
\begin{enumerate}
\item [(ii)] 
If $\lim_{t\to+\infty}\alpha(t)=-\infty$, then
\begin{align}\label{eq:limit1}
\lim_{t\to+\infty}\frac{2\lambda}{t^2}\int_{0}^{t}(t-s)\int_{-\infty}^{\alpha(t)}f(\eta)\mathbf{1}_{J_s}(\eta)\di \eta\di s=\lim_{t\to+\infty}\frac{\lambda}{t}\int_{0}^{t} \int_{-\infty}^{\alpha(t)}f(\eta)\mathbf{1}_{J_s}(\eta)\di \eta\di s=0.
\end{align}
\item [(iii)] 
If $\lim_{t\to+\infty}\alpha(t)=+\infty$, then
\begin{equation}\label{eq:limit2}
\begin{aligned}
\lim_{t\to+\infty}\frac{2\lambda}{t^2}\int_{0}^{t}(t-s)\int_{-\infty}^{\alpha(t)}f(\eta)\mathbf{1}_{J_s}(\eta)\di \eta\di s&=\lim_{t\to+\infty}\frac{\lambda}{t}\int_{0}^{t} \int_{-\infty}^{\alpha(t)}f(\eta)\mathbf{1}_{J_s}(\eta)\di \eta\di s\\
&=\lambda\int_{\mathbb{R}}f(\eta) \mathbf{1}_{J}(\eta)\di \eta.
\end{aligned}
\end{equation}
\item [(iv)] 
We have
\begin{equation}\label{eq:limit3}
\begin{aligned}
\lim_{t\to+\infty}\frac{2\lambda}{t^2}\int_{0}^{t}(t-s)\int_{\mathbb{R}}f(\eta)\mathbf{1}_{J_s}(\eta)\di \eta\di s&=\lim_{t\to+\infty}\frac{\lambda}{t}\int_{0}^{t} \int_{{\mathbb{R}}}f(\eta)\mathbf{1}_{J_s}(\eta)\di \eta\di s\\
&=\lambda\int_{{\mathbb{R}}}f(\eta) \mathbf{1}_{J}(\eta)\di \eta.  
\end{aligned}
\end{equation}
\end{enumerate}
\end{lemma}

\begin{proof}
The proof is exactly the same as \cite[Lemma 3.1]{Gaoliu2025}. 
\end{proof}

We have the following theorem for the rescaled limit, which corresponds to Part (i) of Theorem \ref{thm:mainasym}:
\begin{theorem}\label{thm:mainthm1}
Let $(u(t),\rho(t),\mu(t),\nu(t))\in\mathcal{D}_{\lambda}$ be an $\lambda$-dissipative ($\lambda\in[0,1]$) solution to the 2HS system subject to the initial data $(\bar{u},\bar{\rho},\bar{\mu},\bar{\nu})\in\mathcal{D}_{\lambda}$ in the sense of Definition~\ref{def:weak}.
Let $E_\lambda$ be defined in \eqref{eq:finalenergy}. Then, the rescaled limit \eqref{eq:v} holds.
Moreover, we have 
\begin{equation}\label{eq:Limitlambda-xbarlambda=k1}
\begin{aligned}
&\quad \lim_{t\to+\infty}\left[\tilde{\alpha}(t)-\bar{x}(\tilde{\alpha}(t))-\frac{\lambda}{t}\int_{0}^{t} \int_{-\infty}^{\tilde{\alpha}(t)}f(\eta)\mathbf{1}_{J_s}(\eta)\di \eta\di s\right]\\
&=\lim_{t\to+\infty}\left[\tilde{\alpha}(t)-\bar{x}(\tilde{\alpha}(t))-\frac{2\lambda}{t^2}\int_{0}^{t} (t-s)\int_{-\infty}^{\tilde{\alpha}(t)}f(\eta)\mathbf{1}_{J_s}(\eta)\di \eta\di s\right]\\
&=\lim_{t\to+\infty}\frac{\bar{x}(\tilde{\alpha}(t))}{t^2}=\lim_{t\to+\infty}\frac{\bar{x}(\bar{\alpha}(t))}{t^2}=0,
\end{aligned}
\end{equation}
and
\begin{equation}\label{eq:Limitlambda-xbarlambda=k2}
\begin{aligned}
&\lim_{t\to+\infty}\left[\bar{\alpha}(t)-\bar{x}(\bar{\alpha}(t))-\frac{\lambda}{t}\int_{0}^{t} \int_{-\infty}^{\bar{\alpha}(t)}f(\eta)\mathbf{1}_{J_s}(\eta)\di \eta\di s\right]\\
=&\lim_{t\to+\infty}\left[\bar{\alpha}(t)-\bar{x}(\bar{\alpha}(t))-\frac{2\lambda}{t^2}\int_{0}^{t} (t-s)\int_{-\infty}^{\bar{\alpha}(t)}f(\eta)\mathbf{1}_{J_s}(\eta)\di \eta\di s\right]=E_\lambda,
\end{aligned}
\end{equation}
where $\tilde{\alpha}(t)$ satisfies $y(\tilde{\alpha}(t),t)=-\frac{t^2}{4}\kappa E_\lambda$ and $\bar{\alpha}(t)$ satisfies $y(\bar{\alpha}(t),t)=\frac{t^2}{4}(1-\kappa)E_\lambda$ .
\end{theorem}

\begin{proof}
The proof is the same as \cite[Theorem 3.1]{Gaoliu2025}.
\end{proof}

Next, we are going to show the asymptotic expansions of the $\lambda$-dissipative solutions to the 2HS system. According to the scaling property of the equation, for any fixed $x\in\mathbb{R}$, the Lagrangian label $\alpha$ satisfying $xt^2/4=y(\alpha,t)$ is important. We define
\begin{definition}\label{def:lambda01}
\begin{align}\label{eq:lambda01}
\alpha_l(t):=\sup\left\{\alpha:~y(\alpha,t)<-\frac{t^2}{4}\kappa E_{\lambda}\right\},\quad\mbox{and}\quad \alpha_r(t):=\inf\left\{\alpha:~y(\alpha,t)>\frac{t^2}{4}(1-\kappa) E_{\lambda}\right\}.
\end{align}
\end{definition}
Since $y$ is continuous, the above definition actually implies that
\begin{equation}\label{eq:y(xi_0,t)=0_and_y(xi_1,t)=t^2barmu/4}
\begin{aligned}
y(\alpha_l(t),t)=-\frac{t^2}{4}\kappa E_{\lambda},\quad\mbox{and}\quad y(\alpha_r(t),t)=\frac{t^2}{4}(1-\kappa) E_{\lambda}.
\end{aligned}
\end{equation}
Hence, $\alpha_\ell(t)$ is one of $\tilde{\alpha}(t)$ and also satisfies \eqref{eq:Limitlambda-xbarlambda=k1}; $\alpha_r(t)$ is one of $\bar{\alpha}(t)$ and  satisfies \eqref{eq:Limitlambda-xbarlambda=k2}.

Now we are ready to show the following theorem:
\begin{theorem}
Let $(u(t),\rho(t),\mu(t),\nu(t))\in\mathcal{D}_{\lambda}$ be an $\lambda$-dissipative ($\lambda\in[0,1]$) solution to the 2HS system subject to the initial data $(\bar{u},\bar{\rho},\bar{\mu},\bar{\nu})\in\mathcal{D}_{\lambda}$, in the sense of Definition~\ref{def:weak}.
Then, the asymptotic expansions, i.e., \eqref{eq:Linfty}, \eqref{eq:L2}, and \eqref{eq:L2rho} hold. Moreover,  the energy on the singular part of measure and $\rho$ will tend to zero as $t\to\infty$, i.e., \eqref{eq:singular part} holds.
\end{theorem}
\begin{proof}
The proof of \eqref{eq:Linfty} is similar to the proof of (1.10) in \cite[Theorem 3.2]{Gaoliu2025} and here we only prove  \eqref{eq:L2},  \eqref{eq:L2rho}, and \eqref{eq:singular part}.
We note by the definition of $v$ in \eqref{eq:v}
\begin{equation*}
\frac{t}{2} v\left(\frac{4x}{t^2}\right) =
\left\{
\begin{aligned}
&-\frac{t}{2}\kappa E_{\lambda}, \quad x<-\frac{t^2}{4}\kappa E_{\lambda},\\
&\frac{2x}{t} , \quad -\frac{t^2}{4}\kappa E_{\lambda}\leq x\leq \frac{t^2}{4}(1-\kappa) E_{\lambda}, \\
&\frac{t}{2}(1-\kappa)E_{\lambda}, \quad x> \frac{t^2}{4}(1-\kappa) E_{\lambda}.
\end{aligned}
\right.
\end{equation*}
We need to consider the following difference
\begin{gather}\label{eq:diff}
u(x,t)-\frac{t}{2} v\left(\frac{4x}{t^2}\right)=\left\{
\begin{split}
&u(x,t)+\frac{t}{2}\kappa E_{\lambda}, \quad x<-\frac{t^2}{4}\kappa E_{\lambda},\\
&u(x,t)-\frac{2x}{t} , \quad -\frac{t^2}{4}\kappa E_{\lambda}\leq x\leq \frac{t^2}{4}(1-\kappa) E_{\lambda}, \\
&u(x,t)-\frac{t}{2}(1-\kappa)E_{\lambda}, \quad x> \frac{t^2}{4}(1-\kappa) E_{\lambda}.
\end{split}
\right.
\end{gather}
Evaluating \eqref{eq:diff} at $x=y(\alpha,t)$, and using the explicit formulae \eqref{eq:lambda_dissipation} and \eqref{eq:solution2} for $y$ and $u$ respectively, we eventually obtain
\begin{equation}\label{eq:difference}
\begin{aligned}
&u(y(\alpha,t),t)-\frac{t}{2} v\left(\frac{4y(\alpha,t)}{t^2}\right)
\\
=&\left\{
\begin{aligned}
&\bar{u}(\bar{x}(\alpha))+\frac{t}{2}\left[\alpha-\bar{x}(\alpha) -\kappa\bar{\nu}(\mathbb{R})+\kappa E_\lambda\right]\\
&\qquad-\frac{\lambda}{2}\int_{0}^{t} \left[\int_{-\infty}^{\alpha}-\kappa\int_{\mathbb{R}}\right]f(\eta)\mathbf{1}_{J_s}(\eta)\di \eta\di s,\quad y(\alpha,t)<-\frac{t^2}{4}\kappa E_{\lambda},\\
&-\frac{2}{t}\bar{x}(\alpha)-\bar{u}(\bar{x}(\alpha)) - \frac{\lambda}{2}\int_{0}^{t} \left[\int_{-\infty}^{\alpha}-\kappa\int_{\mathbb{R}}\right]f(\eta)\mathbf{1}_{J_s}(\eta)\di \eta\di s\\
&\qquad +  \frac{\lambda}{t}\int_{0}^{t} (t-s) \left[\int_{-\infty}^{\alpha} -\kappa\int_{\mathbb{R}}\right]f(\eta)\mathbf{1}_{J_s}(\eta)\di \eta\di s
,\quad -\frac{t^2}{4}\kappa E_{\lambda}\leq  y(\alpha,t) \leq \frac{t^2}{4}(1-\kappa) E_{\lambda}, \\
&\bar{u}(\bar{x}(\alpha))+\frac{t}{2}\left[\alpha-\bar{x}(\alpha) -\kappa\bar{\nu}(\mathbb{R})+\kappa E_\lambda-E_\lambda\right]\\
&\qquad-\frac{\lambda}{2}\int_{0}^{t} \left[\int_{-\infty}^{\alpha}-\kappa\int_{\mathbb{R}}\right]f(\eta)\mathbf{1}_{J_s}(\eta)\di \eta\di s,\quad y(\alpha,t)>\frac{t^2}{4}(1-\kappa) E_{\lambda}.
\end{aligned}
\right.
\end{aligned}
\end{equation}

\noindent\textbf{Proof of \eqref{eq:L2} and \eqref{eq:L2rho}:} According to \eqref{eq:diff}, we have
\begin{multline}\label{eq:I123}
\|\rho(\cdot,t)\|_{L^2}^2+\left\|u_x(x,t)-\partial_x\left[\frac{t}{2} v\left(\frac{4x}{t^2}\right)\right]\right\|_{L^2}^2=\int_{(-\infty,-\frac{t^2}{4}\kappa E_{\lambda})}(\rho^2+u_x^2)(x,t)\di x\\
+\int_{(-\frac{t^2}{4}\kappa E_{\lambda},\frac{t^2}{4}(1-\kappa) E_{\lambda})}\rho^2(x,t)+\left[u_x(x,t)-\frac{2}{t}\right]^2\di x
+\int_{(\frac{t^2}{4}(1-\kappa) E_{\lambda},+\infty)}(\rho^2+u_x^2)(x,t)\di x\\
=:I_1(t)+I_2(t)+I_3(t).
\end{multline}
We will apply the change of variables to estimate $I_i$ for $i=1$, $2$, $3$.
According to \eqref{eq:ABt}, we have
\begin{multline*}
I_1(t)=\int_{(-\infty,-\frac{t^2}{4}\kappa E_{\lambda})}(\rho^2+u_x^2)(x,t)\di x\\
=\int_{(-\infty,-\frac{t^2}{4}\kappa E_{\lambda})\cap y(B^L_t,t)}(\rho^2+u_x^2)(x,t)\di x=\int_{(-\infty,\alpha_l(t))\cap B^L_t}(\rho^2+u_x^2)(y(\alpha,t),t)y_\alpha(\alpha,t)\di \alpha.
\end{multline*}
By \eqref{eq:acpart} (or \eqref{eq:acpart2}) and \eqref{eq:Limitlambda-xbarlambda=k1}, we obtain
\begin{equation}\label{eq:termI1}
\begin{aligned}
I_1(t)=& \int_{(-\infty,\alpha_l(t))}f(\alpha)[1-\lambda \mathbf{1}_{J_t}(\alpha)]\di \alpha=\alpha_l(t)-\bar{x}(\alpha_l(t))-\lambda\int_{-\infty}^{\alpha_l(t)}f(\alpha)\mathbf{1}_{J_t}(\alpha)\di \alpha\\
\leq& \alpha_l(t)-\bar{x}(\alpha_l(t))-\frac{\lambda}{t}\int_0^t\int_{-\infty}^{\alpha_l(t)}f(\alpha)\mathbf{1}_{J_s}(\alpha)\di \alpha\di s\to 0,
\end{aligned}
\end{equation}
as $t\to+\infty$,  where we used $\mathbf{1}_{J_s}(\alpha)\leq \mathbf{1}_{J_t}(\alpha)$ for $s\leq t$ in the last inequality.

For $I_3(t)$, by \eqref{eq:limit3} and \eqref{eq:Limitlambda-xbarlambda=k2}, we obtain
\begin{equation}\label{eq:termI3}
\begin{aligned}
&I_3(t)=\int_{(\frac{t^2}{4}(1-\kappa) E_{\lambda},+\infty)}(\rho^2+u_x^2)(x,t)\di x\\
=&\int_{(\alpha_r(t),+\infty)\cap B^L_t}(\rho^2+u_x^2)(y(\alpha,t),t)y_\alpha(\alpha,t)\di \alpha = \int_{\alpha_r(t)}^{+\infty}f(\alpha)[1-\lambda \mathbf{1}_{J_t}(\alpha)]\di \alpha\\
\leq& \bar{\nu}(\mathbb{R})-[\alpha_r(t)-\bar{x}(\alpha_r(t))]-\frac{\lambda}{t}\int_0^t\int_{\alpha_r(t)}^{+\infty}f(\alpha) \mathbf{1}_{J_s}(\alpha)\di \alpha\di s\\
=&\bar{\nu}(\mathbb{R})-\frac{\lambda}{t}\int_0^t\int_{\mathbb{R}}f(\alpha) \mathbf{1}_{J_s}(\alpha)\di \alpha\di s-\left[\alpha_r(t)-\bar{x}(\alpha_r(t))-\frac{\lambda}{t}\int_0^t\int^{\alpha_r(t)}_{-\infty}f(\alpha) \mathbf{1}_{J_s}(\alpha)\di \alpha \di s\right]\to 0,
\end{aligned}
\end{equation}
as $t\to+\infty$, where we also recall the definition $E_{\lambda}$ in \eqref{eq:finalenergy}.

For $I_2(t)$, we have 
\begin{multline*}
I_2(t)=\int_{(-\frac{t^2}{4}\kappa E_{\lambda},\frac{t^2}{4}(1-\kappa) E_{\lambda})}\rho^2(x,t)+\left[u_x(x,t)-\frac{2}{t}\right]^2\di x\\
= \int_{(\alpha_l(t),\alpha_r(t))\cap B_t}(\rho^2+u_x^2)(y(\alpha,t),t)y_\alpha(\alpha,t)\di \alpha + 
\int_{(\alpha_l(t),\alpha_r(t))} \left[ \frac{4}{t^2}y_\alpha(\alpha,t)-\frac{4}{t}u_x(y(\alpha,t),t)y_\alpha(\alpha,t) \right]\di \alpha,
\end{multline*}
which implies
\begin{equation}\label{eq: term I2}
\begin{aligned}
I_2(t)&= \int_{(\alpha_l(t),\alpha_r(t))}f(\alpha)[1-\lambda \mathbf{1}_{J_t}(\alpha)]\di \alpha+\frac{4}{t^2}[y(\alpha_r(t),t)-y(\alpha_l(t),t)]\\
&\qquad\qquad\qquad-\frac{4}{t}[u(y(\alpha_r(t),t),t)-u(y(\alpha_l(t),t),t)]\\
&=[\alpha_r(t)-\bar{x}(\alpha_r(t))]  -  [\alpha_l(t)-\bar{x}(\alpha_l(t))] -\lambda\int_{\alpha_l(t)}^{\alpha_r(t)}f(\alpha) \mathbf{1}_{J_t}(\alpha) \di \alpha\\
&\qquad\qquad
+\frac{4}{t^2}\left(\frac{t^2}{4} E_{\lambda}\right )
-\frac{4}{t}u\left(\frac{t^2}{4}(1-\kappa)E_{\lambda},t\right)  +  \frac{4}{t}u\left(-\frac{t^2}{4}\kappa E_{\lambda},t\right)\\
&\leq \left[\alpha_r(t)-\bar{x}(\alpha_r(t))-\frac{\lambda}{t}\int_0^t\int^{\alpha_r(t)}_{-\infty}f(\alpha) \mathbf{1}_{J_s}(\alpha)\di \alpha\di s\right]  \\
&\qquad\qquad -\left[\alpha_l(t)-\bar{x}(\alpha_l(t))-\frac{\lambda}{t}\int_0^t\int^{\alpha_\ell(t)}_{-\infty}f(\alpha) \mathbf{1}_{J_s}(\alpha)\di \alpha\di s\right]\\
&\qquad\qquad+ E_{\lambda}
-\frac{4}{t}u\left(\frac{t^2}{4}(1-\kappa)E_{\lambda},t\right)  +  \frac{4}{t}u\left(-\frac{t^2}{4}\kappa E_{\lambda},t\right).
\end{aligned}
\end{equation}
Combining the rescaled limit in \eqref{eq:v} we have proven, \eqref{eq:Limitlambda-xbarlambda=k1} and \eqref{eq:Limitlambda-xbarlambda=k2}, we finally have
\begin{equation*}
\lim_{t\to+\infty}   I_2(t) \leq E_\lambda-0 + E_{\lambda} - 2(1-\kappa)E_{\lambda} -2\kappa E_{\lambda}=0.
\end{equation*}
This completes the proof of \eqref{eq:L2} and \eqref{eq:L2rho}. We note \eqref{eq:L2} especially implies that 
\begin{equation}\label{eq:energylim}
\lim_{t\to+\infty}\|u_x(\cdot,t)\|_{L^2(\mathbb R)}^2
=\left\|\partial_x\left[\frac{t}{2}v\left(\frac{4x}{t^2}\right)\right]\right\|_{L^2(\mathbb R)}^2
=E_\lambda.
\end{equation}

\noindent\textbf{Proof of \eqref{eq:singular part}:} From \eqref{eq:solution2}$_2$, i.e., $\mu(t)=y(\cdot,t)\# [f(1-\lambda \mathbf{1}_{J_t})\di \alpha]$, we have
\begin{align*}
\lim_{t\to+\infty}\mu(t)(\mathbb{R})=\lim_{t\to+\infty}\int_{\mathbb{R}}f(1-\lambda \mathbf{1}_{J_t})(\alpha)\di\alpha=E_\lambda.
\end{align*}
The definition of the energy measure $\mu$ implies that  
\[
\|u_x(\cdot, t)\|_{L^2}^2+\|\rho(\cdot,t)\|_{L^2}^2  + {\mu}_{s}(t)(\mathbb{R})=\mu(t)(\mathbb{R})
\]
for all $t>0$, so passing to the limit as $t\to+\infty$, we finally obtain \eqref{eq:singular part}. 
This completes the whole proof of Theorem~\ref{thm:mainthm1}.

\end{proof}

\appendix

\section{Some useful facts from real analysis}\label{app:real}
In this appendix, we state three useful lemmas from real analysis. All of them are fundamental and somewhat classical. The proofs can be found in the appendix of \cite{gao2022regularity}.

\begin{lemma}\label{lmm:A1}
The following two statements holds:
\begin{enumerate}
\item[(i)] The real line $\mathbb{R}$ cannot be written as the union of uncountably many disjoint subsets with positive measures. 
\item[(ii)] Let $X:\mathbb{R}\to\mathbb{R}$ be an absolutely continuous function satisfying $X_\alpha(\alpha)>0$ for a.e. $\alpha\in\mathbb{R}$. Then, there exists a unique absolutely continuous inverse of $X$.
\end{enumerate}
\end{lemma}

We have the following lemma for push-forward measures:
\begin{lemma}\label{lmm:A2} 
Let $X:\mathbb{R}\to\mathbb{R}$ be a continuous increasing surjective function. Define two pseudo-inverse functions of $X$ by
\[
Z_1(x)=\inf\{\alpha:~~X(\alpha)=x\},\quad Z_2(x)=\sup\{\alpha:~~X(\alpha)=x\}.
\]
Define
\[
A^{pp}=\{\alpha:~~X_\alpha(\alpha)=0,~~Z_1(X(\alpha))<Z_2(X(\alpha))\},\quad A^{sc}=\{\alpha:~~X_\alpha(\alpha)=0,~~Z_1(X(\alpha))=Z_2(X(\alpha))\},
\]
and
\[
B=\{\alpha:~~X_\alpha(\alpha)>0\}.
\]
Here, $A^{pp}$ is defined in point-wise sense, and $A^{sc}$ and $B$ are defined in a.e. sense.

Let $0\leq g  \in L^1(\mathbb{R})$ . Consider the measure 
\[
\mu =X\# (g\di \alpha).
\]
Let $g_1=g\cdot 1_B$, $g_2=g\cdot 1_{A^{pp}}$ and $g_3=g\cdot 1_{A^{sc}}$, here $1_{A^{pp}}$, $1_{A^{sc}}$ and $1_B$ are characteristic functions on $A^{pp}$, $A^{sc}$ and $B$ respectively. Then the following statements hold:
\begin{enumerate}
\item[(i)]  $\mathcal{L}(X(A^{pp}\cup A^{sc}) )=0$, where $\mathcal{L}$ is the Lebesgue measure.
\item[(ii)] The absolutely continuous part of $\mu$ is given by $X\#(g_1\di \alpha)$.
\item[(iii)] The set  $X(A^{pp})$ is a countable set, and the pure point part of $\mu$ is given by $X\#(g_2\di \alpha)$.
\item[(iv)]  The singular continuous part of $\mu$ is given by $X\#(g_3\di \alpha)$. 
\end{enumerate}  
\end{lemma}	

The last lemma shows that a locally absolutely continuous (i.e., absolutely continuous on any closed and bounded interval of $\mathbb{R}$) function has to be globally absolutely continuous if its derivative is in $L^p(\mathbb{R})$ for some $p\in[1,\infty]$.
\begin{lemma}\label{lem:A3}
Let $1\le p \le \infty$. Then any locally absolutely continuous function  $u$ on $\mathbb{R}$  with its derivative $u_x \in L^p(\mathbb{R})$ is globally absolutely continuous. 
\end{lemma}

\section{Fully dissipative case $\lambda=1$}\label{app:dissipative}
In this appendix, we will sketch the main idea for the characteristics method for the $\lambda$-dissipative solution with $\lambda=1$. For simplicity, we only consider the case $\kappa=0$.

Recall that $\mathcal{D}_1$ is defined by (2) in Definition \ref{def:D}.

Under the boundedness requirement on $u_x$, the incoming cusps are allowed, but not the outgoing cusps, as in Examples \ref{example}, \ref{example2}
for $\lambda=1$. 

Next, we show the explicit formula of characteristics for the dissipative solutions, which shares some similarities with the one used in \cite{bressan2005global}. Without singular energy measures initially, we will directly use the Lagrangian label $\xi\in\mathbb{R}$ instead of $\bar{x}(\alpha)$ for the initial data of characteristics.
Similarly to \eqref{eq:lambda_dissipation}, 
the  explicit formula of characteristics $X(\xi,t)$ is then given by:
\begin{equation}\label{eq:lambda_dissipation1}
\begin{aligned}
X(\xi,t)=&\xi +\bar{u}(\xi)t+\frac{t^2}{4} \int_{-\infty}^\xi(\bar{u}_x^2+\bar{\rho}^2)(\eta)\di \eta -    \frac{1}{2}\int_{0}^{t}(t-s)\int_{-\infty}^{\xi}  (\bar{u}_x^2+\bar{\rho}^2)(\eta)\mathbf{1}_{J_s}(\eta)\di \eta\di s\\
=&\xi +\bar{u}(\xi)t+\frac{1}{2}\int_{0}^{t}(t-s)\int_{(-\infty,\xi)\setminus J_s}  (\bar{u}_x^2+\bar{\rho}^2)(\eta) \di \eta\di s,
\end{aligned}
\end{equation}
where we used the repeated integral formula for the last equality and 
\begin{equation}\label{eq:blowupsets11}
J_t:=\left\{\xi\in \mathbb{R}:\bar{u}_x(\xi)\leq-\frac{2}{t},~~ \bar{\rho}(\xi)=0\right\}.
\end{equation}
The above explicit characteristics \eqref{eq:lambda_dissipation1} resembles the one used in \cite{bressan2005global} for dissipative solutions of HS equation. The set $\mathbb{R}\setminus J_t$ stands for the Lagrangian labels  for which no dissipation occurred before $t$.
Utilizing the characteristics, the dissipative solution $(u(t),\rho(t))$ to the 2HS equation can be recovered by:
\begin{equation}\label{eq:solution22}
\left\{
\begin{aligned}
u(x,t)&=\partial_tX(\xi,t)=\bar{u}(\xi)+\frac{1}{2}\int_{0}^{t} \int_{(-\infty,\xi)\setminus J_s}  (\bar{u}_x^2+\bar{\rho}^2)(\eta) \di \eta \di s, \quad\textrm{where } x=X(\xi,t),\\
\rho(x,t)\di x&=X(\cdot,t)\#(\bar{\rho}\di\xi).
\end{aligned}
\right.
\end{equation}
Then the regularity structures and existence of dissipative solutions to \eqref{eq:2HS} in the sense of Definition \ref{def:dissipative} can be obtained by the same argument as in Theorem \ref{thm:measure} and Theorem \ref{thm:mainexistence}.

For the uniqueness of the dissipative solutions, the idea is the same as \cite{dafermos2011generalized} for the HS equation. The absence of outgoing cusps is the key for uniqueness. We have
\begin{theorem}[Uniqueness of dissipative solutions]\label{thm:uniqueness2}
Let $(\tilde{u},\tilde{\rho})$ be a dissipative solution to the 2HS system \eqref{eq:2HS} in the sense of Definition \ref{def:dissipative} with initial datum $(\bar{u},\bar{\rho})\in\mathcal{D}_1$.  Then the solution is the same as the one given by \eqref{eq:solution22} and hence uniqueness follows.
\end{theorem}
\begin{proof}
The main idea of the proof follows that of \cite{dafermos2011generalized} for the HS equation. Proceeding as in the proof of \cite[Theorem 4.1]{dafermos2011generalized}, the primary difficulty lies in deriving and solving ODE systems analogous to \eqref{eq:ODEs}. This is readily overcome by the same method as in Lemma \ref{lmm:ODEs} and Proposition \ref{eq:Jt_tildeAt}. The remaining details are left to the interested reader.

\end{proof}
For long-time behavior, the proof is similar to the $\lambda$-dissipative case for $0\leq\lambda<1$. 

\

\noindent\textbf{Acknowledgements.} 
This work is supported by Shenzhen Start-Up Research Fundation (Grant No. HA1140900425), Shenzhen Basic Research Foundation (Grant No. JCYJ20240813105503005),  Guangdong Basic and Applied Basic Research Foundation (Grant No.  ZJQNRC20241219170240010, 2026A1515011751). 

\

\noindent\textbf{Declarations.}  All authors declare that they have no conflicts of interest.

\

\noindent\textbf{Data availability. } 
Data sharing is not applicable to this article as no datasets
were generated or analyzed during the current study.

\bibliographystyle{plain}
\bibliography{bibofChara}

@article {Gaoliu2025,
    AUTHOR = {Gao, Yu and Liu, Hao},
     TITLE = {On the large-time asymptotic behaviors of
              {$\lambda$}-dissipative solutions to the {H}unter--{S}axton
              equation},
   JOURNAL = {Calc. Var. Partial Differential Equations},
  FJOURNAL = {Calculus of Variations and Partial Differential Equations},
    VOLUME = {65},
      YEAR = {2026},
    NUMBER = {1},
     PAGES = {Paper No. 18},
      ISSN = {0944-2669,1432-0835},
   MRCLASS = {35},
  MRNUMBER = {4999789},
       DOI = {10.1007/s00526-025-03092-5},
       URL = {https://doi.org/10.1007/s00526-025-03092-5},
}

@article{aratyn2006negative,
	title={On negative flows of the {AKNS} hierarchy and a class of deformations of a bihamiltonian structure of hydrodynamic type},
	author={Aratyn, H. and Gomes, J. F. and Zimerman, A. H.},
	journal={J. Phys. A},
	volume={39},
	number={5},
	pages={1099},
	year={2006},
	publisher={IOP Publishing}
}

@article{bressan2005global,
	title={Global Solutions of the {H}unter--{S}axton Equation},
	author={Bressan, A. and Constantin, A.},
	journal={SIAM J. Math. Anal.},
	volume={37},
	number={3},
	pages={996--1026},
	year={2005},
	publisher={SIAM}
}

@article{bressan2005optimal,
	title={An Optimal Transportation Metric for Solutions of the {C}amassa-{H}olm Equation},
	author={Bressan, A. and Fonte, M.},
	journal={Methods Appl. Anal.},
	volume={12},
	number={2},
	pages={191--220},
	year={2005},
	publisher={International Press of Boston}
}

@article{bressan2007asymptotic,
	title={Asymptotic variational wave equations},
	author={Bressan, A. and Zhang, P. and Zheng, Y.},
	journal={Arch. Ration. Mech. Anal.},
	volume={183},
	number={1},
	pages={163--185},
	year={2007},
	publisher={Springer}
}

@article{bressan2007global,
	title={Global conservative solutions of the {C}amassa--{H}olm equation},
	author={Bressan, A. and Constantin, A.},
	journal={Arch. Ration. Mech. Anal.},
	volume={183},
	number={2},
	pages={215--239},
	year={2007},
	publisher={Springer}
}

@article{Bressan2010,
	title={Lipschitz metric for the {H}unter--{S}axton equation},
	author={Bressan, A. and Holden, H. and Raynaud, X.},
	journal={J. Math. Pures Appl.},
	volume={94},
	number={1},
	pages={68--92},
	year={2010},
	publisher={Elsevier}
}

@article{bressan2015unique,
	title={Unique conservative solutions to a variational wave equation},
	author={Bressan, A. and Chen, G. and Zhang, Q.},
	journal={Arch. Ration. Mech. Anal.},
	volume={217},
	number={3},
	pages={1069--1101},
	year={2015},
	publisher={Springer}
}

@article{bressan2017lipschitz,
  title={{L}ipschitz metrics for a class of nonlinear wave equations},
  author={Bressan, A. and Chen, G.},
  journal={Arch. Ration. Mech. Anal.},
  volume={226},
  number={3},
  pages={1303--1343},
  year={2017},
  publisher={Springer}
}

@article{bressan2016uniqueness,
	title={Uniqueness of conservative solutions for nonlinear wave equations via characteristics},
	author={Bressan, A.},
	journal={Bull. Braz. Math. Soc.},
	volume={47},
	number={1},
	pages={157--169},
	year={2016},
	publisher={Springer}
}

@article{chen2018existence,
  title={Existence and uniqueness of the global conservative weak solutions for the integrable {N}ovikov equation},
  author={Chen, G. and Chen, R. M. and Liu, Y.},
  journal={Indiana Univ. Math. J.},
  pages={2393--2433},
  year={2018},
  publisher={JSTOR}
}

@article{chen2006two,
	title={A two-component generalization of the {C}amassa-{H}olm equation and its solutions},
	author={Chen, M. and Liu, S.-Q. and Zhang, Y.},
	journal={Letters in Mathematical Physics},
	volume={75},
	number={1},
	pages={1--15},
	year={2006},
	publisher={Springer}
}

@article{constantin2008integrable,
	title={On an integrable two-component {C}amassa--{H}olm shallow water system},
	author={Constantin, A. and Ivanov, R. I.},
	journal={Phys. Lett. A},
	volume={372},
	number={48},
	pages={7129--7132},
	year={2008},
	publisher={Elsevier}
}

@article{dafermos2011generalized,
	title={Generalized characteristics and the {H}unter--{S}axton equation},
	author={Dafermos, C. M.},
	journal={J. Hyperbolic Differ. Equ.},
	volume={8},
	number={01},
	pages={159--168},
	year={2011},
	publisher={World Scientific}
}

@article{gao2022regularity,
  title={Regularity Structure of Conservative Solutions to the {H}unter--{S}axton Equation},
  author={Gao, Y. and Liu, H. and Wong, T. K.},
  journal={SIAM J. Math. Anal.},
  volume={54},
  number={1},
  pages={423--452},
  year={2022},
  publisher={SIAM}
}

@article{gao2023asymptotic,
	title={Asymptotic Behavior of Conservative Solutions to the {H}unter--{S}axton Equation},
	author={Gao, Y. and Liu, H. and Wong, T. K.},
	journal={SIAM J. Math. Anal.},
	volume={55},
	number={5},
	pages={5483--5525},
	year={2023},
	publisher={SIAM}
}

@article{grunert2022lipschitz,
	title={Lipschitz stability for the {H}unter--{S}axton equation},
	author={Grunert, K. and Tandy, M.},
	journal={J. Hyperbolic Differ. Equ.},
	volume={19},
	number={02},
	pages={275--310},
	year={2022},
	publisher={World Scientific}
}

@article{grunert2018existence,
	title={Existence and {L}ipschitz stability for $\alpha$-dissipative solutions of the two-component {H}unter--{S}axton system},
	author={Grunert, K. and Nordli, A.},
	journal={J. Hyperbolic Differ. Equ.},
	volume={15},
	number={03},
	pages={559--597},
	year={2018},
	publisher={World Scientific}
}

@article{grunert2012global,
	title={Global solutions for the two-component {C}amassa--{H}olm system},
	author={Grunert, K. and Holden, H. and Raynaud, X.},
	journal={Comm. Partial Differential Equations},
	volume={37},
	number={12},
	pages={2245--2271},
	year={2012},
	publisher={Taylor \& Francis}
}

@inproceedings{grunert2015continuous,
	title={A continuous interpolation between conservative and dissipative solutions for the two-component {C}amassa--{H}olm system},
	author={Grunert, K. and Holden, H. and Raynaud, X.},
	booktitle={Forum Math. Sigma},
	volume={3},
	year={2015},
	organization={Cambridge University Press}
}

@article{guan2011global,
	title={Global weak solutions and smooth solutions for a two-component {H}unter-{S}axton system},
	author={Guan, C. and Yin, Z.},
	journal={J. Math. Phys.},
	volume={52},
	number={10},
	pages={103707},
	year={2011},
	publisher={American Institute of Physics}
}

@article{holden2007global,
	title={Global conservative solutions of the {C}amassa--{H}olm equation—a {L}agrangian point of view},
	author={Holden, H. and Raynaud, X.},
	journal={Comm. Partial Differential Equations},
	volume={32},
	number={10},
	pages={1511--1549},
	year={2007},
	publisher={Taylor \& Francis}
}

@article{holden2008global,
	title={Global dissipative multipeakon solutions of the {C}amassa--{H}olm equation},
	author={Holden, H. and Raynaud, X.},
	journal={Comm. Partial Differential Equations},
	volume={33},
	number={11},
	pages={2040--2063},
	year={2008},
	publisher={Taylor \& Francis}
}

@article{holden2009dissipative,
	title={Dissipative solutions for the {C}amassa-{H}olm equation},
	author={Holden, H. and Raynaud, X.},
	journal={Discrete Contin. Dyn. Syst},
	volume={24},
	number={4},
	pages={1047--1112},
	year={2009},
	publisher={Citeseer}
}

@article{hunter1991dynamics,
	title={Dynamics of director fields},
	author={Hunter, J. K. and Saxton, R.},
	journal={SIAM J. Appl. Math.},
	volume={51},
	number={6},
	pages={1498--1521},
	year={1991},
	publisher={SIAM}
}

@article{hunter1994completely,
title={On a completely integrable nonlinear hyperbolic variational equation},
author={Hunter, J. K. and Zheng, Y.},
journal={Physica D},
volume={79},
number={2-4},
pages={361--386},
year={1994},
publisher={Elsevier}
}

@article{lenells2009n,
	title={On the ${N}= 2$ supersymmetric {C}amassa--{H}olm and {H}unter--{S}axton equations},
	author={Lenells, J. and Lechtenfeld, O.},
	journal={J. Math. Phys.},
	volume={50},
	number={1},
	pages={012704},
	year={2009},
	publisher={American Institute of Physics}
}

@article{lenells2013hunter,
	title={The {H}unter-{S}axton system and the geodesics on a pseudosphere},
	author={Lenells, J. and Wunsch, M.},
	journal={Comm. Partial Differential Equations},
	volume={38},
	number={5},
	pages={860--881},
	year={2013},
	publisher={Taylor \& Francis}
}

@article{lenells2013spheres,
	title={Spheres, {K}{\"a}hler geometry and the {H}unter--{S}axton system},
	author={Lenells, J.},
	journal={Proc. R. Soc. Lond. Ser. A},
	volume={469},
	number={2154},
	pages={20120726},
	year={2013},
	publisher={The Royal Society Publishing}
}

@article{nordli2016lipschitz,
	title={A {L}ipschitz metric for conservative solutions of the two-component {H}unter--{S}axton system},
	author={Nordli, A.},
	journal={Methods Appl. Anal.},
	volume={23},
	number={3},
	pages={215--232},
	year={2016},
	publisher={International Press of Boston}
}

@article{Olver,
	title={Tri-Hamiltonian duality between solitons and solitary-wave solutions having compact support},
	author={Olver, P. J. and Rosenau, P.},
	journal={Phys. Rev. E},
	volume={53},
	number={2},
	pages={1900},
	year={1996},
	publisher={APS}
}

@article{pavlov2005gurevich,
	title={The {G}urevich--{Z}ybin system},
	author={Pavlov, M. V.},
	journal={J. Phys. A},
	volume={38},
	number={17},
	pages={3823},
	year={2005},
	publisher={IOP Publishing}
}

@article{wunsch2009hunter,
	title={On the {H}unter--{S}axton system},
	author={Wunsch, M.},
	journal={Discrete Contin. Dyn. Syst.-B},
	volume={12},
	number={3},
	pages={647},
	year={2009},
	publisher={American Institute of Mathematical Sciences}
}

@article{wunsch2011weak,
	title={Weak geodesic flow on a semidirect product and global solutions to the periodic {H}unter--{S}axton system},
	author={Wunsch, M.},
	journal={Nonlinear Anal.},
	volume={74},
	number={15},
	pages={4951--4960},
	year={2011},
	publisher={Elsevier}
}

@article{wunsch2010generalized,
	title={The generalized {H}unter--{S}axton system},
	author={Wunsch, M.},
	journal={SIAM J. Math. Anal.},
	volume={42},
	number={3},
	pages={1286--1304},
	year={2010},
	publisher={SIAM}
}

\end{document}